\documentclass[11pt,reqno]{amsart}
\usepackage{amsfonts}
\usepackage{amsmath}
\usepackage{amssymb}
\usepackage{multicol}
\usepackage{stmaryrd}
\usepackage{cite}
\usepackage{epsfig}
\usepackage{color}
\usepackage{graphics}
\usepackage{graphicx}
\usepackage{epstopdf}
\usepackage{multicol,graphics}
\newcommand\bes{\begin{eqnarray}}
\newcommand\ees{\end{eqnarray}}

\newtheorem{theorem}{Theorem}[section]
\newtheorem{lemma}[theorem]{Lemma}
\newtheorem{corollary}[theorem]{Corollary}

\newtheorem{definition}[theorem]{Definition}
\newtheorem{remark}[theorem]{Remark}

\numberwithin{equation}{section}
\allowdisplaybreaks

\allowdisplaybreaks[4]
\usepackage{bbm}
\usepackage{cite}
\usepackage{amsmath,amssymb}
\usepackage{multicol}
\usepackage{graphicx}
\usepackage{epsfig}
\usepackage{color}
\usepackage{amsthm}
\usepackage{indentfirst}
\usepackage{graphics,amsmath,amssymb,amsthm,mathrsfs,bm}
\usepackage{enumitem}
\usepackage[titletoc]{appendix}

\begin{document}

\title[Reaction-diffusion-advection systems in periodic media]{Uniqueness and stability of monostable pulsating fronts for multi-dimensional reaction-diffusion-advection systems in periodic media}

\author[Du, Li and Xin]{Li-Jun Du$^{1}$, Wan-Tong Li$^{2,*}$ and Ming-Zhen Xin$^{2}$}
\thanks{\hspace{-.6cm}
$^1$ School of Science, Chang'an University, Xi'an, Shaanxi 710064, P. R. China.}
\thanks{\hspace{-.6cm}
$^2$ School of Mathematics and Statistics, Lanzhou University, Lanzhou, Gansu 730000, P. R. China.}
\thanks{\hspace{-.6cm}
$^*$ {Corresponding author} (wtli@lzu.edu.cn)}

\date{\today}
\maketitle

\begin{abstract}

In this paper, we consider the phenomenon of monostable pulsating fronts for multi-dimensional reaction-diffusion-advection systems in periodic media. Recent results have addressed the existence of pulsating fronts and the linear determinacy of spreading speed  (Du, Li and Shen, \textit{J. Funct. Anal.} \textbf{282} (2022) 109415). In the present paper, we investigate the uniqueness and stability of monostable pulsating fronts with nonzero speed. We first derive precise asymptotic behaviors of these fronts as they approach the unstable limiting state. Utilizing these properties, we then prove the uniqueness modulo translation of pulsating fronts with nonzero speed. Furthermore, we show that these pulsating fronts are globally asymptotically stable for solutions of the Cauchy problem with front-like initial data. In particular,  we establish the uniqueness and global stability of the critical pulsating front in such systems. These results  are subsequently applied to a two-species competition system.

\noindent\textbf{Keywords:} Cooperative system; Uniqueness; Asymptotic stability; Critical pulsating traveling front; Competition system.

\noindent\textbf{AMS Subject Classification (2000)}: 35K57; 37C65; 92D30.
\end{abstract}

\tableofcontents

\section{Introduction}
\noindent

Different species inhabited in a common environment may cooperate or compete for living. Due to the presence of heterogeneities in natural environments, the spatial dynamics of reaction-diffusion systems in heterogeneous media is gaining more and more attention. The evolution of multiple components is often described by following reaction-diffusion-advection systems
\begin{equation}\label{u1-um}
\begin{cases}
\frac{\partial u_i(t,x)}{\partial t}=d_i(t,x)\Delta u_i+q_i(t,x)\cdot\nabla u_i
+f_i(t,x,u_1,u_2,\cdots,u_m),\quad \ \ x\in\mathbb R^N,\\
i=1,2,\cdots,m,
\end{cases}
\end{equation}
where $(u_1,u_2,\cdots,u_m)\in\mathbb R^m$, $m\ge 2$ and $N\ge 1$. In the biological context, $u_i$ may refer to population destinies of $m$ cooperative species under the settings $\partial f_i/\partial u_j\ge 0$ for $i, j=1,2,\cdots,m$, $i\not=j$. Among central dynamical issues of reaction-diffusion systems are propagation phenomena due to their widespread applications in biology, epidemiology, physics and chemistry, and a large number of researches have been carried out toward spreading speeds and monostable traveling wave solutions of some special kinds of multi-component system \eqref{u1-um}. For example, one can see \cite{Sattinger1976,Gardner1982,dunbar1983,kan1995,kan1996,Roquejoffre1996,lewis2002,LiweiLe2005, LiLinRuan2006,LiangZhao2007,huang2010,Liang2010,Faye2020} for the study of propagation phenomena in homogeneous media, \cite{zhao2011,zhao2014,bao2013,Yu2017,duJDE2019,DuBao2021,LiuOu2021,WangOu2021} for the study of reaction-diffusion systems with two components, \cite{LiangYiZhao2006,fang2014,Weinberger2002} for some abstract results in time or space periodic media, and \cite{Nolen2005,Nadin2009,fang2017,DuLiShen2022} for the study of propagation in
time-space periodic media. Recently, Du et al. \cite{DuLiShen2022} established some abstract results on monotone semiflows which can be used to study spreading speeds and periodic traveling waves of system \eqref{u1-um} with $m\ge 1$ in time-space periodic media.

However, the study on the uniqueness of traveling wave solutions and the convergence of the profile of solutions of the Cauchy problem to that of traveling wave solutions in heterogeneous media is much less known in literature. For scalar reaction-diffusion equations, Hamel and Roques \cite{hamel2011} proved the uniqueness and global stability of pulsating traveling fronts in spatially periodic media by using some qualitative properties of pulsating traveling fronts in periodic media  established in \cite{hamel2008}. Very recently, Guo \cite{guohj2023} proved some qualitative properties of pushed fronts for periodic reaction-diffusion-advection equations with general monostable nonlinearities. Shen \cite{Shen2011} investigated the existence, uniqueness and stability of generalized traveling solutions in time dependent equations, and further proved the stability of transition waves of Fisher-KPP equations with general
time and space dependence in \cite{Shen2017}. As long as the multi-component systems are concerned, the issues become more subtle and not much is known in the general case. In the time periodic media, Zhao and Ruan \cite{zhao2011,zhao2014} studied the existence, uniqueness and asymptotic stability of time periodic traveling waves for two-component reaction-diffusion competitive and cooperative systems.

To the best of our knowledge, there is no work on the uniqueness and stability of monostable traveling wave solutions of \eqref{u1-um} for $m\ge 2$ with space periodic and time independent coefficients, that is, concerning the following system
\begin{equation}\label{u1-um-system}
\begin{cases}
\frac{\partial u_i(t,x)}{\partial t}=d_i(x)\Delta u_i+q_i(x)\cdot\nabla u_i
+f_i(x,u_1,u_2,\cdots,u_m),\quad \ \ x\in\mathbb R^N, \\
i=1,2,\cdots,m,
\end{cases}
\end{equation}
where $\Delta:=\sum_{i=1}^{N}\frac{\partial^2}{\partial x_i^2}$, $\nabla:=(\frac{\partial}{\partial x_1},\cdots,\frac{\partial}{\partial x_N})$, $q_i=(q_{i1},\cdots,q_{iN})$, $d_i$, $q_{ik}\in C^{\nu}(\mathbb R^N)$ for some $\nu\in(0,1)$, $d_i(\cdot)\ge d_0>0$, $f_i(x,u_1,u_2,\cdots,u_m)$ are of class $C^{\nu}(\mathbb R^N)$
with respect to $x$ locally uniformly in $(u_1,u_2,\cdots,u_m)\in\mathbb R^m$, and of class $C^{2}(\mathbb R^m)$ with respect to $u_i$ locally uniformly in $x\in\mathbb R^N$, and $\partial f_i/\partial u_j\ge 0$ for $i,j=1,\cdots,m$, $i\not =j$. Moreover, the system is assumed to be $L$-periodic with respect to $L=(L_1,L_2,\cdots,L_N)$, in the sense that
$$d_i(x)=d_i(x+p), \quad q_i(x)=q_i(x+p), \quad f_i(x,u_1,u_2,\cdots,u_m)=f_i(x+p,u_1,u_2,\cdots,u_m)$$
for all $i=1,2,\cdots,m$, $x\in\mathbb R^N$ and $p\in\mathcal L$, where
$$\mathcal L:=\Pi_{i=1}^{N}L_i\mathbb Z,$$
and $L_1,\cdots,L_N$ are given positive real numbers, with the periodicity cell defined by
$$\mathcal D=\left\{x\in\mathbb R^N: x\in(0,L_1)\times\cdots\times(0,L_N)\right\}.$$

The objective of the current paper, as the follow-up of the paper \cite{DuLiShen2022} on propagation phenomena for periodic monotone semiflows and applications to cooperative systems in multi-dimensional media, is to further investigate the uniqueness and stability of pulsating traveling fronts (see Definition \ref{def}) of system \eqref{u1-um-system}. Noting that the evolution of two competitive species in the whole space is often described by the following reaction-diffusion-advection competition system
\begin{equation}\label{competition-sys}
\begin{cases}
\frac{\partial u_1(t,x)}{\partial t}=d_1(x)\Delta u_1+a_1(x)\cdot\nabla u_1
+u_1\left(b_1(x)-a_{11}(x)u_1-a_{12}(x)u_2\right),\\
\frac{\partial u_2(t,x)}{\partial t}=d_2(x)\Delta u_2+a_2(x)\cdot\nabla u_2
+u_2\left(b_2(x)-a_{21}(x)u_1-a_{22}(x)u_2\right),
\end{cases}
x\in\mathbb{R}^N,
\end{equation}
where $d_i$, $a_i$, $b_i$, $a_{ij}\in C^{\frac{\nu}{2},\nu}(\mathbb{R}^N)\;(\nu\in (0,1))$ are $L$-periodic functions, and $d_i(\cdot)\ge d_0>0$, $i,j=1,2$.
Under certain assumptions and by a change of variables, the competition system \eqref{competition-sys} can be transformed into a cooperative system
in the form \eqref{u1-um-system}. As an application, the uniqueness and stability of traveling wave fronts of \eqref{competition-sys} are discussed in this work.

As mentioned above, the study of uniqueness and stability of pulsating traveling waves become much more subtle in the general case. In \cite{hamel2011}, Hamel and Roques proved the asymptotic stability for solutions of the Cauchy problem with front-like initial data for spatially periodic scalar equations with general monostable nonlinearities, by using the result of exponential decay of traveling fronts in \cite{hamel2008}.  Later, Zhao and Ruan \cite{zhao2011,zhao2014} proved the asymptotic stability of time periodic traveling waves for two-component reaction-diffusion systems. Nevertheless, all these mentioned issues have been left open so far for multi-component systems with space dependence. Motivated by \cite{hamel2008,hamel2011,zhao2011,zhao2014}, this work aims to study the uniqueness and global stability of pulsating traveling fronts with nonzero speed for a general reaction-diffusion-advection cooperative systems \eqref{u1-um-system} in periodic media.

Firstly, we present some results concerning the existence and monotonicity in the co-moving frame coordinate of monostable pulsating traveling fronts, and give a set of sufficient conditions for the spreading speed to be linearly determinate. In fact, similar results were earlier established in \cite{DuBao2021,DuLiShen2022},
where some more general results on the existence and linear determinacy of the spreading speed in time-space periodic media were proved in \cite{DuLiShen2022},
and the results on the existence and monotonicity of pulsating traveling fronts for two-component cooperative systems in \cite{DuBao2021} can be extended to
the study of multi-component cooperative systems \eqref{u1-um-system}. In this part, we only state some main results and refer to \cite{DuBao2021,DuLiShen2022} for more details.

Secondly, we establish some exact asymptotic behavior properties of pulsating traveling fronts as they approach the unstable limiting state. These properties are not only of essential importance in deriving the uniqueness and stability of pulsating traveling fronts, but also play a key role in constructing some front-like entire solutions (see, e.g., \cite{duJDE2019}). One of the main difficulties relies on the interaction between multiple components of the system, as compared with the case of scalar equations, and hence some priori estimates of different components need to be established. In particular, we investigate the exact asymptotic behavior of the critical pulsating traveling front.

Thirdly, we prove the uniqueness of pulsating traveling fronts with any given speed. The general strategy is based on the sliding method, and the main difficulty comes to compare two given traveling wave fronts globally in $(x,s)\in\mathbb R^N\times\mathbb R$, especially in the region where they approach $\bm 0$ as $s\to-\infty$, and in particular, one needs to obtain a unified estimate for multiple components of the system, which is not present in the case of scalar equations.

Finally, we prove the global stability for solutions of the Cauchy problem with front-like initial data. The initial data is assumed to be close to the pulsating traveling front at $t=0$ at both ends, and it is proved that solutions of the Cauchy problem with such initial conditions converge to pulsating traveling fronts with a shift in time at large times. The general strategy of the proof is to trap the solution of the Cauchy problem with front-like initial data between appropriate sub- and supersolutions which are close to some shifts of the pulsating traveling front, and then to show that the shifts can be chosen small enough as $t\to\infty$. One of the main difficulties relies on the fact that the critical pulsating traveling front is not decaying as a purely exponential function but which multiplied with a polynomial factor $|s|$, and one must take this fact into account in constructing appropriate sub- and supersolutions in the critical case.

To give some precise observation of the main results, we consider the two-species competition system \eqref{competition-sys}. By introducing some specified assumptions, we shall show that \eqref{competition-sys} admits pulsating traveling fronts if and only if $c\ge c_+^0(e)$, where $c_+^0(e)$ is explicitly given by the eigenvalues of the periodic linearized problem. Furthermore, the pulsating traveling front with any given nonzero speed is unique modulo translation, and it is globally stable for solutions of the Cauchy problem with front-like initial data.

We would like to mention here that, though the general strategy of the current paper is motivated by  \cite{hamel2008,hamel2011,zhao2011,zhao2014}, our techniques and arguments become much more involved and complicated, and one needs to be more careful in dealing with system \eqref{u1-um-system}
due to the space dependence of the coefficients and the general coupling between different components in multi-component systems which becomes a nontrivial work. We also mention here that the critical pulsating traveling front presents a completely different asymptotic behavior at infinity which requires a different treatment comparing with the non-critical one. It seems to be the first time that the uniqueness and stability of general multi-component systems in periodic media is studied.

\subsection{Basic notations and assumptions}

In this subsection, we give some basic notations and assumptions of this paper.

Let
$$I=\{1,2,\cdots,m\}.$$
Denote
$$\mathbb R^m=\{\bm u=(u_1,u_2,\cdots,u_m):\ u_i\in\mathbb R,\ \forall\; i\in I\},\quad m\ge2,$$
where we equip $\mathbb R^m$ with the norm
$$|\bm u|:=\sum_{i=1}^m|u_i|,\quad \forall\;\bm u\in\mathbb R^m.$$
Usual notations for partial order in the space of functions in
$\mathbb R^m$ are used here, that is, for any $\bm{u}=(u_1,u_2,\cdots,u_m)$, $\bm{v}=(v_1,v_2,\cdots,v_m)$ and $c_1,c_2\in\mathbb R$,
$c_1\bm u\pm c_2\bm v=(c_1u_1\pm c_2v_1,c_1u_2\pm c_2v_2,\cdots,c_1u_m\pm c_2v_m)$,
the relation $\bm{u}\le\bm{v}$ (resp.  $\bm{u}\ll\bm{v}$)
is to be understood as $u_i\le v_i$ (resp. $u_i< v_i$)
for each $i$, and $\bm{u}<\bm{v}$ is to be understand $\bm{u}\leq \bm{v}$ but $\bm{u}\not\equiv \bm {v}$.
The other relations, such as ``max", ``min", ``sup" and ``inf", are similarly to be understood componentwise. In particular, denote $$\bm{0}=(0,\cdots,0),\qquad \bm{1}=(1,\cdots,1),\qquad [\bm{0},\bm{1}]=\{\bm{u}:\ \bm{0}\le\bm{u}\le\bm{1}\}.$$

In the following, we always use the vector-valued function $$\bm{u}(t,x)=(u_1(t,x),u_2(t,x),\cdots,u_m(t,x))$$
to denote the densities of $m$ species, and rewrite system \eqref{u1-um-system} as
\begin{equation}\label{m-system}
\frac{\partial\bm{u}(t,x)}{\partial t}=D(x)\Delta\bm{u}+q(x)\cdot\nabla\bm{u}
+\bm{F}(x,\bm{u}),\quad x\in\mathbb R^N,
\end{equation}
where $D(x)={\rm diag}\{d_i(x)\}_{i\in I}$, $q(x)={\rm diag}\{q_i(x)\}_{i\in I}$  with $q_i(x)=(q_{i1}(x),q_{i2}(x),\cdots,q_{iN}(x))$, and $$\bm{F}(x,\bm{u})=(f_1(x,\bm{u}),f_2(x,\bm{u}),\cdots,f_m(x,\bm{u})).$$

Let $X_p$ be the set of all continuous and $L$-periodic functions from $\mathbb R^N$ to $\mathbb R^m$ with the norm
$$|\bm{w}|_{p}:=\max_{x\in\mathbb R^N}|\bm{w}(x)|,\quad\forall\; \bm w\in X_p,$$
and $X_p^+:=\{\bm{w}\in X_p:\; \bm{w}(x)\ge\bm{0},\ \forall\; x\in\mathbb R^N\}$. For system \eqref{m-system}, we always assume that it admits two periodic solutions $\bm{p}^-(x)\ll\bm{p}^+(x)$ in $X_p$, and consider its propagation between $\bm{p}^-$ and $\bm{p}^+$. Noting that, without loss of generality, one can always assume that $\bm{p}^-=\bm{0}$ and $\bm{p}^+=\bm{1}$. In fact, by a change of variables
$$\tilde{\bm{u}}(t,x)=\frac{\bm{u}(t,x)-\bm{p}^-(x)}{\bm{p}^+(x)-\bm{p}^-(x)},$$
$\bm{0}$ and $\bm{1}$ can always be referred to as two periodic solutions of system \eqref{m-system}.
Let $E$ be the set of all periodic solutions of system \eqref{m-system} between
$\bm 0$ and $\bm 1$, that is,
\begin{align*}
E=\left\{\bm{\nu}\in X_p^+: \ \bm{0}\le\bm{\nu}\le\bm{1},\
\ \bm{\nu}(x)\ \text{is a periodic solution of \eqref{m-system}}\right\}.
\end{align*}
For any $\bm{\nu}=(\nu_1,\nu_2,\cdots,\nu_m)\in E\setminus\{\bm{1}\}$ and
$h_i\in C^{\nu,1}(\mathbb R^N\times\mathbb R^m)$, $i\in I$, denote
\begin{align*}
h_1^{\bm{\nu}}(x,u)&=h_1(x,u,\nu_2,\cdots,\nu_m),\\
h_i^{\bm{\nu}}(x,u)&=h_i(x,u_1,\cdots,u_{i-1},u,u_{i+1},\cdots,u_m),\quad i=2,3,\cdots,m.
\end{align*}

Assume that $d$, $q$ and $b$ are $L$-periodic functions in $C^{\nu}(\mathbb R^N)$,
and $d(x)\ge d_0>0$~for any $x\in\mathbb{R}^N$. By \cite[Proposition 1.12]{Berestycki2008},
the periodic eigenvalue problem
\begin{equation}\label{T-L-PE}
\begin{cases}
\lambda_0\phi=d(x)\Delta\phi+q(x)\cdot\nabla\phi+b(x)\phi, \quad x\in\mathbb R^N,\\
\phi(x)=\phi(x+p),\quad\forall\; p\in\mathcal{L}, \quad x\in\mathbb R^N
\end{cases}
\end{equation}
admits a principal eigenvalue $\lambda_0=\lambda_0(d,q,b)$ associated with
a periodic eigenfunction $\phi(x)>0$ for any $x\in\mathbb{R}^N$.

We make the following standing assumptions on \eqref{m-system}.

\begin{description}
\item[(H1)] $f_1(x,\bm{u})=u_1h_1(x,\bm{u})$ and
$f_i(x,\bm{u})=\sum_{j=1}^{i-1}a_{ij}(x)u_j+u_ih_i(x,\bm{u})$ for each $i\ge2$,
where $h_i\in C^{\nu,2}(\mathbb R^N\times\mathbb R^m)$ and $a_{ij}\in C^{\nu}(\mathbb R^N)$
are periodic in $x$. Moreover, for each $i\ge2$,
$h_i(x,\bm 0)<0$ and $a_{ij}(x)\ge0$ for any $x\in\mathbb R^N$,
and there exists $j\le i-1$ such that $a_{ij}(x)>0$ for any $x\in\mathbb R^N$.

\item[(H2)] $\bm F(x,\bm{1})\equiv\bm 0$, and
$h_1^{\bm{\nu}}(x,u)<h_1^{\bm{\nu}}(x,0)$ for any $u\in(0,1)$
and $\bm{\nu}\in E\setminus\{\bm{1}\}$.
\item[(H3)]
$\partial f_i(x,\bm{u})/\partial u_j\ge0$ for all $(x,\bm{u})\in\mathbb R^N\times[\bm 0,\bm 1]$,
where $i\not=j$, $i,j=1,\cdots,m$,
that is, $\bm F(x,\bm u)$ is cooperative in $\mathbb R^N\times[\bm 0,\bm 1]$.

\item[(H4)]
$\lambda_0(d_1,q_1,\zeta^1)>0$, where $\zeta^1(x):=h_1(x,\bm 0)$.

\item[(H5)] For any $\bm{u}_0\in X_p^+$ with $\bm{0}\ll\bm{u}_0\le\bm{1}$,
$$\lim_{t\to+\infty}|\bm{u}(t,x;\bm{u}_0)-\bm{1}|=0\quad\text{uniformly in}\ x\in\mathbb R^N,$$
where $\bm{u}(t,x;\bm{u}_0)$ is the solution of system \eqref{m-system} with
$\bm{u}(0,\cdot;\bm{u}_0)=\bm{u}_0$.
\end{description}

\begin{remark}\rm
(i) In view of (H1), it is easy to see that $\bm F(x,\bm{0})\equiv\bm 0$.

(ii) By (H1) and (H4), the Jacobian matrix $D_{\bm{u}}\bm{F}(\cdot,\bm{0})$
of $\bm F$ at $\bm 0$ admits a principal eigenvalue
$\lambda_0(d_1,q_1,\zeta^1)>0$ associated with a positive periodic eigenfunction,
and hence $\bm{0}$ is an unstable (invadable) periodic solution.

(iii) Noting that system \eqref{m-system} may admit boundary periodic solutions
between $\bm 0$ and $\bm 1$.

(iv) By (H2) and (H5), the periodic solution $\bm 1$ is globally stable
with respect to initial values in $X_p^+$. Moreover,
if $\bm{\nu}=(\nu_1,\nu_2,\cdots,\nu_m)$ is a periodic solution of \eqref{m-system}
such that $\bm{\nu}\in E\setminus\{\bm{1}\}$, then $\nu_1\equiv0$ (see, e.g., \cite{DuLiShen2022}).
\end{remark}

Under the periodic framework, the usual notion of traveling wave solutions
which are invariant in the frame moving in the direction of $e\in\mathcal S^{N-1}$
needs to be extended to that of pulsating traveling fronts, the definition of which is given as follows.

\begin{definition}\label{def}
Given a unit vector $e\in\mathcal S^{N-1}$, a pulsating traveling
front of \eqref{m-system} propagating in the direction of $e$
is a time-global solution $\bm{u}\in C^{1,2}(\mathbb R\times\mathbb R^N,[\bm{0},\bm{1}])$,
which can be written as
\begin{equation}\label{PF}
\bm{u}(t,x)=U(x,ct-x\cdot e),\quad\forall\;(t,x)\in\mathbb R\times\mathbb R^N,
\end{equation}
where $U(x,s)=(U_1(x,s),U_2(x,s),\cdots,U_m(x,s))$ is periodic in $x$ and nondecreasing in $s$,
and $c\not=0$ is called the wave speed.
Furthermore, we say that $U$ connects $\bm{0}$ to $\bm{1}$, if
$$\mathop{\lim}\limits_{s\to-\infty}|U(x,s)|=0,\quad
\mathop{\lim}\limits_{s\to+\infty}|U(x,s)-\bm{1}|=0\quad\text{uniformly in}\ x\in\mathbb R^N.$$
\end{definition}

Let
\begin{align*}
\Omega_z^-:&=\{(t,x)\in\mathbb R\times\mathbb R^N:\ ct-x\cdot e\le z\},\\
\Omega_z^+:&=\{(t,x)\in\mathbb R\times\mathbb R^N:\ ct-x\cdot e\ge z\},\\
\Omega_z^\sigma:&=\{(t,x)\in\mathbb R\times\mathbb R^N: z\le ct-x\cdot e\le\sigma\},
\ \ \ \forall\;z\le\sigma.
\end{align*}
Notice that \eqref{PF} can be rewritten as
\begin{equation*}
\bm{u}\left(\frac{s+x\cdot e}{c},x\right)=U(x,s),\quad\forall\;(x,s)\in\mathbb R^N\times\mathbb R,
\end{equation*}
where
\begin{gather}\label{uv-P}
\bm{u}\left(t-\frac{p\cdot e}{c},x\right)=\bm{u}(t,x+p),
\quad\forall\;(t,x)\in\mathbb R\times\mathbb R^N,\quad \forall\; p\in\mathcal L,
\end{gather}
and
$$
\mathop{\lim}\limits_{\varsigma\to-\infty}\mathop{\sup}
\limits_{(t,x)\in\Omega_{\varsigma}^-}|\bm{u}(t,x)|=0,\qquad
\mathop{\lim}\limits_{\varsigma\to+\infty}\mathop{\sup}
\limits_{\Omega_{\varsigma}^+}|\bm{u}(t,x)-\bm{1}|=0.
$$

To study exact asymptotic behaviors of pulsating traveling fronts as they
approach the unstable periodic solution, we need to introduce a few more notations.

Assume that $d$, $q$ and $\eta$ are $L$-periodic functions in $C^{\nu}(\mathbb R^N)$,
and $d(x)\ge d_0>0$~for any $x\in\mathbb{R}^N$. For any $e\in\mathcal S^{N-1}$ and $\lambda\in\mathbb R$, let
$\kappa_{e}(d,q,\eta,\lambda)$ be the principal eigenvalue of the operator
\begin{equation}\label{PE}
L_{e}(d,q,\eta,\lambda)
:=d(x)\Delta+(q(x)-2\lambda d(x)e)\cdot\nabla+(d(x)\lambda^2-\lambda q(x)\cdot e+\eta(x))
\end{equation}
acting on
$$C_{per}^N:={\left\{\phi\in C^2(\mathbb R^N):\; \phi(x)\,\,
\text{is periodic in}\,\, x\right\}},$$
associated with a periodic eigenfunction $\phi(x)>0$  for any $x\in\mathbb{R}^N$
(see, e.g., \cite[Proposition 1.12]{Berestycki2008}).

Denote
\begin{equation}\label{kappa-i}
\kappa_i(\lambda,e):=\kappa_{e}(d_i,q_i,\zeta^i,\lambda),\quad i=1,2,\cdots,m,
\end{equation}
where
\begin{equation}\label{zeta-i}
\zeta^i(x):=h_i(x,\bm 0),\quad i=1,2,\cdots,m.
\end{equation}
Noting that the function $\lambda\mapsto\kappa_1(\lambda,e)$ is analytic and convex in $\mathbb R$
for any fixed $e$ (see, e.g., \cite[Lemma 3.1]{fang2017}). Moreover, $\kappa_1(0,e)=\lambda_0(d_1,q_1,\zeta^1)>0$ by (H4).
Let
\begin{equation*}
c_+^0(e):=\inf_{\lambda>0}\frac{\kappa_1(\lambda,e)}{\lambda},
\end{equation*}
then $c_+^0(e)$ is well defined for each $e$, and there exists $\lambda_+^0(e)>0$ such that
\begin{equation}\label{minimal-ss}
c_+^0(e)=\inf_{\lambda>0}\frac{\kappa_1(\lambda,e)}{\lambda}
=\frac{\kappa_1(\lambda_+^0(e),e)}{\lambda_+^0(e)}.
\end{equation}
Let
$$F_{c}=\{\lambda>0: \kappa_1(\lambda,e)-c\lambda=0\},\quad\forall\;c\ge c_+^0(e).$$
It is known that (see, e.g., \cite[Lemma 2.1]{hamel2008}) the positive real number
\begin{equation}\label{c1}
\lambda_c:=\min F_{c}
\end{equation}
is well defined, and in particular, $F_{c_+^0(e)}=\{\lambda_+^0(e)\}$.
Moreover, $0<\lambda_c\le\lambda_+^0(e)$ for any $c\ge c_+^0(e)$.

In the rest of the paper, we let $e\in\mathcal S^{N-1}$ be any given unit vector,
and use the notation
$$c_+^0=c_+^0(e),\quad\lambda_+^0=\lambda_+^0(e),
\quad\kappa_i(\lambda)=\kappa_i(\lambda,e),\ \ i=1,2,\cdots,m$$
without confusion of the dependence of $c_+^0(\cdot)$, $\lambda_+^0(\cdot)$
and $\kappa_i(\lambda,\cdot)$ on $e$.

Consider the following periodic eigenvalue problem
\begin{equation}\label{u1um-linear-PE}
\begin{cases}
\kappa\phi_1=d_1\Delta\phi_1+(q_1-2\lambda d_1e)\cdot\nabla\phi_1
+(d_1\lambda^2-\lambda q_1\cdot e+\zeta^1(x))\phi_1,\\
\kappa\phi_j=d_j\Delta\phi_j+(q_j-2\lambda d_je)\cdot\nabla\phi_j
+\sum\limits_{k=1}^{j-1}a_{jk}\phi_k+(d_j\lambda^2-\lambda q_j\cdot e+\zeta^j(x))\phi_j,
\ \ j=2,3,\cdots,m,\\
\phi_i(x)=\phi_i(x+p),\quad\forall\;x\in\mathbb R^N,\ \ i=1,2,\cdots,m,\ \ p\in\mathcal{L},
\end{cases}
\end{equation}
where $\zeta^i(x)$ is given by \eqref{zeta-i}, and $\lambda\in\mathbb R$ is a constant.

\begin{lemma}\label{P-EFunc}
Assume (H1)-(H5). If $\kappa_1(\lambda_+^0)>\max_{j=2,3,\cdots,m}\kappa_j(\lambda_+^0)$,
then for any $0\le\lambda\le\lambda_+^0$,
problem \eqref{u1um-linear-PE} admits a positive periodic eigenfunction
$\bm{\Phi}_{\lambda}(x)=(\phi_1^{\lambda}(x),\phi_2^{\lambda}(x),\cdots,\phi_m^{\lambda}(x))$
associated with the principal eigenvalue $\kappa=\kappa_1(\lambda)$.
\end{lemma}
\begin{proof}
Noting that $\kappa_i(\lambda)$ is convex in $\lambda\in\mathbb R$
for each $i\in I$, and that $\kappa_1(0)=\lambda_0(d_1,q_1,\zeta^1)>0
>\max_{j=2,3,\cdots,m}\lambda_0(d_j,q_j,\zeta^j)
=\max_{j=2,3,\cdots,m}\kappa_j(0)$ by (H1) and (H4), which together with
$\kappa_1(\lambda_+^0)>\max_{j=2,3,\cdots,m}\kappa_j(\lambda_+^0)$ yields that
\begin{equation}\label{1-j}
\kappa_1(\lambda)>\max_{j=2,3,\cdots,m}\kappa_j(\lambda),
\quad\forall\;0\le\lambda\le\lambda_+^0.
\end{equation}
For each $0\le\lambda\le\lambda_+^0$, let $\phi_1^{\lambda}(x)>0$ be the periodic eigenfunction associated with the principal eigenvalue $\kappa_1(\lambda)$.
Noting that $a_{21}(x)\phi_1^{\lambda}(x)>0$ for any $x\in\mathbb{R}^N$ and $\kappa_1(\lambda)>\kappa_2(\lambda)$,
by using arguments similar to those of \cite[Proposition 4.2]{Yu2017},
there exists a periodic function $\phi_2^{\lambda}(x)>0$ of
the $\phi_2$-equation in \eqref{u1um-linear-PE} with $\phi_1=\phi_1^{\lambda}$,
associated with the principal eigenvalue $\kappa_1(\lambda)$.
Since for each $j=3,4,\cdots,m$, there exists $k\le j-1$ such that $a_{jk}\gneq0$ by (H1),
an induction argument shows that there exists a periodic function
$\phi_j^{\lambda}(x)>0$ of the $\phi_j$-equation in \eqref{u1um-linear-PE} with $\phi_k=\phi_k^{\lambda}$, $k=1,2,\cdots,j-1$, associated with $\kappa_1(\lambda)$.
Let $\bm{\Phi}_{\lambda}(x):=(\phi_1^{\lambda}(x),\phi_2^{\lambda}(x),\cdots,\phi_m^{\lambda}(x))$,
then $\bm{\Phi}_{\lambda}(x)$ is a positive periodic eigenfunction of \eqref{u1um-linear-PE}
associated with the principal eigenvalue $\kappa=\kappa_1(\lambda)$.
The proof is complete.
\end{proof}

Next, consider the following periodic linearized system of \eqref{m-system} at $\bm 0$
\begin{equation}\label{u1u2-linearized}
\frac{\partial\bm{u}(t,x)}{\partial t}=D(x)\Delta\bm{u}+q(x)\cdot\nabla\bm{u}
+D_{\bm u}\bm F(x,\bm{0})\bm u,\quad (t,x)\in\mathbb R\times\mathbb R^N,
\end{equation}
where
$$D_{\bm u}\bm F(x,\bm{u})
:=\left(\frac{\partial f_1(x,\bm u)}{\partial \bm u},\cdots,
\frac{\partial f_m(x,\bm u)}{\partial \bm u}\right)^T,\quad
\frac{\partial f_i(x,\bm u)}{\partial \bm u}
:=\left(\frac{\partial f_i(x,\bm u)}{\partial u_1},\cdots,
\frac{\partial f_i(x,\bm u)}{\partial u_m}\right).$$

We introduce a concept of \emph{front-like linearized solutions} of system \eqref{m-system}
as follows.

\begin{definition}
For any $c\ge c_+^0$, an entire solution $\bm{w}_c\in C^{1,2}(\mathbb R\times\mathbb R^N)$
of the linearized system \eqref{u1u2-linearized} is called a
front-like linearized solution of \eqref{m-system},
if it can be written as $\bm{w}_c(t,x)=W(x,ct-x\cdot e)$, where
$W(x,s)$ is periodic in $x$ and nondecreasing in $s$, and $\mathop{\lim}\limits_{s\to-\infty}|W(x,s)|=0$.
\end{definition}

\begin{remark}\rm
Let $0<\lambda_c\le\lambda_+^0$ be such that $\kappa_1(\lambda_c)=c\lambda_c$, and define
\begin{equation}\label{Linear-sol}
\bm{w}_c(t,x)=e^{\lambda_c(ct-x\cdot e)}\bm{\Phi}_{\lambda_c}(x),
\quad (t,x)\in\mathbb R\times\mathbb R^N,
\end{equation}
where $\bm{\Phi}_{\lambda_c}(x)>\bm 0$ is the periodic eigenfunction
associated with $\kappa_1(\lambda_c)$ given by Lemma \ref{P-EFunc}.
Then it is easy to verify that $\bm{w}_c$ is a front-like linearized solution of \eqref{m-system}.
\end{remark}

To this end, we introduce the following assumptions.
\begin{description}
\item[(H6)] $\kappa_1(\lambda_+^0)>\max_{j=2,3,\cdots,m}\kappa_j(\lambda_+^0)$.
\item[(H7)] $h_i(x,\bm{w}_c)\le h_i(x,\bm 0)$ for each $i$ and
front-like linearized solution $\bm{w}_c$ of \eqref{m-system} with $c\ge c_+^0$.
\item[(H8)] The periodic eigenvalue problem
\begin{equation*}
\begin{cases}
\mu\bm{\Psi}=D(x)\Delta\bm{\Psi}+q(x)\cdot\nabla\bm{\Psi}+D_{\bm u}\bm F(x,\bm{1})\bm{\Psi},
\quad x\in\mathbb R^N,\\
\bm{\Psi}(x)=\bm{\Psi}(x+p),\quad \forall\;x\in\mathbb R^N,\ p\in\mathcal L
\end{cases}
\end{equation*}
admits a principal eigenvalue $\mu=\mu^-<0$ associated with a positive periodic eigenfunction
$\bm{\Psi}(x)=(\psi_1(x),\psi_2(x),\cdots,\psi_m(x))$.
\end{description}

\begin{remark}\rm
(i) Noting that (H6) holds in particular if $d_j\equiv d_1$, $q_j\equiv q_1$
and $h_1(x,\bm 0)>h_j(x,\bm 0)$ for all $x\in\mathbb{R}^N$ and $j=2,3,\cdots,m$.

(ii) By (H7), the nonlinearity $\bm F$ is of KPP type along any front-like linearized solution $\bm{w}_c$, in the sense that
\begin{equation*}
\bm F(x,\bm{w}_c)\le D_{\bm u}\bm F(x,\bm{0})\bm{w}_c,
\quad\forall\; x\in\mathbb R^N,\ \ \forall\; c\ge c_+^0.
\end{equation*}
\end{remark}

In the following, we are devoted to study system \eqref{m-system} under assumptions (H1)-(H8).

\subsection{Main results}

In this subsection, we first state some known results established in \cite{DuLiShen2022} on
the existence of pulsating traveling fronts and the linear determinacy of the spreading speed,
then we present our main results of this work. For this purpose, some more notations need to be introduced.

Let $\bm{\nu}\in E\setminus\{\bm{0},\bm{1}\}$. Then $h_1^{\bm{\nu}}(x,0)=h_1(x,0,\nu_2,\cdots,\nu_m)\ge\zeta^1(x)$ by (H3), and
it follows from (H4) that
$\kappa_{e}(d_1,q_1,h_1^{\bm{\nu}}(\cdot,0),0)
=\lambda_0(d_1,q_1,h_1^{\bm{\nu}}(\cdot,0))\ge\lambda_0(d_1,q_1,\zeta^1)>0$,
where $\kappa_{e}(\cdot,\cdot,\cdot,\lambda)$ is the principal eigenvalue of
the operator \eqref{PE}. Hence the quantity
\begin{equation}\label{scalar-rightward}
c_{\bm{\nu}}^-(e):=\inf_{\lambda>0}\frac{\kappa_{e}(d_1,q_1,h_1^{\bm{\nu}}(\cdot,0),\lambda)}{\lambda}
\end{equation}
is well defined.
Noting that for any $\bm{\nu}\in E\setminus\{\bm 0, \bm{1}\}$,
there exists $2\le l\le m$ such that
$$\bm{\nu}=\bm{\nu}^l=(0,\cdots,0,\nu_l,\nu_{l+1},\cdots,\nu_m),
\quad \text{where}\ \nu_l\not\equiv0.$$
Let
$$g(x,\bm{\nu}^l)=\frac{\partial f_l}{\partial u_l}(x,\bm{\nu}^l)
=h_l^{\bm{\nu}^l}(x,\nu_l)+\nu_lr^{\bm{\nu}^l}(x,\nu_l),$$
where
$$r^{\bm{\nu}^l}(x,w):=\frac{\partial h_l}{\partial u_l}(x,0,\cdots,0,w,\nu_{l+1},\cdots,\nu_m).$$

We make the following assumption on boundary periodic solutions of \eqref{m-system}.

\begin{description}
\item[(C)] For any $\bm{\nu}^l$, $\bm{\nu}_1$, $\bm{\nu}_2\in E\setminus\{\bm{0},\bm{1}\}$,
there hold
\begin{itemize}
\item[(C1)] $\lambda_0(d_l,q_l,g(\cdot,\bm{\nu}^l))>0$.

\item[(C2)] $r^{\bm{\nu}^l}(x,\nu_l)\ge\max\{0,r^{\bm{\nu}^l}(x,w)\}$ for any $0\le w\le\nu_l$.

\item[(C3)] $c_{\bm{\nu}_1}^+(e)+c_{\bm{\nu}_2}^-(e)>0$,
    where $c_{\bm{\nu}}^-(e)$ is given by \eqref{scalar-rightward}, and
    $c_{\bm{\nu}}^+(e)$ is defined by
\begin{equation*}
c_{\bm{\nu}^l}^+(e)
:=\inf_{\lambda>0}\frac{\kappa_{e}(d_l,q_l,g(\cdot,\bm{\nu}^l),-\lambda)}{\lambda}.
\end{equation*}
\end{itemize}
\end{description}

The existence and nonexistence of pulsating traveling fronts are stated as follows.

\begin{theorem}[see \cite{DuBao2021,DuLiShen2022}]\label{Ex-PTW}
Assume (H1)-(H5) and (C). Then for each $e\in\mathcal S^{N-1}$, there exists
$c^*_+(e)$ such that for any $c\ge c^*_+(e)$,
system \eqref{m-system} admits a pulsating traveling front
$U(x,ct-x\cdot e)$ connecting $\bm{0}$ to $\bm{1}$, and for any $c<c^*_+(e)$,
there is no such a pulsating traveling front.
Moreover, $U_s(x,s)\textcolor{red}{>}\textcolor{blue}{\gg}\bm 0$ for any $(x,s)\in\mathbb R^N\times\mathbb R$.
\end{theorem}

In Theorem \ref{Ex-PTW}, the quantity $c^*_+(e)$ is called the (fastest)
spreading speed of system \eqref{m-system}.
Next, we give a set of sufficient conditions for the spreading speed to be linearly determinate.
Recall that $c_+^0(e)$ is defined by \eqref{minimal-ss} as
$$c_+^0(e)=\inf_{\lambda>0}\frac{\kappa_1(\lambda,e)}{\lambda}.$$
By \cite[Lemma 3.3]{DuLiShen2022}, we know that $c^*_+(e)\ge c_+^0(e)$ for any $e\in\mathcal{S}^{N-1}$.
The \emph{linear determinacy} of the spreading speed $c^*_+(e)$ is defined to mean that
$$c^*_+(e)=c_+^0(e)=:c_*(e).$$

We introduce the following assumption.

\begin{description}
\item[(D)] $h_1^{\bm{\nu}}(x,0)\textcolor{red}{\gneq}\textcolor{blue}{>} h_1^{\bm 0}(x,0)$
for any $\bm{\nu}\in E\setminus\{\bm{0},\bm{1}\}$.
\end{description}

\begin{remark}\rm
Noting that if there exists $j\in\{2,3,\cdots,m\}$ such that
$\frac{\partial h_1}{\partial u_j}(x,\bm u)\not\equiv0$ for
any $(x,\bm u)\in\mathbb R^N\times[0,\bm\nu]$ with $\bm{\nu}\in E\setminus\{\bm{0},\bm{1}\}$,
then assumption (D) holds.
\end{remark}

The linear determinacy of the spreading speed is stated as follows.

\begin{theorem}[see \cite{DuLiShen2022}]\label{line-det}
Assume (H1)-(H7) and (D). Then
$$c^*_+(e)=c_+^0(e)=\inf_{\lambda>0}\frac{\kappa_1(\lambda,e)}{\lambda}.$$
\end{theorem}

In the sense of Theorems \ref{Ex-PTW} and \ref{line-det}, we call
$U(x,ct-x\cdot e)$ the \emph{super-critical pulsating traveling fronts} provided
$c>c_+^0(e)$, and $U(x,ct-x\cdot e)$ the \emph{critical pulsating traveling fronts}
provided $c=c_+^0(e)$, which was described as linear and nonlinear speed selection for
monostable wave propagations in literature (see, e.g., \cite{WangOu2021}).

\begin{remark}\rm\label{remark-on-existence}
(i) The proofs of Theorems \ref{Ex-PTW} and \ref{line-det}
can be shown by using the abstract results established in \cite{DuLiShen2022},
in which the authors proved these results for
time-space periodic cooperative systems, which can be directly used to
prove Theorems \ref{Ex-PTW} and \ref{line-det}, by letting the Poincar\'{e} map $Q_T=Q_1$ in \cite{DuLiShen2022}.

(ii) The existence as well as monotonicity of pulsating traveling fronts
can also be proved by using similar arguments to \cite[Theorem 3.1]{DuBao2021},
in which the authors considered spatially periodic two-component systems.

(iii) Noting from the proof of \cite[Theorem 3.2]{DuLiShen2022} that assumption (H7) only
need to be satisfied for $\bm{w}_c=\bm{w}_{c_+^0}$ given by \eqref{Linear-sol}
in proving Theorem \ref{line-det}.
\end{remark}

We are now in position to state the main results of the present work.
In the following of the paper, we always assume that {\bf (H1)}-{\bf (H8)} hold.
Let
$$U(x,ct-x\cdot e)=(U_1(x,ct-x\cdot e),U_2(x,ct-x\cdot e),\cdots,U_m(x,ct-x\cdot e))$$
be a pulsating traveling front of \eqref{m-system} connecting $\bm 0$ to $\bm 1$,
then $c\ge c_+^0(e)$, where
\begin{equation*}
c_+^0(e)=\inf_{\lambda>0}\frac{\kappa_1(\lambda,e)}{\lambda}
=\frac{\kappa_1(\lambda_+^0,e)}{\lambda_+^0}.
\end{equation*}

Our first main result is concerned with the exact asymptotic behavior
of pulsating traveling fronts as they approach the unstable limiting state,
which is stated as follows.

\begin{theorem}\label{Asy}
Assume (H1)-(H8). Let $U(x,ct-x\cdot e)$ be a pulsating traveling front of \eqref{m-system}.
Then there exists $\rho>0$ such that
\begin{itemize}
\item[(i)]
If $c>c_+^0(e)$, then
\begin{equation*}
\mathop{\lim}\limits_{s\to-\infty}\frac{U(x,s)}{\rho e^{\lambda_c s}\bm{\Phi}_{\lambda_c}(x)}
=\bm 1\quad\text{uniformly in}\ x\in\mathbb R^N.
\end{equation*}

\item[(ii)]
If $c=c_+^0(e)$, then
\begin{equation*}
\mathop{\lim}\limits_{s\to-\infty}\frac{U(x,s)}{\rho|s|e^{\lambda_+^0s}
\bm{\Phi}_{\lambda_+^0}(x)}
=\bm 1\text{~~uniformly in~~}x\in\mathbb R^N.
\end{equation*}
\end{itemize}
\end{theorem}

Theorem \ref{Asy} shows that the super-critical pulsating traveling fronts
are decaying exponentially to $\bm 0$ as $s\to-\infty$, while
the critical pulsating traveling front is decaying as an exponential function multiplied with a polynomial factor $|s|$. These results can be viewed as an extension of asymptotic behaviors
of pulsating traveling fronts for periodic scalar equations (see, e.g., \cite{hamel2008})
to periodic multi-component systems.

Using these asymptotic behavior properties, we obtain
the following result of the uniqueness of pulsating traveling fronts.

\begin{theorem}\label{uniqueness}
Assume (H1)-(H8).
Let $\bm{u}(t,x)=U(x,ct-x\cdot e)$ and $\bm{v}(t,x)=V(x,ct-x\cdot e)$ be
two pulsating traveling fronts of \eqref{m-system} with $c\not=0$.
Then there exists $z_0\in\mathbb R$ such that
\begin{equation*}
U(x,s+z_0)=V(x,s),\quad\forall\;(x,s)\in\mathbb R^N\times\mathbb R,
\end{equation*}
that is, there exists $\sigma\in\mathbb R$ ($\sigma=z_0/c$) such that
\begin{equation*}
\bm u(t+\sigma,x)=\bm v(t,x),\quad\forall\;(t,x)\in\mathbb R\times\mathbb R^N.
\end{equation*}
\end{theorem}

Theorem \ref{uniqueness} yields the uniqueness, modulo translation, of
pulsating traveling fronts with nonzero speed in a given direction of $e$.
Notice that if $z_0\not=0$, then $U\not=V$ since all the fronts are strictly
monotone in the co-moving frame coordinate.

To this end, we give the global stability of pulsating traveling fronts
for solutions of the Cauchy problem with front-like initial data.
Let $Y=BUC(\mathbb R^N,\mathbb R^m)$ be the set of all bounded and uniformly
continuous functions from $\mathbb R^N$ to $\mathbb R^m$ with the norm
$$\|\bm{u}\|:=\max_{x\in\mathbb R^N}|\bm{u}(x)|,\quad\forall\; \bm u\in Y,$$
and $Y_+:=\{\bm{u}\in Y:\; \bm{u}(x)\ge\bm{0},\ \forall\; x\in\mathbb R^N\}$. The relation $\bm{u}\le\bm{v}$
is to be understood as $u_i(x)\le v_i(x)$
for each $i$, and $x\in\mathbb{R}^N$ and $\bm{u}<\bm{v}$ is to be understand $\bm{u}\leq \bm{v}$ but $\bm{u}\not\equiv \bm {v}$.
\begin{theorem}\label{stability}
Assume (H1)-(H8).
Let $U(x,ct-x\cdot e)$ be a pulsating traveling front of \eqref{m-system} with $c\ge c_+^0(e)$,
and $\bm u(t,x;\bm u_0)$ be a solution of \eqref{m-system} with
initial value $\bm{u}(0,\cdot;\bm{u}_0)=\bm{u}_0\in Y_+$.
Assume that $\bm 0<\bm u_0<\bm 1$, and that
\begin{equation}\label{u0v0-main}
\liminf_{\varsigma\to+\infty}\left\{\inf_{x\in\mathbb R^N,\ -x\cdot e\ge\varsigma}\bm u_0(x)\right\}
\ge(1-\varepsilon_0)\bm 1
\end{equation}
for some $\varepsilon_0\in(0,\frac{1}{2})$ small enough.
Moreover, we assume that there exists $k>0$ such that
\begin{equation}\label{u0v0-behavior-main}
\limsup_{\varsigma\to-\infty}\left\{\sup_{\substack{x\in\mathbb R^N\\-x\cdot e\le\varsigma}}
\left|\frac{\bm u_0(x)}{k|x\cdot e|^\tau e^{-\lambda_c(x\cdot e)}
\bm{\Phi}_{\lambda_c}(x)}-\bm 1\right|\right\}=0,
\end{equation}
where $\tau=0$ if $c>c_+^0(e)$ and $\tau=1$ if $c=c_+^0(e)$.
Then there exists $s_0\in\mathbb R$ such that
\begin{equation*}
\lim_{t\to\infty}\sup_{x\in\mathbb R^N}|\bm u(t,x;\bm u_0)-U(x,ct-x\cdot e+s_0)|=0.
\end{equation*}
\end{theorem}

Theorem \ref{stability} shows that if the front-like initial data is close in some sense
to the pulsating traveling front at $t=0$ at both ends, then solutions
of the Cauchy problem converge to the pulsating traveling front
with a shift in time at large times, that is, the propagation speed of
the solution $\bm u(t,x;\bm u_0)$ at large times strongly depends on the
asymptotic behavior of the initial value $\bm u_0$ as it approaches the unstable state $\bm 0$.
The stability of pulsating traveling fronts of reaction-diffusion systems is
indeed one of the most important observation in understanding the
large time behavior of solutions of the Cauchy problem.
Due to the general framework and assumptions, and the interaction of multiple components
in the system, the proof of this result is rather involved and requires some careful treatments.

At the end of this section, we discuss some applications of the main results to two-species competition system
\begin{equation*}
\begin{cases}
\frac{\partial u_1(t,x)}{\partial t}=d_1(x)\Delta u_1+a_1(x)\cdot\nabla u_1
+u_1\left(b_1(x)-a_{11}(x)u_1-a_{12}(x)u_2\right),\\
\frac{\partial u_2(t,x)}{\partial t}=d_2(x)\Delta u_2+a_2(x)\cdot\nabla u_2
+u_2\left(b_2(x)-a_{21}(x)u_1-a_{22}(x)u_2\right),
\end{cases}x\in\mathbb R^N,
\end{equation*}
where $d_i$, $a_i$, $b_i$, $a_{ij}\in C^{\nu}(\mathbb R^N)$
are $L$-periodic functions, $d_i(x)\ge d_0>0$ and
$a_{ij}(x)\ge a_0>0$ ($i,j=1,2$) for any $x\in\mathbb R^N$.

Note that if $\lambda_0(d_i,a_i,b_i)>0$ for $i=1,2$, then
there exist two positive periodic functions $u_1^*(x)$ and $u_2^*(x)$
such that $(u_1^*(x),0)$ and $(0,u_2^*(x))$ are two periodic solutions of \eqref{competition-sys}.
We make the following standing assumptions for \eqref{competition-sys}.

\begin{description}
\item[(A1)] $\lambda_0(d_i,a_i,b_i)>0$ for $i=1,2$, and $\lambda_0(d_1,a_1,b_1-a_{12}u_2^*)>0$.
\item[(A2)] System \eqref{competition-sys} has no positive periodic solution between $(0,0)$ and $(u_1^*,u_2^*)$.
\end{description}

By (A1), we see that $(0,u_2^*(x))$ is an unstable periodic solution of \eqref{competition-sys},
which together with (A2) shows that $(u_1^*(x),0)$ is globally asymptotically stable
for all initial values $(\phi_1,\phi_2)\in\mathbb P_+$ with $\phi_1\not\equiv0$
(see, e.g., \cite[Theorem 2.1]{Yu2017}), where
$\mathbb P$ is the set of all continuous and periodic functions from $\mathbb R^N$ to
$\mathbb R^2$ with the maximum norm $|\cdot|$,
and $\mathbb P_+:=\{(\phi_1,\phi_2)\in\mathbb P:\ (\phi_1,\phi_2)\ge(0,0),\ \forall x\in\mathbb R^N\}$.

Using a change of variables
$$\tilde u_1(t,x)=\frac{u_1(t,x)}{u_1^*(x)},\quad\tilde u_2(t,x)=\frac{u_2^*(x)-u_2(t,x)}{u_2^*(x)}$$
and dropping the title, we transform \eqref{competition-sys} into the following system
\begin{equation}\label{LV-cooper}
\begin{cases}
\frac{\partial u_1(t,x)}{\partial t}=d_1(x)\Delta u_1+q_1(x)\cdot\nabla u_1+f_1(x,u_1,u_2),\\
\frac{\partial u_2(t,x)}{\partial t}=d_2(x)\Delta u_2+q_2(x)\cdot\nabla u_2+f_2(x,u_1,u_2),
\end{cases}x\in\mathbb{R}^N,
\end{equation}
where $q_i(x)=a_i(x)+2d_i(x)\nabla u_i^*(x)/u_i^*(x)$ for $i=1,2$, and
\begin{align*}
f_1(x,u_1,u_2)&=u_1h_1(x,u_1,u_2),\hspace{1.88cm}
h_1(x,u_1,u_2)=a_{11}^*(x)(1-u_1)-a_{12}^*(x)(1-u_2),\\
f_2(x,u_1,u_2)&=a_{21}^*(x)u_1+u_2h_2(x,u_1,u_2), ~ h_2(x,u_1,u_2)=a_{22}^*(x)(u_2-1)-a_{21}^*(x)u_1,
\end{align*}
and $a_{11}^*(x)=a_{11}(x)u_1^*(x)$, $a_{12}^*(x)=a_{12}(x)u_2^*(x)$,
$a_{21}^*(x)=a_{21}(x)u_1^*(x)$, $a_{22}^*(x)=a_{22}(x)u_2^*(x)$.
Noting that system \eqref{LV-cooper} has three periodic solutions
$\bm{0}$, $\bm{\nu}$ and $\bm{1}$, where $\bm{\nu}:=(0,1)$, that is,
$$E=\{\bm{0},\bm{\nu},\bm{1}\}.$$

Let
\begin{equation*}
c_+^0=\inf_{\lambda>0}\frac{\kappa_{e}(d_1,q_1,a_{11}^*-a_{12}^*,\lambda)}{\lambda}
=\frac{\kappa_{e}(d_1,q_1,a_{11}^*-a_{12}^*,\lambda_+^0)}{\lambda_+^0}.
\end{equation*}
For any given $c\ge c_+^0$,
let $(\phi_1^c(x),\phi_2^c(x))$ be the positive periodic eigenfunction
associated with $\kappa_{e}(d_1,q_1,a_{11}^*-a_{12}^*,\lambda_c)$ given by Lemma \ref{P-EFunc}.
We introduce assumptions (A3)-(A6) as follows.

\begin{description}
\item[(A3)] $a_{11}u_1^*> a_{12}u_2^*$ and $a_{22}u_2^*>a_{21}u_1^*$.
\item[(A4)]
$\kappa_e(d_1,q_1,a_{11}^*-a_{12}^*,\lambda_+^0)>\kappa_e(d_2,q_2,-a_{22}^*,\lambda_+^0)$.
\item[(A5)] $\frac{u_1^*(x)\phi_1^c(x)}{u_2^*(x)\phi_2^c(x)}
\ge\max\left\{\frac{a_{12}(x)}{a_{11}(x)},\frac{a_{22}(x)}{a_{21}(x)}\right\}$,
$\forall\;x\in\mathbb R^N$, $c\ge c_+^0$.

\item[(A6)] $c_{\bm\nu}^-(e)+c_{\bm\nu}^+(e)>0$, where $\bm\nu=(0,1)$ and
$$c_{\bm\nu}^-(e)=\inf_{\lambda>0}\frac{\kappa_e(d_1,q_1,a_{11}^*,\lambda)}{\lambda},\qquad
c_{\bm\nu}^+(e)=\inf_{\lambda>0}\frac{\kappa_e(d_2,q_2,a_{22}^*,-\lambda)}{\lambda}.$$
\end{description}

It is not difficult to verify that all assumptions of (H1)-(H8), (C) and (D) hold true
for system \eqref{LV-cooper} under assumptions (A1)-(A6). By Theorems \ref{Asy}-\ref{stability},
we have the following results.

\begin{theorem}
Assume (A1)-(A6). Then the following statements are valid:
\begin{itemize}
\item[(1)] For any $c\ge c_+^0(e)$, system \eqref{LV-cooper} admits a pulsating traveling front
$(U_1(x,ct-x\cdot e),U_2(x,ct-x\cdot e))$ connecting $(0,0)$ to $(1,1)$,
and for any $c<c_+^0(e)$, there is no such a front.

\item[(2)] Let $(U_1(x,ct-x\cdot e),U_2(x,ct-x\cdot e))$ be a pulsating traveling front of \eqref{LV-cooper}. Then there exists $\rho>0$ such that
\begin{equation*}
\mathop{\lim}\limits_{s\to-\infty}\frac{U_1(x,s)}{\rho|s|^\tau e^{\lambda_c s}\phi_1^c(x)}=1,\quad
\mathop{\lim}\limits_{s\to-\infty}\frac{U_2(x,s)}{\rho|s|^\tau e^{\lambda_c s}\phi_2^c(x)}=1
\text{~~uniformly in~~}x\in\mathbb R^N,
\end{equation*}
where $\tau=0$ if $c>c_+^0(e)$ and $\tau=1$ if $c=c_+^0(e)$.

\item[(3)] If $(V_1(x,ct-x\cdot e),V_2(x,ct-x\cdot e))$ is a pulsating traveling front of \eqref{LV-cooper}, then there exists $z_0\in\mathbb R$ such that
$$(U_1(x,s+z_0),U_2(x,s+z_0))=(V_1(x,s),V_2(x,s)),
\quad\forall\;(x,s)\in\mathbb R^N\times\mathbb R.$$

\item[(4)] Let $(u_1(t,x;u_{01},u_{02}),u_2(t,x;u_{01},u_{02}))$
be a solution of \eqref{LV-cooper} with $(0,0)<(u_{01},u_{02})<(1,1)$ satisfying
\eqref{u0v0-main} and \eqref{u0v0-behavior-main}. Then there exists $s_0\in\mathbb R$ such that
\begin{align*}
\lim_{t\to\infty}&\left\{\sup_{x\in\mathbb R^N}|u_1(t,x;u_{01},u_{02})-U_1(x,ct-x\cdot e+s_0)| \right .\\
&\quad\left . +\sup_{x\in\mathbb R^N}|u_2(t,x;u_{01},u_{02})-U_2(x,ct-x\cdot e+s_0)|\right\}=0.
\end{align*}
\end{itemize}
\end{theorem}

The rest of the paper is organized as follows.
In section \ref{Pre}, we provide some preliminary lemmas that will be used in the following section.
In section \ref{AS}, we establish the exact asymptotic behavior of pulsating traveling fronts
near their unstable limiting state.
In section \ref{UNI}, we are devoted to the proof of the uniqueness of pulsating traveling fronts.
Section \ref{STA} focuses on the globally stability of pulsating traveling fronts.

\section{Preliminaries}\label{Pre}

In this section, we give some preliminary lemmas that will be used in the following.

\begin{lemma}\label{Harnack}
Assume (H1)-(H5). Let $\bm{u}(t,x)=U(x,ct-x\cdot e)$ be a pulsating traveling front of \eqref{m-system}.
Then for any fixed $r>0$, there exists $N_r>0$ such that
\begin{equation}\label{Har-ineq}
\mathop{\sup}\limits_{z\in I_{r/4}(s)}U(x,z)\le N_r\mathop{\inf}\limits_{z\in I_{r/4}(s)}U(x^\prime,z),\quad\forall\; x, x^\prime\in\mathbb R^N,\ \forall\;s\in\mathbb R,
\end{equation}
where $I_{r/2}(s):=(s-\frac{r}{2},s+\frac{r}{2})$,
and $N_r>0$ is a constant independent of $U$.
\end{lemma}
\begin{proof}
We only prove for $x$, $x^\prime\in\overline{\mathcal D}$ since $U(\cdot,s)$ is periodic for each $s$. For any fixed $r>0$, let $\gamma_r=\frac{|s|+r+|L|}{|c|}$, then
there exist $\theta_r>0$ and $p_r$, $p_r^\prime\in\mathcal L$ such that
$$2\gamma_r+\theta_r\le\frac{p_r\cdot e}{c}\le2\theta_r,\quad
2\gamma_r+3\theta_r\le\frac{p_r^\prime\cdot e}{c}\le4\theta_r.$$
It is easy to verify that for any $z$, $z^\prime\in I_{r/2}(s)$ and $x$, $x^\prime\in\overline{\mathcal D}$,
$$\gamma_r+\theta_r\le t:=\frac{z+x\cdot e+p_r\cdot e}{c}\le\gamma_r+2\theta_r,\quad
\gamma_r+3\theta_r\le t^\prime:=\frac{z^\prime+x^\prime\cdot e+p_r^\prime\cdot e}{c}
\le\gamma_r+4\theta_r.$$
Let $D=B(O,R)$ be the ball in $\mathbb R^N$ centered at $O$ with radius
$R=|L|+|p_r|+|p_r^\prime|$. Noting that $\bm F(x,\bm 0)=\bm 0$, then
\begin{equation*}
\frac{\partial\bm{u}(t,x)}{\partial t}=D(x)\Delta\bm{u}+q(x)\cdot\nabla\bm{u}
+\left(\int_0^1{D_{\bm u}\bm F(x,s\bm u)ds}\right)\bm u,
\end{equation*}
where $D_{\bm u}\bm F$ is a cooperative matrix.
It then follows from the Harnack type inequalities for cooperative parabolic systems
(see, e.g., \cite[Lemma 3.6]{Foldes2009}) and \eqref{uv-P}
that there exists $N_r>0$ independent of $\bm u$ such that
\begin{align*}
\mathop{\sup}\limits_{z\in I_{r/2}(s),\ x\in\mathbb R^N}U(x,z)
&=\mathop{\sup}\limits_{z\in I_{r/2}(s),\ x\in\overline{\mathcal D}}
\bm u\left(\frac{z+(x+p_r)\cdot e}{c},x+p_r\right)\\
&\le\mathop{\sup}\limits_{(t,x)\in[\gamma_r+\theta_r,\gamma_r+2\theta_r]\times\overline{D}}
\bm u\left(t,x\right)\\
&\le N_r\mathop{\inf}\limits_{(t^\prime,x^\prime)\in[\gamma_r+3\theta_r,\gamma_r+4\theta_r]
\times\overline{D}}\bm u\left(t^\prime,x^\prime\right)\\
&\le N_r\mathop{\inf}\limits_{z^\prime\in I_{r/2}(s),\ x^\prime\in\overline{\mathcal D}}
\bm u\left(\frac{z^\prime+(x^\prime+p_r^\prime)\cdot e}{c},x^\prime+p_r^\prime\right)\\
&=N_r\mathop{\inf}\limits_{z^\prime\in I_{r/2}(s),\ x^\prime\in\overline{\mathcal D}}
U(x^\prime,z^\prime)\\
&=N_r\mathop{\inf}\limits_{z^\prime\in I_{r/2}(s),\ x^\prime\in\mathbb R^N}U(x^\prime,z^\prime)
,\ \ \ \forall\;s\in\mathbb R.
\end{align*}
The proof is complete.
\end{proof}

\begin{lemma}\label{UcpV}
Assume (H1)-(H5). Let
$U(x,ct-x\cdot e)=(U_1(x,ct-x\cdot e),\cdots,U_m(x,ct-x\cdot e))\textcolor{blue}{^T}$
be a pulsating traveling front of \eqref{m-system}.
Then there exists $K_c>0$ such that
$$U_1(x,s)\le K_c\min_{i=2,\cdots,m}\{U_i(x,s)\},\quad\forall\;(x,s)\in\mathbb R^N\times\mathbb R.$$
\end{lemma}
\begin{proof}
Noting that $\mathop{\lim}\limits_{s\to+\infty}U(x,s)=\bm 1$ uniformly in $x\in\mathbb R^N$,
then there exist $K_1>0$ and $M_1>0$ such that for each $i=2,\cdots,m$,
one has $U_1(x,s)\le K_1U_i(x,s)$ for any $(x,s)\in\mathbb R^N\times[M_1,\infty)$.
Since for each $x\in\mathbb R^N$ and $i=2,\cdots,m$, there exists $\alpha>0$ such that
$U_i(x,\cdot)\ge\alpha>0$ on any compact subset of $\mathbb R$,
it suffices to prove that there exist $-M_2<0$ and $\sigma>0$ such that
\begin{equation}\label{toprove}
\inf_{x\in\mathbb R^N}\frac{U_i(x,s)}{U_1(x,s)}>\sigma,
\quad\forall\;s\le-M_2,\quad i=2,\cdots,m.
\end{equation}
We first prove for the case $i=2$, that is
\begin{equation}\label{induction-1}
\inf_{x\in\mathbb R^N}\frac{U_2(x,s)}{U_1(x,s)}>\sigma,\quad\forall\;s\le-M_2.
\end{equation}
Assume to the contrary that there exists a sequence $\{(x_n,s_n)\}_{n\in\mathbb{N}}$ such that
$$s_n\to-\infty\ \ (n\to\infty),\quad\lim_{n\to\infty}\frac{U_2(x_n,s_n)}{U_1(x_n,s_n)}=0.$$
Let
\begin{align*}
u_1^n(t,x)&=\frac{u_1(t+\frac{s_n}{c},x)}{u_1(\frac{s_n+x_n\cdot e}{c},x_n)}
=\frac{U_1(x,ct-x\cdot e+s_n)}{U_1(x_n,s_n)},\\
u_2^n(t,x)&=\frac{u_2(t+\frac{s_n}{c},x)}{u_1(\frac{s_n+x_n\cdot e}{c},x_n)}
=\frac{U_2(x,ct-x\cdot e+s_n)}{U_1(x_n,s_n)}.
\end{align*}
Observe that
\begin{equation*}
u_2^n(t,x)=\frac{U_2(x,ct-x\cdot e+s_n)}{U_2(x_n,s_n)}\cdot\frac{U_2(x_n,s_n)}{U_1(x_n,s_n)}.
\end{equation*}
It then follows from Lemma \ref{Harnack} that $\{u_1^n\}_{n\in\mathbb N}$ and $\{u_2^n\}_{n\in\mathbb N}$ are locally bounded in
$\mathbb R\times\mathbb R^N$, and in particular,
$\mathop{\lim}\limits_{n\to\infty}u_2^n(t,x)=0$ locally uniformly in $\mathbb R\times\mathbb R^N$.
By a direct calculation, we have
\begin{equation*}
\begin{cases}
\frac{\partial u_1^n(t,x)}{\partial t}=d_1(x)\Delta u_1^n+q_1(x)\cdot\nabla u_1^n
+h_1(x,\bm u(t+\frac{s_n}{c},x))u_1^n,\\
\frac{\partial u_2^n(t,x)}{\partial t}=d_2(x)\Delta u_2^n+q_2(x)\cdot\nabla u_2^n
+a_{21}u_1^n+h_2(x,\bm u(t+\frac{s_n}{c},x))u_2^n.
\end{cases}
\end{equation*}
Note that
$$\mathop{\lim}\limits_{n\to\infty}\bm u\left(t+\frac{s_n}{c},x\right)=
\mathop{\lim}\limits_{n\to\infty}U(x,ct-x\cdot e+s_n)=\bm 0
\quad\text{locally uniformly in}\ (t,x)\in\mathbb R\times\mathbb R^N.$$
By the standard parabolic estimates and up to an extraction of subsequence,
$\{(u_1^n,u_2^n)\}_{n\in\mathbb N}$ converges to some $(u_1^\infty,u_2^\infty)\ge(0,0)$ locally uniformly in
$\mathbb R\times\mathbb R^N$, and
\begin{equation}\label{uv-infinity}
\begin{cases}
\frac{\partial u_1^\infty(t,x)}{\partial t}=d_1(x)\Delta u_1^\infty+q_1(x)\cdot\nabla u_1^\infty
+h_1(x,\bm 0)u_1^\infty,\\
\frac{\partial u_2^\infty(t,x)}{\partial t}=d_2(x)\Delta u_2^\infty+q_2(x)\cdot\nabla u_2^\infty
+a_{21}u_1^\infty+h_2(x,\bm 0)u_2^\infty.
\end{cases}
\end{equation}
Since $U(\cdot,s)$ is periodic, we may assume without loss of generality
that $x_n\in\overline{\mathcal D}$ such that $x_n\to x_\infty$ as $n\to\infty$.
Then it is easy to see that $u_1^\infty\left(\frac{x_\infty\cdot e}{c},x_\infty\right)=1$,
and hence $u_1^\infty>0$ in $\mathbb R\times\mathbb R^N$ by the maximum principle.
On the other hand, since $u_2^\infty(t,x)=\mathop{\lim}\limits_{n\to\infty}u_2^n(t,x)=0$,
the second equation in \eqref{uv-infinity} then shows that
$a_{21}u_1^\infty\equiv0$ in any compact set of $\mathbb R\times\mathbb R^N$,
which is a contradiction since $a_{21}(x)>0$ for any $x\in\mathbb{R}^N$ by (H1).
Therefore \eqref{induction-1} holds.

Suppose now that \eqref{toprove} hold for all $i\le k-1$, where $3\le k\le m$.
By (H1), there exists $l\le k-1$ such that $a_{kl}(x)>0$ for any $x\in\mathbb{R}^N$.
Next we prove that
\begin{equation}\label{induction-k-i}
\inf_{x\in\mathbb R^N}\frac{U_k(x,s)}{U_l(x,s)}>\sigma,\quad\forall\;s\le-M_2.
\end{equation}
Assume to the contrary that there exists $\{(y_n,z_n)\}_{n\in\mathbb N}$ such that
$$y_n\to y_\infty\in\overline{\mathcal D},\ \ z_n\to-\infty\ \ (n\to\infty),\quad\lim_{n\to\infty}\frac{U_k(y_n,z_n)}{U_l(y_n,z_n)}=0.$$
Let
\begin{align*}
u_j^n(t,x)=\frac{u_j(t+\frac{z_n}{c},x)}{u_l(\frac{z_n+y_n\cdot e}{c},y_n)}
=\frac{U_j(x,ct-x\cdot e+z_n)}{U_l(y_n,z_n)},\quad j=1,2,\cdots,k.
\end{align*}
Noting that
\begin{equation*}
u_k^n(t,x)=\frac{U_k(x,ct-x\cdot e+z_n)}{U_k(y_n,z_n)}\cdot\frac{U_k(y_n,z_n)}{U_l(y_n,z_n)}
\end{equation*}
and
$\mathop{\lim}\limits_{n\to\infty}u_k^n(t,x)=0$ locally uniformly in $\mathbb R\times\mathbb R^N$.
Moreover, a direct calculation shows that
\begin{equation*}
\frac{\partial u_j^n(t,x)}{\partial t}=d_j(x)\Delta u_j^n+q_j(x)\cdot\nabla u_j^n
+\sum_{p=1}^{j-1}a_{jp}u_p^n+h_j(x,\bm u(t+\frac{z_n}{c},x))u_j^n,\quad j=1,2,\cdots,k.
\end{equation*}
By a similar argument as above, $\{(u_l^n,u_k^n)\}_{n\in\mathbb N}$ converges to some $(u_l^\infty,u_k^\infty)\ge(0,0)$ locally uniformly in
$\mathbb R\times\mathbb R^N$, and it follows from (H1) that
\begin{equation*}
\begin{cases}
\frac{\partial u_l^\infty(t,x)}{\partial t}\ge d_l(x)\Delta u_l^\infty+q_l(x)\cdot\nabla u_l^\infty
+h_l(x,\bm 0)u_l^\infty,\\
\frac{\partial u_k^\infty(t,x)}{\partial t}\ge d_k(x)\Delta u_k^\infty+q_k(x)\cdot\nabla u_k^\infty
+a_{kl}u_l^{\infty}+h_k(x,\bm 0)u_k^\infty.
\end{cases}
\end{equation*}
Notice that $u_l^\infty\left(\frac{y_\infty\cdot e}{c},y_\infty\right)=1$,
and hence $u_l^\infty(t,x)>0$ for any $(t,x)\in\mathbb{R}\times\mathbb{R}^N$  by the maximum principle.
On the other hand, since $u_k^\infty(t,x)=\mathop{\lim}\limits_{n\to\infty}u_k^n(t,x)=0$,
we must have $a_{kl}u_l^\infty\equiv0$ in any compact set of $\mathbb R\times\mathbb R^N$,
which is a contradiction since $a_{kl}(x)>0$ for any $x\in\mathbb{R}^N$.
Therefore \eqref{induction-k-i} holds, and it further follows from the assumption that
\begin{equation*}
\inf_{x\in\mathbb R^N}\frac{U_k(x,s)}{U_1(x,s)}>\sigma,\quad\forall\;s\le-M_2.
\end{equation*}
By using an induction argument, one can prove that \eqref{toprove} hold for all $i=2,3,\cdots,m$. The proof is complete.
\end{proof}

\begin{definition}
\begin{itemize}
\item[(i)] Let $D$ be an open and connected domain in $\mathbb R\times\mathbb R^N$.
A continuous function $\bm u$ is said to be a (regular) supersolution of \eqref{m-system} in $D$,
provided that
\begin{equation*}
\frac{\partial\bm{u}(t,x)}{\partial t}\ge D(x)\Delta\bm{u}+q(x)\cdot\nabla\bm{u}
+\bm{F}(x,\bm{u}),\quad (t,x)\in D.
\end{equation*}
It is called a (regular) subsolution if the above inequality is reversed.

\item[(ii)] A continuous function $\bm u$ is said to be an irregular
supersolution of \eqref{m-system}, if there exist regular supersolutions $\bm u_1$ and
$\bm u_2$ such that $\bm u=\min\{\bm u_1,\bm u_2\}$,
and it is called an irregular subsolution if there exist regular subsolutions
$\bm u_1$ and $\bm u_2$ such that $\bm u=\max\{\bm u_1,\bm u_2\}$.
\end{itemize}
\end{definition}

We give two comparison principles as follows.

\begin{lemma}\label{CP+}
Assume that $\underline{\bm u}(t,x)=\underline{U}(x,ct-x\cdot e)$
is a subsolution of \eqref{m-system} in $\mathbb R\times\mathbb R^N$ such that
$\underline{U}(x,s)$ is periodic in $x$ and $\bm 0\le\underline{\bm u}\ll\bm 1$, and
that $\min\{\bm w(t,x),\bm 1\}:=\overline{\bm u}(t,x)=\overline{U}(x,ct-x\cdot e)$ is
an irregular supersolution of \eqref{m-system} in $\mathbb R\times\mathbb R^N$,
where $\bm w(t,x)=W(x,ct-x\cdot e)$ is a supersolution of \eqref{m-system}
in $\Omega_{\hat{s}}^-$ with some $\hat{s}\le+\infty$, $\bm w>\bm 0$,
$W(x,s)$ is periodic in $x$ and nondecreasing in $s$,
and there is $\bar{\sigma}<\hat{s}$ such that $\overline{\bm u}(t,x)=\bm 1$
for any $(t,x)\in\Omega_{\bar{\sigma}}^+$.
If there exists $\sigma<\bar{\sigma}$ such that
$\underline{U}(x,\sigma)\ll\overline{U}(x,\sigma)$ for all $x\in\mathbb R^N$,
then
$$\underline{U}(x,s)\ll\overline{U}(x,s),\quad\forall\;(x,s)\in\mathbb R^N\times[\sigma,\infty).$$
\end{lemma}
\begin{proof}
Let
\begin{align*}
\delta_i=\inf\left\{\delta\ge0\ |\ \underline{U}_i(x,s-\delta)\le\overline{U}_i(x,s),
\ \forall\;(x,s)\in\mathbb R^N\times[\sigma+\delta,\infty)\right\},\quad i\in I.
\end{align*}
It is easy to see from the assumption that
$\delta_i\in[0,\bar{\sigma}-\sigma)$ for each $i\in I$.
Let $\delta_k=\max_{i\in I}\{\delta_i\}$,
and we next prove that $\delta_k=0$.
Assume to the contrary that $\delta_k>0$, then there exist sequences
$\{\delta_n\}_{n\in\mathbb{N}}$ with $0\le\delta_n\le\delta_k$ and
$\{(x_n,s_n)\}_{n\in\mathbb{N}}$ such that $s_n\ge\sigma+\delta_n$, such that
$$\delta_n\to\delta_k\ (n\to\infty),\quad\underline{U}_k(x_n,s_n-\delta_n)>\overline{U}_k(x_n,s_n),
\quad\lim_{n\to\infty}\{\underline{U}_k(x_n,s_n-\delta_n)-\overline{U}_k(x_n,s_n)\}=0.$$
Then $s_n<\bar{\sigma}$ by the assumption, and thus we may assume up to a subsequence that
$s_n\to s_*\in[\sigma+\delta_k,\bar{\sigma}]$.
Since $\underline{U}(\cdot,s)$ and $\overline{U}(\cdot,s)$ are periodic,
we may assume $x_n\to x_*\in\overline{\mathcal D}$.
Then
$$\underline{U}_k(x_*,s_*-\delta_k)=\overline{U}_k(x_*,s_*)<1\ \ \text{and}
\ \ \underline{U}_i(x_*,s_*-\delta_k)\le\overline{U}_i(x_*,s_*),\ \ \forall\; i\ne k.$$
Moreover, by assumption we have
\begin{equation}\label{eq3}
\underline{U}(x_*,\sigma)\ll\overline{U}(x_*,\sigma)\le\overline{U}(x_*,\sigma+\delta_k).
\end{equation}
Therefore, $s_*\in(\sigma+\delta_k,\bar{\sigma})$.
Let
\begin{align*}
\tilde{u}_i(t,x)=\underline{u}_i\left(t-\frac{\delta_k}{c},x\right)-\overline{u}_i(t,x)
=\underline{U}_i(x,ct-x\cdot e-\delta_k)-\overline{U}_i(x,ct-x\cdot e),\quad i\in I.
\end{align*}
Then for each $i$, there hold $\tilde{u}_i(t,x)\le0$ for any
$(t,x)\in\Omega_{\sigma+\delta_k}^+$, and
$\tilde{u}_k(t_*,x_*)=0$ and $\tilde{u}_i(t_*,x_*)\le0$ for each $i\ne k$, where $t_*=\frac{s_*+x_*\cdot e}{c}$.
Noting that
\begin{align*}
\frac{\partial\tilde{u}_k(t,x)}{\partial t}-d_k(x)\Delta\tilde{u}_k-q_k(x)\cdot\nabla\tilde{u}_k
&\le f_k\left(x,\underline{\bm u}\left(t-\frac{\delta_k}{c},x\right)\right)
-f_k(x,\overline{\bm u}(t,x))\\
&\le\left(\int_0^1{\frac{\partial f_k}{\partial u_k}
\left(x,\tau\underline{\bm u}\left(t-\frac{\delta_k}{c},x\right)
+(1-\tau)\overline{\bm u}(t,x)\right)d\tau}\right)\tilde{u}_k.
\end{align*}
It then follows from the maximum principle that
\begin{equation}\label{uv-set}
\overline{u}_k(t,x)=\underline{u}_k\left(t-\frac{\delta_k}{c},x\right),
\quad\forall\; (t,x)\in\Omega_*,
\end{equation}
where $\Omega_*$ is a connected subset of
$\Omega_{\sigma+\delta_k}^+\cap\{t\le t_*\}\cap\{\overline{u}_k<1\}$ containing $(t_*,x_*)$.

Now if $c>0$, let
$$\hat{t}=\frac{\sigma+\delta_k+x_*\cdot e}{c},$$
then $\hat{t}<t_*$ since $\sigma+\delta_k<s_*$, and
$\overline{u}_k(\hat{t},x_*)\le\overline{u}_k(t_*,x_*)<1$ since $\overline{U}(x,s)$ is
nondecreasing in $s$.
Hence $\overline{U}_k(x_*,\sigma+\delta_k)=\underline{U}_k(x_*,\sigma)$ by \eqref{uv-set},
which contradicts \eqref{eq3}.

If $c<0$, then $(t,x_*)\in\Omega_{\sigma+\delta_k}^+$ for all $t\le t_*$. Note that $\overline{u}_k(t_*,x_*)=\underline{u}_k\left(t_*-\frac{\delta_k}{c},x_*\right)<1$ by \eqref{uv-set}, then there exists $t_0>0$ such that $\overline{u}_k(t,x_*)<1$
for any $t_*-t_0\le t\le t_*$. Let
$$\underline{t}=\inf\{ t^\prime\le t_*\ |\ \overline{u}_k(t,x_*)<1,
\ \forall\; t^\prime\le t\le t_*\},$$
then $-\infty\le\underline{t}\le t_*-t_0<t_*$.
If $\underline{t}>-\infty$ is a real number,
then $\overline{u}_k(\underline{t},x_*)=\underline{u}_k\left(\underline{t}-\frac{\delta_1}{c},x_*\right)<1$
by \eqref{uv-set}, which contradicts the definition of $\underline{t}$.
Hence $\underline{t}=-\infty$, and then $\overline{u}_k(t,x_*)<1$ for all $t\le t_*$,
which further yields that
$$\overline{u}_k(t,x_*)=\underline{u}_k\left(t-\frac{\delta_k}{c},x_*\right),
\quad\forall\;t\le t_*.$$
Since $\underline{u}_k<1$ and $\overline{u}_k(t,x_*)=\overline{U}_k(x_*,ct-x_*\cdot e)=1$
for all $t\le\frac{\bar{\sigma}+x_*\cdot e}{c}$, we reach a contradiction.

As a result, $\delta_i=0$ for each $i$, and hence $\underline{U}(x,s)\le\overline{U}(x,s)$
for any $(x,s)\in\mathbb R^N\times[\sigma,\infty)$.
Moreover, if there exist $i$ and $(x_1,s_1)\in\mathbb R^N\times[\sigma,\infty)$ such that
$\underline{U}_i(x_1,s_1)=\overline{U}_i(x_1,s_1)$, then $s_1>\sigma$.
By setting $\delta_i=0$ and following similar arguments as above,
we obtain a contradiction. Therefore $\underline{U}(x,s)\ll\overline{U}(x,s)$
for any $(x,s)\in\mathbb R^N\times[\sigma,\infty)$. The proof is complete.
\end{proof}

\begin{lemma}\label{CP-+}
Assume that $\overline{\bm u}(t,x)=\overline{U}(x,ct-x\cdot e)$
is a supersolution of \eqref{m-system} in $\mathbb R\times\mathbb R^N$ such that
$\bm 0<\overline{\bm u}\le\bm 1$, $\overline{U}(x,s)$ is periodic in $x$ and nondecreasing in $s$,
and $\mathop{\liminf}\limits_{s\to+\infty}\inf_{x\in\mathbb R^N}\overline{U}(x,s)=\bm 1$, and
that $\max\{\bm w(t,x),\bm 0\}:=\underline{\bm u}(t,x)=\underline{U}(x,ct-x\cdot e)$ is
an irregular subsolution of \eqref{m-system} in $\mathbb R\times\mathbb R^N$,
where $\bm w(t,x)=W(x,ct-x\cdot e)$ is a subsolution of \eqref{m-system} in
$\Omega_{s_0}^-$ with some $s_0\in\mathbb R$, $W(x,s)$ is periodic in $x$, and
$$\mathop{\sup}\limits_{(x,s)\in\mathbb R^N\times(-\infty,s_0]}W(x,s)\ll\bm 1,
\quad \sup_{x\in\mathbb R^N}W(x,s_0)\le\bm 0.$$
If there exists $\sigma<s_0$ such that $\underline{U}(x,\sigma)\ll\overline{U}(x,\sigma)$
for all $x\in\mathbb R^N$, then
$$\underline{U}(x,s)\ll\overline{U}(x,s),\quad\forall\;(x,s)\in\mathbb R^N\times[\sigma,s_0].$$
\end{lemma}
\begin{proof}
Let
\begin{align*}
\theta_i:=\inf\{\theta\ge0\ |\ \underline{U}_i(x,s)\le\overline{U}_i(x,s+\theta),
\ \forall\;(x,s)\in\mathbb R^N\times[\sigma,s_0]\},\quad\ i\in I.
\end{align*}
Noting from the assumptions that
$$\mathop{\liminf}\limits_{s\to+\infty}\inf_{x\in\mathbb R^N}\overline{U}(x,s)=\bm 1\gg
\mathop{\sup}\limits_{(x,s)\in\mathbb R^N\times(-\infty,s_0]}\underline{U}(x,s).$$
Hence $\theta_i\ge 0$ is well defined for each $i$.
Let $\theta_k=\max_{i\in I}\{\theta_i\}$,
it suffices to prove that $\theta_k=0$. In fact, if $\theta_k>0$, then there exists
$(x_*,s_*)\in\mathbb R^N\times[\sigma,s_0)$ such that
$$\underline{U}_k(x_*,s_*)=\overline{U}_k(x_*,s_*+\theta_k)\ \ \text{and}\ \
\underline{U}_i(x_*,s_*)\le\overline{U}_i(x_*,s_*+\theta_k),\ \ \forall\;i\ne k.$$
Noting that
$\underline{U}_k(x,\sigma)<\overline{U}_k(x,\sigma)\le\overline{U}_k(x,\sigma+\theta_k)$
for any $x\in\mathbb R^N$, then $s_*\in(\sigma,s_0)$.
Let
\begin{align*}
\hat{u}_i(t,x)=\overline{u}_i\left(t+\frac{\theta_k}{c},x\right)-\underline{u}_i(t,x)
=\overline{U}_i(x,ct-x\cdot e+\theta_k)-\underline{U}_i(x,ct-x\cdot e),\quad i\in I,
\end{align*}
then $\hat{u}_i(t,x)\ge0$ for each $i$ and for any $(t,x)\in\Omega_{\sigma}^{s_0}$,
and in particular $\hat{u}_k(t_*,x_*)=0$, where $t_*:=\frac{s_*+x_*\cdot e}{c}$.
By a direct calculation, we have
\begin{align*}
\frac{\partial\hat{u}_k(t,x)}{\partial t}-d_k(x)\Delta\hat{u}_k-q_k(x)\cdot\nabla\hat{u}_k
\ge\left(\int_0^1{\frac{\partial f_k}{\partial u_k}
\left(x,\tau\overline{\bm u}(t+\frac{\theta_k}{c},x)
+(1-\tau)\underline{\bm u}(t,x)\right)d\tau}\right)\hat{u}_k.
\end{align*}
The maximum principle then yields that
$\hat{u}_k\equiv0$ for all $(t,x)\in\Omega_*$, where
$\Omega_*$ is a connected subset of
$\Omega^{s_0}_{\sigma}\cap\{t\le t_*\}\cap\{\underline{u}_k>0\}$ containing $(t_*,x_*)$.
By using similar arguments to the proof of Lemma \ref{CP+},
we have $\underline{U}(x,s)\ll\overline{U}(x,s)$
for all $(x,s)\in\mathbb R^N\times[\sigma,s_0]$. The proof is complete.
\end{proof}


\section{Asymptotic behavior near the unstable limiting state}\label{AS}

In this section, we investigate the asymptotic behavior of pulsating traveling fronts
$U(x,ct-x\cdot e)$ as $ct-x\cdot e\to-\infty$, in the case $c>c_+^0(e)$ and the critical case $c=c_+^0(e)$, respectively.

\subsection{The super-critical case}

We consider the super-critical case in this subsection, that is, $c>c_+^0$, where
$$c_+^0=\inf_{\lambda>0}\frac{\kappa_1(\lambda)}{\lambda}
=\frac{\kappa_1(\lambda_+^0)}{\lambda_+^0}$$
is defined by \eqref{minimal-ss}, and
$\lambda_c=\min\{\lambda>0: \kappa_1(\lambda)-c\lambda=0\}$
is given by \eqref{c1}, with
$$\kappa_1(\lambda)=\kappa_{e}(d_1,q_1,\zeta^1,\lambda).$$
For any $c>c_+^0$, let
\begin{equation}\label{eps}
0<\epsilon<\min\left\{\frac{\lambda_+^0-\lambda_c}{2},\frac{\lambda_c}{2}\right\}.
\end{equation}
It is easy to see that
$$\sigma_\epsilon:=\kappa_1(\lambda_c+\epsilon)-c(\lambda_c+\epsilon)<0.$$

Let
$$\bm{\Phi}_{\lambda_c}(x)=\left(\phi_1^c(x),\phi_2^c(x),\cdots,\phi_m^c(x)\right)
\quad \ \text{and}\quad \ \bm{\Phi}_{\lambda_c+\epsilon}(x)
=\left(\phi_1^{\epsilon}(x),\phi_2^{\epsilon}(x),\cdots,\phi_m^{\epsilon}(x)\right)$$
be positive periodic eigenfunctions of problem \eqref{u1um-linear-PE} with $\lambda=\lambda_c$
and $\lambda=\lambda_c+\epsilon$ associated with principal eigenvalues
$\kappa_1(\lambda_c)$ and $\kappa_1(\lambda_c+\epsilon)$, respectively.
Denote
$$M_c=\max_{i\in I}\left\{\max_{x\in\mathbb R^N}\phi_i^c(x)\right\},\quad
m_c=\min_{i\in I}\left\{\min_{x\in\mathbb R^N}\phi_i^c(x)\right\},
\quad\theta_c=\frac{M_c}{m_c},$$
and
$$M_{\epsilon}=\max_{i\in I}\left\{\max_{x\in\mathbb R^N}\phi_i^{\epsilon}(x)\right\},\quad
m_{\epsilon}=\min_{i\in I}\left\{\min_{x\in\mathbb R^N}\phi_i^{\epsilon}(x)\right\},
\quad\theta_{\epsilon}=\frac{M_{\epsilon}}{m_{\epsilon}}.$$

\begin{lemma}\label{sub}
Assume (H1)-(H6).
If $c>c_+^0$, then there exists $s_*\in\mathbb R$ such that for any $0<\delta_2\le\delta_1$
and $s_0=s_0(\delta_1)\le s_*$ sufficiently small,
there exists $n_0=n_0(\delta_1)>0$ such that the function
$\underline{\bm u}(t,x)=(\underline{u}_1(t,x),\underline{u}_2(t,x),\cdots,\underline{u}_m(t,x))\textcolor{blue}{^T}$
defined by
\begin{align*}
\underline{u}_1(t,x)&=\underline{U}_1(x,ct-x\cdot e)=\delta_1e^{\lambda_c(ct-x\cdot e)}
\left(\phi_1^c(x)-n_0e^{\epsilon(ct-x\cdot e)}\phi_1^{\epsilon}(x)\right),\\
\underline{u}_i(t,x)&=\underline{U}_i(x,ct-x\cdot e)=\delta_2e^{\lambda_c(ct-x\cdot e)}
\left(\phi_i^c(x)-\frac{n_0\delta_1}{\delta_2}e^{\epsilon(ct-x\cdot e)}\phi_i^{\epsilon}(x)\right),
\ \ i=2,3,\cdots,m
\end{align*}
is a subsolution of \eqref{m-system} for
$(t,x)\in\Omega_{s_0}^-=\{\mathbb R\times\mathbb R^N: ct-x\cdot e\le s_0\}$,
where $\epsilon>0$ is given by \eqref{eps}. Moreover,
$\underline{U}(x,s)=(\underline{U}_1(x,s),\underline{U}_2(x,s),\cdots,\underline{U}_m(x,s))$
satisfies
$$\mathop{\sup}\limits_{(x,s)\in\mathbb R^N\times(-\infty,s_0]}
\underline{U}(x,s)\ll\bm 1,\qquad\sup_{x\in\mathbb R^N}\underline{U}(x,s_0)\le\bm 0.$$
\end{lemma}
\begin{proof}
Let
\begin{equation}\label{s*}
s_*=\min\left\{\frac{1}{\lambda_c-\epsilon}\ln\frac{|\sigma_\epsilon|m_{\epsilon}}
{\gamma_0(1+\theta_{\epsilon})^2(M_c+M_{\epsilon})(|\bm{\Phi}_{\lambda_c}|
+|\bm{\Phi}_{\lambda_c+\epsilon}|)},\ -1\right\},
\end{equation}
where
$$\gamma_0:=\max_{i,j\in I}
\left\{\max_{(x,\bm u)\in\mathbb R^N\times[\bm{-\theta},\bm{\theta}]}
\left|\frac{\partial h_i(x,\bm u)}{\partial u_j}\right|\right\},\quad\ \bm{\theta}=m\theta_{\epsilon}\bm 1.$$
Let $s_0\le s_*$ be such that
\begin{equation}\label{s-0}
e^{-\lambda_c s_0}\ge\delta_1M_c\quad\text{and}\quad
n_0:=\theta_{\epsilon}e^{-\epsilon s_0}\ge\delta_1.
\end{equation}
Noting that $n_0e^{\epsilon s}\le\theta_{\epsilon}$ for any $s\le s_0$, and $n_0\ge\delta_1$,
a direct calculation shows that
\begin{align*}
\mathcal N_1(x,\underline{\bm u}):&=\frac{\partial\underline{u}_1(t,x)}{\partial t}-
d_1(x)\Delta\underline{u}_1-q_1(x)\cdot\nabla\underline{u}_1-f_1(x,\underline{\bm u})\\
&=(h_1(x,\bm 0)-h_1(x,\underline{\bm u}))\underline{u}_1
-|\sigma_{\epsilon}|n_0\delta_1e^{(\lambda_c+\epsilon)s}\phi_1^{\epsilon}\\
&\le\gamma_0|\underline{\bm u}||\underline{u}_1|
-|\sigma_{\epsilon}|n_0\delta_1e^{(\lambda_c+\epsilon)s}\phi_1^{\epsilon}\\
&\le\gamma_0\delta_1^2e^{2\lambda_cs}\sum_{k=1}^m|
\phi_k^c-n_0e^{\epsilon s}\phi_k^{\epsilon}|
|\phi_1^c-n_0e^{\epsilon s}\phi_1^{\epsilon}|
-|\sigma_{\epsilon}|n_0\delta_1e^{(\lambda_c+\epsilon)s}\phi_1^{\epsilon}\\
&\le\gamma_0\delta_1^2e^{2\lambda_cs}
(M_c+M_{\epsilon})(|\bm{\Phi}_{\lambda_c}|+|\bm{\Phi}_{\lambda_c+\epsilon}|)
\left(1+n_0e^{\epsilon s}\right)^2
-|\sigma_{\epsilon}|n_0\delta_1e^{(\lambda_c+\epsilon)s}m_{\epsilon}\\
&\le n_0\delta_1e^{(\lambda_c+\epsilon)s}
\left\{\gamma_0(1+\theta_{\epsilon})^2(M_c+M_{\epsilon})(|\bm{\Phi}_{\lambda_c}|+|\bm{\Phi}_{\lambda_c+\epsilon}|)
e^{(\lambda_c-\epsilon)s}-|\sigma_{\epsilon}|m_{\epsilon}\right\}\\
&\le0,
\end{align*}
and similarly,
\begin{align*}
\mathcal N_i(x,\underline{\bm u}):&=\frac{\partial\underline{u}_i(t,x)}{\partial t}-
d_i(x)\Delta\underline{u}_i-q_i(x)\cdot\nabla\underline{u}_i-f_i(x,\underline{\bm u})\\
&={a_{i1}(\delta_2-\delta_1)e^{\lambda_cs}\phi_1^c}
+(h_i(x,\bm 0)-h_i(x,\underline{\bm u}))\underline{u}_i
-|\sigma_{\epsilon}|n_0\delta_1e^{(\lambda_c+\epsilon)s}\phi_i^{\epsilon}\\
&\le(h_i(x,\bm 0)-h_i(x,\underline{\bm u}))\underline{u}_i
-|\sigma_{\epsilon}|n_0\delta_1e^{(\lambda_c+\epsilon)s}\phi_i^{\epsilon}\\
&\le\gamma_0\delta_1^2e^{2\lambda_cs}
(M_c+M_{\epsilon})(|\bm{\Phi}_{\lambda_c}|+|\bm{\Phi}_{\lambda_c+\epsilon}|)
\left(1+n_0e^{\epsilon s}\right)^2
-|\sigma_{\epsilon}|n_0\delta_1e^{(\lambda_c+\epsilon)s}m_{\epsilon}\\
&\le0,\quad i=2,3,\cdots,m.
\end{align*}
Therefore $\underline{\bm u}(t,x)$ is a subsolution of \eqref{m-system} in
$\Omega_{s_0}^-$.
Furthermore, it is easy to see from \eqref{s-0} that
$$\mathop{\sup}\limits_{(x,s)\in\mathbb R^N\times(-\infty,s_0]}\underline{U}(x,s)
\ll\mathop{\sup}\limits_{x\in\mathbb R^N}\delta_1e^{\lambda_cs_0}M_c\bm 1\le\bm 1,$$
and $\underline{U}(x,s_0)\le\bm 0$ for all $x\in\mathbb R^N$.
The proof is complete.
\end{proof}

\begin{lemma}\label{super}
Assume (H1)-(H7). If $c>c_+^0$, then
for any constant $k>0$, the function $\overline{\bm u}(t,x)=\min\{\bm w_c(t,x),\bm 1\}$ is an
irregular supersolution of \eqref{m-system} in $\mathbb R\times\mathbb R^N$, where
$$\bm{w}_c(t,x)=W(x,ct-x\cdot e)=ke^{\lambda_c(ct-x\cdot e)}\bm{\Phi}_{\lambda_c}(x),
\quad\forall\;(t,x)\in\mathbb R\times\mathbb R^N.$$
\end{lemma}
\begin{proof}
It suffices to prove that $\bm w_c$ is a supersolution of \eqref{m-system}
since $\bm 1$ is a (super)solution of \eqref{m-system}.
By a direct calculation and in view of (H7), we have
\begin{align*}
\frac{\partial\bm{w}_c(t,x)}{\partial t}
&=D(x)\Delta\bm{w}_c+q(x)\cdot\nabla\bm{w}_c+D_{\bm u}\bm F(x,\bm{0})\bm{w}_c\\
&\ge D(x)\Delta\bm{w}_c+q(x)\cdot\nabla\bm{w}_c+\bm F(x,\bm{w}_c).
\end{align*}
The proof is complete.
\end{proof}

\begin{lemma}\label{low-upp}
Assume (H1)-(H7). Let
$\bm u(t,x)=U(x,ct-x\cdot e)=(U_1(x,ct-x\cdot e),\cdots,U_m(x,ct-x\cdot e))\textcolor{blue}{^T}$
be a pulsating traveling front of \eqref{m-system} with $c>c_+^0$, then
\begin{equation*}
\mathop{\limsup}\limits_{s\to-\infty}\left\{\sup_{x\in\mathbb R^N}
\frac{U_1(x,s)}{e^{\lambda_c s}\phi_1^c(x)}\right\}<+\infty,
\qquad\mathop{\liminf}\limits_{s\to-\infty}\left\{\inf_{x\in\mathbb R^N}
\frac{U_1(x,s)}{e^{\lambda_c s}\phi_1^c(x)}\right\}>0.
\end{equation*}
\end{lemma}
\begin{proof}
We divide the proof into three steps.

{\bf Step 1}. We prove that
\begin{equation}\label{limsup}
\mathop{\limsup}\limits_{s\to-\infty}\left\{\sup_{x\in\mathbb R^N}
\frac{U_1(x,s)}{e^{\lambda_c s}\phi_1^c(x)}\right\}<+\infty.
\end{equation}
If this is not true, then there exists a sequence $\{(x_n,s_n)\}_{n\in\mathbb{N}}$ such that
\begin{equation}\label{+oo}
s_n\to-\infty\ (n\to\infty),
\quad\mathop{\lim}\limits_{n\to\infty}\frac{U_1(x_n,s_n)}{e^{\lambda_c s_n}\phi_1^c(x_n)}=\infty.
\end{equation}
For any fixed $\delta_1>0$, let $s_0\le s_*$ and $n_0$ be fixed constants satisfying \eqref{s-0},
where $s_*$ is given by \eqref{s*}. Let $0<\sigma<\min\left\{1,\frac{1}{\theta_cK_c}\right\}$,
where $K_c$ is given by Lemma \ref{UcpV}.
Define $\underline{\bm u}(t,x)=\underline{U}(x,ct-x\cdot e)$ with
$\underline{u}_i(t,x)=\underline{U}_i(x,ct-x\cdot e)$ given by
\begin{align*}
\underline{u}_1(t,x)&=\underline{U}_1(x,ct-x\cdot e)=\delta_1e^{\lambda_c(ct-x\cdot e)}
\left(\phi_1^c(x)-n_0e^{\epsilon(ct-x\cdot e)}\phi_1^{\epsilon}(x)\right),\\
\underline{u}_i(t,x)&=\underline{U}_i(x,ct-x\cdot e)=\delta_1e^{\lambda_c(ct-x\cdot e)}
\left(\sigma\phi_i^c(x)-n_0e^{\epsilon(ct-x\cdot e)}\phi_i^{\epsilon}(x)\right),\ \
i=2,3,\cdots,m,
\end{align*}
where $(t,x)\in\Omega_{s_0}^-$.
Noting that $\lim_{s\to-\infty}U_1(x,s)=0$ uniformly in $x\in\mathbb R^N$, and $\underline{U}_1(x,s)>0$ for all
$(x,s)\in\Sigma_{\hat{s}_0}^-:=\{\mathbb R^N\times\mathbb R:\ s\le\hat{s}_0\}$
with some $\hat{s}_0\le s_0$, then there exist $(x_1,s_1)\in\Sigma_{\hat{s}_0}^-$ and $z_1\le0$ such that $U_1(x_1,s_1+z_1)\le\underline{U}_1(x_1,s_1)$.
Assume without loss of generality that $z_1=0$.
It then follows from \eqref{+oo} that there exists $n^*\in\mathbb N$ such that
$$\forall\; n\ge n^*,\quad s_n<s_1,\quad U_1(x_n,s_n)
\ge N_r\delta_1\theta_ce^{\lambda_c s_n}\phi_1^c(x_n),$$
where $N_r$ is given by \eqref{Har-ineq}. It then follows that
$$U_1(x,s_{n^*})\ge\frac{1}{N_r}U_1(x_{n^*},s_{n^*})
\ge\delta_1e^{\lambda_c s_{n^*}}\phi_1^c(x)>\underline{U}_1(x,s_{n^*}),
\quad\forall\; x\in\mathbb R^N.$$
On the other hand, by Lemma \ref{UcpV},
$$U_i(x,s_{n^*})\ge\frac{1}{K_c}U_1(x,s_{n^*})
\ge\sigma\delta_1e^{\lambda_c s_{n^*}}\phi_i^c(x)
>\underline{U}_i(x,s_{n^*}),\quad\forall\;x\in\mathbb R^N,\ \forall\; i=2,3,\cdots,m.$$
It then follows from Lemma \ref{CP-+} that $\underline{U}(x,s)\ll U(x,s)$ for any
$(x,s)\in\mathbb R^N\times[s_{n^*},s_0]$, which contradicts
$U_1(x_1,s_1)\le\underline{U}_1(x_1,s_1)$.
Therefore \eqref{limsup} holds.

{\bf Step 2}. We prove that there exists $B_c>0$ such that
\begin{equation}\label{Ui-induction}
U_i(x,s)\le B_ce^{\lambda_c s},\ \ \ \forall\;(x,s)\in\mathbb R^N\times\mathbb R,
\ \ i\in I.
\end{equation}
By \eqref{limsup}, there exists $B_c>0$ such that
$U_1(x,s)\le B_ce^{\lambda_c s}$ for any $(x,s)\in\mathbb R^N\times\mathbb R$.
Let $\psi_i^c(x)>0$ be the periodic eigenfunction associated with
$\kappa_i(\lambda_c)=\kappa_{e}(d_i,q_i,\zeta^i,\lambda_c)$, that is,
\begin{equation*}\label{kappa-i-psi-i}
\kappa_i(\lambda_c)\psi_i^c=d_i(x)\Delta\psi_i^c+(q_i-2d_i\lambda_c e)
\cdot\nabla\psi_i^c+(d_i\lambda_c^2-\lambda_c q_i\cdot e+h_i(x,\bm 0))\psi_i^c,\ \ i\in I.
\end{equation*}
Noting from \eqref{1-j} that
$$\sigma_i:=c\lambda_c-\kappa_{e}(d_i,q_i,\zeta^i,\lambda_c)=
\kappa_1(\lambda_c)-\kappa_i(\lambda_c)>0,\quad i=2,3,\cdots,m.$$
Let
$$0<\varepsilon<\frac{\min_{i=2,\cdots,m}\{\sigma_i,\
\min_{x\in\mathbb R^N}{|h_i(x,\bm 0)|}\}}{2}.$$
Since $\mathop{\lim}\limits_{ct-x\cdot e\to-\infty}\bm u(t,x)=\bm 0$,
there exists $Z_\varepsilon>0$ such that
$$|h_i(x,\bm u)-h_i(x,\bm 0)|\le\varepsilon,\quad\forall\;(t,x)\in\Omega_{-Z_\varepsilon}^-,
\quad \forall\; i=2,3,\cdots,m.$$
Define
$$\omega_i(t,x)=K_ie^{\lambda_c(ct-x\cdot e)}\psi_i^c(x),\quad i=2,3,\cdots,m,$$
where
$$K_i\ge\frac{2B_c\max_x{\left(\sum_{j=1}^{i-1}a_{ij}(x)\right)}}
{\min_{i=2,\cdots,m}\{\sigma_i\}\min_{i=2,\cdots,m}\{\min_x{\psi_i^c(x)}\}}$$
is such that
$$K_ie^{\lambda_c(-Z_\varepsilon)}\psi_i^c(x)\ge U_i(x,-Z_\varepsilon),
\quad \forall\; x\in\mathbb R^N.$$
Next we prove that \eqref{Ui-induction} holds for $i=2$. Noting that
\begin{align*}
\frac{\partial \omega_2(t,x)}{\partial t}-d_2\Delta\omega_2-q_2\cdot\nabla\omega_2
=(\sigma_2+h_2(x,\bm 0))\omega_2
&\ge\frac{\sigma_2}{2}\omega_2+(h_2(x,\bm 0)+\varepsilon)\omega_2\\
&\ge a_{21}u_1+(h_2(x,\bm 0)+\varepsilon)\omega_2,
\end{align*}
and
\begin{align*}
\frac{\partial u_2(t,x)}{\partial t}-d_2\Delta u_2-q_2\cdot\nabla u_2
&=a_{21}u_1+h_2(x,\bm 0)u_2+(h_2(x,\bm u)-h_2(x,\bm 0))u_2\\
&\le a_{21}u_1+(h_2(x,\bm 0)+\varepsilon)u_2
\end{align*}
for all $(t,x)\in\Omega_{-Z_\varepsilon}^-$.
Hence, the function $(\omega_2-u_2)$ satisfies
\begin{equation*}
\begin{cases}
\frac{\partial (\omega_2-u_2)(t,x)}{\partial t}-d_2\Delta(\omega_2-u_2)-q_2\cdot\nabla(\omega_2-u_2)
\ge(h_2(x,\bm 0)+\varepsilon)(\omega_2-u_2),\quad (t,x)\in\Omega_{-Z_\varepsilon}^-,\\
(\omega_2-u_2)(t,x)\ge0,\quad (t,x)\in\{\mathbb R\times\mathbb R^N: ct-x\cdot e=-Z_\varepsilon\},\\
\mathop{\lim}\limits_{ct-x\cdot e\to-\infty}(\omega_2-u_2)(t,x)=0.
\end{cases}
\end{equation*}
Since $h_2(x,\bm 0)+\varepsilon<0$ for all $x\in\mathbb R^N$,
we conclude from the maximum principle that
$u_2(t,x)\le K_2e^{\lambda_c(ct-x\cdot e)}\psi_2^c(x)$ for any $(t,x)\in\Omega_{-Z_\varepsilon}^-$.
That is, $U_2(x,s)\le K_2e^{\lambda_cs}\psi_2^c(x)$ for any
$(x,s)\in\mathbb R^N\times(-\infty,-Z_\varepsilon]$.
Due to the boundedness of $U_2(x,s)$ in $\mathbb R^N\times\mathbb R$, there exists
$B_c$ large enough such that $U_2(x,s)\le B_ce^{\lambda_cs}$ for any
$(x,s)\in\mathbb R^N\times\mathbb R$.

Suppose now that \eqref{Ui-induction} hold for all $i\le k-1$,
where $3\le k\le m$, that is,
$U_i(x,s)\le B_ce^{\lambda_c s}$ for all $i=1,2,\cdots,k-1$ in $\mathbb R^N\times\mathbb R$.
We next prove that
\begin{equation}\label{U-k-indu}
U_k(x,s)\le B_ce^{\lambda_c s},\quad\forall\;(x,s)\in\mathbb R^N\times\mathbb R.
\end{equation}
Noting that
\begin{align*}
\frac{\partial \omega_k(t,x)}{\partial t}-d_k\Delta\omega_k-q_k\cdot\nabla\omega_k
&=(\sigma_k+h_k(x,\bm 0))\omega_k
\ge\frac{\sigma_k}{2}\omega_k+(h_k(x,\bm 0)+\varepsilon)\omega_k\\
&\ge\sum_{j=1}^{k-1}a_{kj}u_j+(h_k(x,\bm 0)+\varepsilon)\omega_k,\\
\frac{\partial u_k(t,x)}{\partial t}-d_k\Delta u_k-q_k\cdot\nabla u_k
&=\sum_{j=1}^{k-1}a_{kj}u_j+h_k(x,\bm 0)u_k+(h_k(x,\bm u)-h_k(x,\bm 0))u_k\\
&\le\sum_{j=1}^{k-1}a_{kj}u_j+(h_k(x,\bm 0)+\varepsilon)u_k
\end{align*}
for all $(t,x)\in\Omega_{-Z_\varepsilon}^-$, and
$h_k(x,\bm 0)+\varepsilon<0$ for any $x\in\mathbb R^N$,
similar arguments as above show that \eqref{U-k-indu} holds.
By using an induction argument, one can prove that \eqref{Ui-induction} hold for all $i\in I$.

{\bf Step 3}. We prove that
$$\mathop{\liminf}\limits_{s\to-\infty}\left\{\inf_{x\in\mathbb R^N}\frac{U_1(x,s)}
{e^{\lambda_c s}\phi_1^c(x)}\right\}>0.$$
If this is not true,
then there exists $\{(y_n,z_n)\}_{n\in\mathbb{N}}$ with $y_n\in\overline{\mathcal D}$ such that
\begin{equation*}
z_n\to-\infty,\quad y_n\to y^* \quad (n\to\infty),
\quad\mathop{\lim}\limits_{n\to\infty}\frac{U_1(y_n,z_n)}{e^{\lambda_c z_n}\phi_1^c(y_n)}=0.
\end{equation*}
Let
$$u_i^n(t,x)=U_i^n(x,ct-x\cdot e):=\frac{U_i(x,ct-x\cdot e+z_n)}
{e^{\lambda_c(ct-x\cdot e+z_n)}\psi_i^c(x)}
=\frac{u_i(t+\frac{z_n}{c},x)}{e^{\lambda_c(ct-x\cdot e+z_n)}\psi_i^c(x)},
\ \ i\in I.$$
It follows from Step 2 that $\{u_i^n\}_{n\in\mathbb{N}}$ is uniformly bounded for each $i$,
and in particular,
$$\mathop{\lim}\limits_{n\to\infty}u_1^n\left(\frac{y_n\cdot e}{c},y_n\right)
=\mathop{\lim}\limits_{n\to\infty}\frac{U_1(y_n,z_n)}{e^{\lambda_cz_n}\psi_1^c(y_n)}
=\mathop{\lim}\limits_{n\to\infty}\frac{U_1(y_n,z_n)}
{e^{\lambda_cz_n}\phi_1^c(y_n)}\frac{\phi_1^c(y_n)}{\psi_1^c(y_n)}=0.$$
By a direct calculation, we have
\begin{equation*}
\begin{cases}
\frac{\partial u_i^n(t,x)}{\partial t}=d_i\Delta u_i^n
+\left(q_i+2d_i\left(\frac{\nabla\psi_i^c}{\psi_i^c}-\lambda_ce\right)\right)
\cdot\nabla u_i^n-\sigma_iu_i^n-h_i(x,\bm 0)u_i^n
+\frac{f_i(x,\bm u(t+\frac{z_n}{c},x))}{u_i(t+\frac{z_n}{c},x)}u_i^n,\\
u_i^n(t,x)=u_i^n\left(t+\frac{p\cdot e}{c},x+p\right),
\quad\forall\;(t,x)\in\mathbb R\times\mathbb R^N,\ \forall\;p\in \mathcal L,\ \
i\in I,
\end{cases}
\end{equation*}
where $\sigma_1=0$. Noting that
$$\lim_{n\to\infty}\frac{f_1(x,\bm u(t+\frac{z_n}{c},x))}{u_1(t+\frac{z_n}{c},x)}
=\lim_{n\to\infty}h_1(x,\bm u(t+\frac{z_n}{c},x))=h_1(x,\bm 0)$$
locally uniformly in $(t,x)\in\mathbb R\times\mathbb R^N$,
by the parabolic estimates and up to a subsequence, $\{u_1^n\}_{n\in\mathbb{N}}$ converges in
$C^{1,2}_{loc}(\mathbb R\times\mathbb R^N)$ to a function $u_1^*\ge0$,
which satisfies
\begin{equation*}
\begin{cases}
\frac{\partial u_1^*(t,x)}{\partial t}=d_1\Delta u_1^*
+\left(q_1+2d_1\left(\frac{\nabla\psi_1^c}{\psi_1^c}-\lambda_ce\right)\right)\cdot\nabla u_1^*,\\
u_1^*(t,x)=u_1^*\left(t+\frac{p\cdot e}{c},x+p\right),
\quad\forall\;(t,x)\in\mathbb R\times\mathbb R^N,\ \forall\;p\in \mathcal L.
\end{cases}
\end{equation*}
Observing that $u_1^*\left(\frac{y^*\cdot e}{c},y^*\right)=0$,
then $u_1^*\equiv0$ in $\mathbb R\times\mathbb R^N$ by the maximum principle.
Since
\begin{align*}
&\quad\frac{f_i(x,\bm u(t+\frac{z_n}{c},x))}{u_i(t+\frac{z_n}{c},x)}u_i^n
=\frac{f_i(x,\bm u(t+\frac{z_n}{c},x))}{e^{\lambda_c(ct-x\cdot e+z_n)}\psi_i^c(x)}\\
&=\frac{\sum_{j=1}^{i-1}{a_{ij}u_j(t+\frac{z_n}{c},x)}
+u_i(t+\frac{z_n}{c},x)h_i(x,\bm u(t+\frac{z_n}{c},x))}
{e^{\lambda_c(ct-x\cdot e+z_n)}\psi_i^c(x)}\\
&=\sum_{j=1}^{i-1}{a_{ij}\frac{u_j(t+\frac{z_n}{c},x)}{e^{\lambda_c(ct-x\cdot e+z_n)}\psi_j^c(x)}
\frac{\psi_j^c(x)}{\psi_i^c(x)}}+h_i(x,\bm u(t+\frac{z_n}{c},x))u_i^n\\
&=\sum_{j=1}^{i-1}{a_{ij}\frac{\psi_j^c}{\psi_i^c}}u_j^n
+h_i(x,\bm u(t+\frac{z_n}{c},x))u_i^n,
\end{align*}
using an induction argument, $\{u_i^n\}_{n\in\mathbb{N}}$ converges in
$C^{1,2}_{loc}(\mathbb R\times\mathbb R^N)$ to a function $u_i^*\ge0$ for each $i=2,3,\cdots,m$,
and
\begin{equation*}
\begin{cases}
\frac{\partial u_i^*(t,x)}{\partial t}=d_i\Delta u_i^*
+\left(q_i+2d_i\left(\frac{\nabla\psi_i^c}{\psi_i^c}-\lambda_ce\right)\right)\cdot\nabla u_i^*-\sigma_iu_i^*,\\
u_i^*(t,x)=u_i^*\left(t+\frac{p\cdot e}{c},x+p\right),
\quad\forall\;(t,x)\in\mathbb R\times\mathbb R^N,\ \forall\;p\in \mathcal L.
\end{cases}
\end{equation*}
Since $\sigma_i>0$, the maximum principle then yields that $u_i^*\equiv0$
in $\mathbb R\times\mathbb R^N$. Therefore
$$\mathop{\lim}\limits_{n\to\infty}\frac{U_i(y_n,z_n)}{e^{\lambda_c z_n}\phi_i^c(y_n)}
=\mathop{\lim}\limits_{n\to\infty}\frac{U_i(y_n,z_n)}{e^{\lambda_c z_n}\psi_i^c(y_n)}
\cdot\frac{\psi_i^c(y_n)}{\phi_i^c(y_n)}
=u_i^*\left(\frac{y^*\cdot e}{c},y^*\right)\cdot\frac{\psi_i^c(y^*)}{\phi_i^c(y^*)}=0.$$
Denote
\begin{equation}\label{epsilon-i}
\varepsilon_i^n:=\frac{U_i(y_n,z_n)}{e^{\lambda_c z_n}\phi_i^c(y_n)}\to0
\ \ \text{as}\ n\to\infty,\ \ i\in I.
\end{equation}
Let
$$\bm w(t,x)=W(x,ct-x\cdot e):=N_r\theta_ce^{\lambda_c(ct-x\cdot e)}\bm{\Phi}_{\lambda_c}(x).$$
Then $W(x,s)$ is periodic in $x$ and nondecreasing in $s$, and
it follows from Lemma \ref{super} that $\bm w>\bm 0$ is a supersolution
of \eqref{m-system} in $\mathbb R\times\mathbb R^N$.
Hence
$$\overline{\bm u}(t,x)=\overline{U}(x,ct-x\cdot e):=\min\{\bm w(t,x),\bm 1\}$$
is an irregular supersolution of \eqref{m-system} in $\mathbb R\times\mathbb R^N$.
Furthermore, there exists $\bar{\sigma}\in\mathbb R$ such that
$\overline{U}(x,s)=\bm 1$ for all $(x,s)\in\mathbb R^N\times[\bar{\sigma},\infty)$.
Since $\mathop{\lim}\limits_{s\to-\infty}\overline{U}(x,s)=\bm 0$ and
$\mathop{\lim}\limits_{s\to+\infty}U(x,s)=\bm 1$ uniformly in $x\in\mathbb R^N$,
there exists $(x^\prime,s^\prime)$ with $s^\prime<\bar{\sigma}$ and $z^\prime\ge0$ such that
\begin{equation}\label{s'}
\overline{U}(x^\prime,s^\prime)\le U(x^\prime,s^\prime+z^\prime)<\bm 1.
\end{equation}
Assume without loss of generality that $z^\prime=0$.
By Lemma \ref{Harnack} and in view of \eqref{epsilon-i},
$$U(x,z_n)\le N_rU(y_n,z_n)\le N_r\sum_{i=1}^m\varepsilon_i^n\theta_ce^{\lambda_cz_n}\bm{\Phi}_{\lambda_c}(x),
\quad\forall\;x\in\mathbb R^N.$$
Let $n^\prime\in\mathbb N_+$ be such that $z_{n^\prime}<s^\prime$, and
$$N_r\sum_{i=1}^m\varepsilon_i^{n^\prime}\theta_ce^{\lambda_cz_{n^\prime}}\bm{\Phi}_{\lambda_c}(x)
\ll N_r\theta_ce^{\lambda_cz_{n^\prime}}\bm{\Phi}_{\lambda_c}(x)\le\frac{1}{2}\bm 1,
\quad\forall\;x\in\mathbb R^N.$$
Then $U(x,z_{n^\prime})\ll\overline{U}(x,z_{n^\prime})$ for all $x\in\mathbb R^N$.
By Lemma \ref{CP+}, we have $U(x,s)\ll\overline{U}(x,s)$ for any
$(x,s)\in\mathbb R^N\times[z_{n^\prime},\infty)$, which contradicts \eqref{s'}.
The proof is complete.
\end{proof}

The main result of this subsection is stated as follows.

\begin{theorem}\label{main-1}
Assume (H1)-(H7). Let $U(x,ct-x\cdot e)$ be a pulsating traveling front of \eqref{m-system}
with $c>c_+^0$. Then there exists $\rho>0$ such that
\begin{equation*}
\mathop{\lim}\limits_{s\to-\infty}\frac{U(x,s)}{\rho e^{\lambda_c s}\bm{\Phi}_{\lambda_c}(x)}
=\bm 1\quad\text{uniformly in}\ x\in\mathbb R^N.
\end{equation*}
\end{theorem}
\begin{proof}
In view of Lemma \ref{low-upp}, we have
\begin{equation*}
0<\rho_*:=\mathop{\liminf}\limits_{s\to-\infty}\left\{\inf_{x\in\mathbb R^N}
\frac{U_1(x,s)}{e^{\lambda_c s}\phi_1^c(x)}\right\}
\le\mathop{\limsup}\limits_{s\to-\infty}\left\{\sup_{x\in\mathbb R^N}
\frac{U_1(x,s)}{e^{\lambda_c s}\phi_1^c(x)}\right\}=:\rho^*<+\infty.
\end{equation*}
Next we divide the proof into three steps.

{\bf Step 1}.  We prove that
\begin{equation}\label{1111}
\rho_*=\mathop{\lim}\limits_{s\to-\infty}\left\{\inf_{x\in\mathbb R^N}
\frac{U_1(x,s)}{e^{\lambda_c s}\phi_1^c(x)}\right\}.
\end{equation}
If this is not true, then there exist $\epsilon>0$ and a sequence $\{s_n\}$ such that
\begin{equation}\label{U-1-rho}
s_n\to-\infty\ (n\to\infty),\quad \left\{\inf_{x\in\mathbb R^N}
\frac{U_1(x,s_n)}{e^{\lambda_c s_n}\phi_1^c(x)}\right\}\ge\rho_*(1+2\epsilon).
\end{equation}
Let $\delta_1=\rho_*(1+\frac{3}{2}\epsilon)$ and
$\delta_2=\delta_1\min\left\{1,\frac{1}{\theta_cK_c}\right\}$,
where $K_c$ is given by Lemma \ref{UcpV}. Define
$\underline{\bm u}(t,x)=(\underline{u}_1(t,x),\underline{u}_2(t,x),\cdots,\underline{u}_m(t,x))$
as
\begin{align*}
\underline{u}_1(t,x)&=\underline{U}_1(x,ct-x\cdot e)
=\delta_1e^{\lambda_c(ct-x\cdot e)}\left(\phi_1^c(x)-n_0
e^{\epsilon(ct-x\cdot e)}\phi_1^{\epsilon}(x)\right),\\
\underline{u}_i(t,x)&=\underline{U}_i(x,ct-x\cdot e)
=\delta_2e^{\lambda_c(ct-x\cdot e)}\left(\phi_i^c(x)-\frac{n_0\delta_1}{\delta_2}
e^{\epsilon(ct-x\cdot e)}\phi_i^{\epsilon}(x)\right),\ \ i=2,3,\cdots,m,
\end{align*}
where $(t,x)\in\Omega_{s_0}^-$, and $s_0$ and $n_0>0$ are given by Lemma \ref{sub}.
Since $U_i(x,s_n)\ge\frac{1}{K_c}U_1(x,s_n)$ for each $i$, it follows from \eqref{U-1-rho} that
\begin{equation}\label{uuvv}
\mathop{\lim}\limits_{n\to\infty}\frac{U_i(x,s_n)}{\underline{U}_i(x,s_n)}>1,
\quad\forall\; x\in\mathbb R^N,\ \ i\in I.
\end{equation}
On the other hand, it follows from the definition of $\rho_*$ that
there exists $\{(x_n,z_n)\}_{n\in\mathbb{N}}$ such that
$$z_n\to-\infty\ (n\to\infty),
\quad \mathop{\lim}\limits_{n\to\infty}\frac{U_1(x_n,z_n)}{e^{\lambda_c z_n}\phi_1^c(x_n)}=\rho_*.$$
Therefore there exists $n^*\in\mathbb N_+$ such that
\begin{equation}\label{contr}
z_{n^*}<s_0,\quad U_1(x_{n^*},z_{n^*})\le\rho_*\left(1+\frac{1}{2}\epsilon\right)
e^{\lambda_c z_{n^*}}\phi_1^c(x_{n^*})\le\underline{U}_1(x_{n^*},z_{n^*}).
\end{equation}
Furthermore, it follows from \eqref{uuvv} that there exists
$n^\prime\in\mathbb N_+$ such that
$$s_{n^\prime}<z_{n^*},\quad \underline{U}(x,s_{n^\prime})\ll U(x,s_{n^\prime}),
\quad\forall\; x\in\mathbb R^N.$$
By Lemma \ref{CP-+}, we have
$\underline{U}(x,s)\ll U(x,s)$ for all $(x,s)\in\mathbb R^N\times[s_{n^\prime},s_0]$,
which contradicts \eqref{contr}. Therefore \eqref{1111} holds.

{\bf Step 2}. We prove that $\rho_*=\rho^*$.
Let $\{(x_n^\prime,s_n^\prime)\}_{n\in\mathbb{N}}$ be the sequence such that $x_n^\prime\in\overline{\mathcal D}$
and
$$s_n^\prime\to-\infty,\quad x_n^\prime\to x^*\ (n\to\infty),
\quad \mathop{\lim}\limits_{n\to\infty}
\frac{U_1(x_n^\prime,s_n^\prime)}{e^{\lambda_c s_n^\prime}\phi_1^c(x_n^\prime)}=\rho^*.$$
Let
$$u_1^n(t,x)=\frac{u_1(t+\frac{s_n^\prime}{c},x)}{e^{\lambda_c(ct-x\cdot e+s_n^\prime)}\phi_1^c(x)}
=\frac{U_1(x,ct-x\cdot e+s_n^\prime)}{e^{\lambda_c(ct-x\cdot e+s_n^\prime)}\phi_1^c(x)},$$
then $\{u_1^n\}_{n\in\mathbb{N}}$ is uniformly bounded, and
\begin{align*}
\frac{\partial u_1^n(t,x)}{\partial t}&=d_1(x)\Delta u_1^n
+\left(q_1+2d_1\left(\frac{\nabla\phi_1^c}{\phi_1^c}
-\lambda_ce\right)\right)\cdot\nabla u_1^n-h_1(x,\bm 0)u_1^n
+\frac{f_1(x,\bm u(t+\frac{s_n^\prime}{c},x))}{u_1(t+\frac{s_n^\prime}{c},x)}u_1^n.
\end{align*}
It then follows that $\{u_1^n\}_{n\in\mathbb{N}}$ converges in $C^{1,2}_{loc}(\mathbb R\times\mathbb R^N)$,
up to a subsequence, to a function $u_1^*\ge0$, and
\begin{equation*}
\frac{\partial u_1^*(t,x)}{\partial t}=d_1(x)\Delta u_1^*
+\left(q_1+2d_1\left(\frac{\nabla\phi_1^c}{\phi_1^c}-\lambda_ce\right)\right)\cdot\nabla u_1^*,
\quad (t,x)\in\mathbb R\times\mathbb R^N.
\end{equation*}
Noting that $u_1^*(\frac{x^*\cdot e}{c},x^*)=\rho^*$ and $u_1^*\le\rho^*$ by the definition of $\rho^*$, the maximum principle then shows that
$u_1^*\equiv\rho^*$ for any $(t,x)\in\{\mathbb R\times\mathbb R^N: t\le\frac{x^*\cdot e}{c}\}$,
and furthermore for any $(t,x)\in\mathbb R\times\mathbb R^N$ by the uniqueness of solutions.
Then
$$\rho^*\equiv u_1^*\left(\frac{x\cdot e}{c},x\right)
=\mathop{\lim}\limits_{n\to\infty}u_1^n\left(\frac{x\cdot e}{c},x\right)
=\mathop{\lim}\limits_{n\to\infty}
\frac{U_1(x,s_n^\prime)}{e^{\lambda_c s_n^\prime}\phi_1^c(x)},
\quad \forall\; x\in\overline{\mathcal D}.$$
Since $U_1(\cdot,s)$ is periodic, it is readily seen that
$\mathop{\lim}\limits_{n\to\infty}\left\{\mathop{\inf}
\limits_{x\in\mathbb R^N}\frac{U_1(x,s_n^\prime)}
{e^{\lambda_c s_n^\prime}\phi_1^c(x)}\right\}=\rho^*$,
and hence it follows from \eqref{1111} that
$\rho_*=\rho^*:=\rho$. Therefore
\begin{equation}\label{U-asy}
\mathop{\lim}\limits_{s\to-\infty}\frac{U_1(x,s)}{\rho e^{\lambda_c s}\phi_1^c(x)}=1
\quad\text{uniformly in}\ x\in\mathbb R^N.
\end{equation}

{\bf Step 3}. We prove that
\begin{equation}\label{Ui-asymp}
\mathop{\lim}\limits_{s\to-\infty}\frac{U_i(x,s)}{\rho e^{\lambda_c s}\phi_i^c(x)}=1
\ \ \text{uniformly in}\  x\in\mathbb R^N,\quad \ i=2,3,\cdots,m.
\end{equation}
Let
\begin{align*}
\eta_i(t,x)&=u_i(t,x)-\rho e^{\lambda_c(ct-x\cdot e)}\phi_i^c(x)\\
&=U_i(x,ct-x\cdot e)-\rho e^{\lambda_c(ct-x\cdot e)}\phi_i^c(x)
:=\xi_i(x,ct-x\cdot e),\quad i\in I.
\end{align*}
By \eqref{Ui-induction}, there exists $C_1>0$ such that $|\xi_i|\le C_1e^{\lambda_cs}$
for all $(x,s)\in\mathbb R^N\times\mathbb R$ and each $i$.
Let
$$\underline{\tau}_i:=\liminf_{s\to-\infty}
\left\{\inf_{x\in\mathbb R^N}\frac{\xi_i(x,s)}{e^{\lambda_cs}\phi_i^c(x)}\right\}
\le\limsup_{s\to-\infty}\left\{\sup_{x\in\mathbb R^N}
\frac{\xi_i(x,s)}{e^{\lambda_cs}\phi_i^c(x)}\right\}:=\overline{\tau}_i,\quad\ i=2,3,\cdots,m.$$
Then one only need to prove that $\underline{\tau}_i=\overline{\tau}_i=0$.
Let $\{(\hat{x}_n,\hat{s}_n)\}_{n\in\mathbb{N}}$ with $\hat{x}_n\in\overline{\mathcal D}$ be such that
$$\hat{s}_n\to-\infty,\quad \hat{x}_n\to\hat{x}\ (n\to\infty),\quad
\lim_{n\to\infty}\frac{\xi_i(\hat{x}_n,\hat{s}_n)}
{e^{\lambda_c\hat{s}_n}\phi_i^c(\hat{x}_n)}=\underline{\tau}_i.$$
Define
\begin{align*}
\eta_i^n(t,x)=\frac{\eta_i(t+\frac{\hat{s}_n}{c},x)}
{e^{\lambda_c(ct-x\cdot e+\hat{s}_n)}\psi_i^c(x)}
=\frac{\xi_i(x,s+\hat{s}_n)}{e^{\lambda_c(s+\hat{s}_n)}\psi_i^c(x)},
\quad\ i\in I.
\end{align*}
By a straightforward calculation,
\begin{equation*}
\begin{cases}
\frac{\partial\eta_2^n(t,x)}{\partial t}=d_2\Delta\eta_2^n
+\left(q_2+2d_2\left(\frac{\nabla\psi_2^c}{\psi_2^c}-\lambda_ce\right)\right)
\cdot\nabla\eta_2^n-\sigma_2\eta_2^n+a_{21}\frac{\psi_1^c}{\psi_2^c}\eta_1^n\\
\qquad\quad+\frac{u_2(t+\frac{\hat{s}_n}{c},x)}{e^{\lambda_c(ct-x\cdot e+\hat{s}_n)}\psi_2^c(x)}
\left(h_2(x,\bm u(t+\frac{\hat{s}_n}{c},x))-h_2(x,\bm 0)\right)\eta_2^n,\\
\eta_2^n(t,x)=\eta_2^n(t+\frac{p\cdot e}{c},x+p),\ \ \ \forall\; p\in\mathcal L,
\end{cases}
\end{equation*}
where $\sigma_2=c\lambda_c-\kappa_2(\lambda_c)>0$.
Note from \eqref{U-asy} that
\begin{equation*}
\mathop{\lim}\limits_{n\to\infty}\eta_1^n(t,x)
=\mathop{\lim}\limits_{n\to\infty}\frac{\xi_1(x,s+\hat{s}_n)}{e^{\lambda_c(s+\hat{s}_n)}\psi_1^c(x)}
=\mathop{\lim}\limits_{n\to\infty}\frac{\xi_1(x,s+\hat{s}_n)}{e^{\lambda_c(s+\hat{s}_n)}\phi_1^c(x)}
\cdot\frac{\phi_1^c(x)}{\psi_1^c(x)}=0
\end{equation*}
and $\mathop{\lim}\limits_{n\to\infty}\bm u(t+\frac{\hat{s}_n}{c},x)=\bm 0$
locally uniformly in $(t,x)\in\mathbb R\times\mathbb R^N$, and $\frac{u_2(t+\frac{\hat{s}_n}{c},x)}{e^{\lambda_c(ct-x\cdot e+\hat{s}_n)}\psi_2^c(x)}$
is uniformly bounded in view of \eqref{Ui-induction}.
Therefore $\{\eta_2^n\}_{n\in\mathbb{N}}$ converges in $C^{1,2}_{loc}(\mathbb R\times\mathbb R^N)$,
up to a subsequence, to a function $\eta_2^*\ge0$, which satisfies
\begin{equation*}
\begin{cases}
\frac{\partial\eta_2^*(t,x)}{\partial t}=d_2\Delta\eta_2^*
+\left(q_2+2d_2\left(\frac{\nabla\psi_2^c}{\psi_2^c}-\lambda_ce\right)\right)
\cdot\nabla\eta_2^*-\sigma_2\eta_2^*,\\
\eta_2^*(t,x)=\eta_2^*\left(t+\frac{p\cdot e}{c},x+p\right),
\quad \forall\;p\in\mathcal L.
\end{cases}
\end{equation*}
Therefore $\eta_2^*\equiv0$ in $\mathbb R\times\mathbb R^N$, and hence
$$0=\lim_{n\to\infty}\eta_2^n\left(\frac{\hat{x}_n\cdot e}{c},\hat{x}_n\right)
=\lim_{n\to\infty}\frac{\xi_2(\hat{x}_n,\hat{s}_n)}{e^{\lambda_c\hat{s}_n}\psi_2^c(\hat{x}_n)}
=\lim_{n\to\infty}\frac{\xi_2(\hat{x}_n,\hat{s}_n)}{e^{\lambda_c\hat{s}_n}\phi_2^c(\hat{x}_n)}
\cdot\frac{\phi_2^c(\hat{x}_n)}{\psi_2^c(\hat{x}_n)}
=\underline{\tau}_2\cdot\frac{\phi_2^c(\hat{x})}{\psi_2^c(\hat{x})},$$
which implies that $\underline{\tau}_2=0$.
Similarly, one can prove that $\overline{\tau}_2=0$, and therefore \eqref{Ui-asymp} holds for $i=2$.
Note that for each $i=3,4,\cdots,m$, there hold
\begin{equation*}
\begin{cases}
\frac{\partial\eta_i^n(t,x)}{\partial t}=d_i\Delta\eta_i^n
+\left(q_i+2d_i\left(\frac{\nabla\psi_i^c}{\psi_i^c}-\lambda_ce\right)\right)
\cdot\nabla\eta_i^n-\sigma_i\eta_i^n+\sum_{j=1}^{i-1}a_{ij}\frac{\psi_j^c}{\psi_i^c}\eta_j^n\\
\qquad\quad+\frac{u_i(t+\frac{\hat{s}_n}{c},x)}{e^{\lambda_c(ct-x\cdot e+\hat{s}_n)}\psi_i^c(x)}
\left(h_i(x,\bm u(t+\frac{\hat{s}_n}{c},x))-h_i(x,\bm 0)\right)\eta_i^n,\\
\eta_i^n(t,x)=\eta_i^n(t+\frac{p\cdot e}{c},x+p),\ \ \ \forall\; p\in\mathcal L.
\end{cases}
\end{equation*}
By using an induction and similar argument as above,
one can prove that \eqref{Ui-asymp} hold for all $i=2,3\cdots,m$.
The proof is complete.
\end{proof}


\subsection{The critical case}

In this subsection, we consider the critical case, that is,
$$c=c_*:=c_+^0.$$
Denote
$$\lambda_*:=\lambda_+^0.$$
Noting that $\lambda\mapsto\kappa_i(\lambda)=\kappa_{e}(d_i,q_i,\zeta^i,\lambda)$
is analytic in $\mathbb R$ for each $i\in I$, and
$\kappa_j(\lambda_*)<\kappa_1(\lambda_*)$ for all $j=2,3,\cdots,m$ by (H6).
Therefore there exists $\hat{\epsilon}>0$ such that
$$\kappa_j(\lambda)<\kappa_1(\lambda),
\quad\forall\;\lambda\in(\lambda_*-2\hat{\epsilon},\lambda_*+2\hat{\epsilon}),
\ \ j=2,3,\cdots,m.$$
For any $\lambda\in\mathbb R$, define
\begin{align*}
L_{i,\lambda}&=d_i(x)\Delta+(q_i-2d_i\lambda e)\cdot\nabla
+(d_i\lambda^2-q_i\cdot e\lambda+h_i(x,\bm 0)),\quad \ i\in I.
\end{align*}
It then follows from Lemma \ref{P-EFunc} that the periodic eigenvalue problem
\begin{equation}\label{Ep}
\begin{cases}
\kappa\phi_1=L_{1,\lambda}\phi_1,\\
\kappa\phi_j=L_{j,\lambda}\phi_j+\sum_{k=1}^{j-1}a_{jk}\phi_k, \quad\ j=2,3,\cdots,m,\\
\phi_i(x)=\phi_i(x+p),\quad \forall\; p\in \mathcal L,\ \ \ i\in I
\end{cases}
\end{equation}
admits a positive periodic eigenfunction
$\bm{\Phi}_{\lambda}(x)=(\phi_{1,\lambda}(x),\phi_{2,\lambda}(x),\cdots,\phi_{m,\lambda}(x))$
associated with the eigenvalue $\kappa=\kappa_1(\lambda)$ for any $\lambda\in(\lambda_*-2\hat{\epsilon},\lambda_*+2\hat{\epsilon})$.
Since the function $\lambda\mapsto\kappa_1(\lambda)$ is analytic,
it follows from the standard elliptic estimates that the eigenfunction $\bm{\Phi}_{\lambda}(x)$ associated with $\kappa_1(\lambda)$ is also analytic with respect to
$\lambda\in(\lambda_*-2\hat{\epsilon},\lambda_*+2\hat{\epsilon})$.
Moreover, it follows from the definition of $c_*$ that $\kappa_1^{\prime}(\lambda_*)=c_*$.

Let $\bm{\Phi}_{\lambda}^{(1)}(\cdot)
=(\phi_{1,\lambda}^{(1)}(\cdot),\phi_{2,\lambda}^{(1)}(\cdot),\cdots,\phi_{m,\lambda}^{(1)}(\cdot))$ be the first order derivative of $\bm{\Phi}_{\lambda}(\cdot)$ with respect to $\lambda$,
which is again periodic, and $L_{i,\lambda}^{(1)}$
be the operator whose coefficients are the first order derivatives of
these of $L_{i,\lambda}$ with respect to $\lambda$.
That is,
\begin{align*}
L_{i,\lambda}^{(1)}=-2d_ie\cdot\nabla+(2d_i\lambda-q_i\cdot e),
\quad\ i\in I.
\end{align*}
Note that $\kappa_1(\lambda)\phi_{1,\lambda}=L_{1,\lambda}\phi_{1,\lambda}$
and $\kappa_1(\lambda)\phi_{j,\lambda}=L_{j,\lambda}\phi_{j,\lambda}
+\sum_{k=1}^{j-1}a_{jk}\phi_{k,\lambda}$ for $j=2,3,\cdots,m$.
By differentiating these equations with respect to $\lambda$, we have
\begin{align*}
&(L_{1,\lambda}-\kappa_1(\lambda))\phi_{1,\lambda}^{(1)}
+(L_{1,\lambda}^{(1)}-\kappa_1^\prime(\lambda))\phi_{1,\lambda}=0,\\
&(L_{j,\lambda}-\kappa_1(\lambda))\phi_{j,\lambda}^{(1)}
+(L_{j,\lambda}^{(1)}-\kappa_1^\prime(\lambda))\phi_{j,\lambda}
+\sum_{k=1}^{j-1}a_{jk}\phi_{k,\lambda}^{(1)}=0,\ \ \ j=2,3,\cdots,m.
\end{align*}

Let
$$\bm{\Phi}_{\lambda_*}(x)
=\left(\phi_{1,*}(x),\phi_{2,*}(x),\cdots,\phi_{m,*}(x)\right)$$
be the positive periodic eigenfunction of \eqref{Ep} associated with $\kappa=\kappa_1(\lambda_*)$.
Let $\epsilon_*$ be a fixed constant such that
\begin{equation}\label{e}
0<\epsilon_*\le\min\left\{\hat{\epsilon},\frac{\lambda_*}{2}\right\},
\end{equation}
and
$$\bm{\Phi}_{\lambda_*+\epsilon_*}(x)
=\left(\phi_{1,\epsilon_*}(x),\phi_{2,\epsilon_*}(x),\cdots,\phi_{m,\epsilon_*}(x)\right)\textcolor{blue}{^T}$$
be the positive periodic eigenfunction of \eqref{Ep} associated with $\kappa=\kappa_1(\lambda_*+\epsilon_*)$.
It follows from the definition of $\lambda_*$ and the convexity of $\kappa_1(\cdot)$ that
\begin{equation}\label{sig}
\sigma_*:=c_*(\lambda_*+\epsilon_*)-\kappa_1(\lambda_*+\epsilon_*)<0.
\end{equation}

Denote
$$M_*=\max_{i\in I}\left\{\max_{x\in\mathbb R^N}\phi_{i,*}(x)\right\},\quad
m_*=\min_{i\in I}\left\{\min_{x\in\mathbb R^N}\phi_{i,*}(x)\right\},
\quad\theta_*=\frac{M_*}{m_*},$$
$$M_*^{(1)}=\max_{i\in I}\left\{\max_{x\in\mathbb R^N}|\phi_{i,*}^{(1)}(x)|\right\},\quad
M_{\epsilon_*}=\max_{i\in I}\left\{\max_{x\in\mathbb R^N}\phi_{i,\epsilon_*}(x)\right\},\quad
m_{\epsilon_*}=\min_{i\in I}\left\{\min_{x\in\mathbb R^N}\phi_{i,\epsilon_*}(x)\right\}.$$

\begin{lemma}\label{sub-star}
Assume (H1)-(H6). Then there exists $s_*\in\mathbb R$ such that for any $0<\delta_2\le\delta_1$
and $s_0=s_0(\delta_1)\le s_*$ sufficiently small, there exist $m_0=m_0(\delta_1)>0$ and $n_0=n_0(\delta_1)>0$ such that the function
$\underline{\bm u}(t,x)=(\underline{u}_1(t,x),\underline{u}_2(t,x),\cdots,\underline{u}_m(t,x))$
defined by
\begin{align*}
\underline{u}_1(t,x)&=\underline{U}_1(x,c_*t-x\cdot e)\\
&=\delta_1e^{\lambda_*(c_*t-x\cdot e)}
\left(|c_*t-x\cdot e|\phi_{1,*}(x)-m_0\phi_{1,*}(x)-\phi_{1,*}^{(1)}(x)
+n_0e^{\epsilon_*(c_*t-x\cdot e)}\phi_{1,\epsilon_*}(x)\right),\\
\underline{u}_i(t,x)&=\underline{U}_i(x,c_*t-x\cdot e)\\
&=\delta_2e^{\lambda_*(c_*t-x\cdot e)}
\left(|c_*t-x\cdot e|\phi_{i,*}(x)-\frac{m_0\delta_1}{\delta_2}\phi_{i,*}(x)-\phi_{i,*}^{(1)}(x)
+\frac{n_0\delta_1}{\delta_2}e^{\epsilon_*(c_*t-x\cdot e)}\phi_{i,\epsilon_*}(x)\right),\\
i&=2,3,\cdots,m
\end{align*}
is a subsolution of \eqref{m-system} for
$(t,x)\in\Omega_{s_0,*}^-:=\{\mathbb R\times\mathbb R^N: c_*t-x\cdot e\le s_0\}$,
where $\epsilon_*$ is given by \eqref{e}.
Moreover,
$$\mathop{\sup}\limits_{(x,s)\in\mathbb R^N\times(-\infty,s_0]}
\underline{U}(x,s)\ll\bm 1,\qquad\sup_{x\in\mathbb R^N}\underline{U}(x,s_0)\le\bm 0.$$
\end{lemma}
\begin{proof}
Let $\hat{s}\le0$ be such that
$$\frac{\lambda_*-\epsilon_*}{2}s+2\ln|s|\le0,\quad\forall\; s\le\hat{s},$$
and
$$\hat{s}\le\frac{2}{\lambda_*-\epsilon_*}\ln\frac{|\sigma_*|m_{\epsilon_*}}
{6^2\gamma_0M_*|\bm{\Phi}_{\lambda_*}|},$$
where
$$\gamma_0:=\max_{i,j\in I}
\left\{\max_{(x,\bm u)\in\mathbb R^N\times[\bm{-\theta},\bm{\theta}]}
\left|\frac{\partial h_i(x,\bm u)}{\partial u_j}\right|\right\},
\quad\ \bm{\theta}=\frac{4}{3}\bm 1.$$
Let
$$s_*=\min\left\{-1,-\frac{1}{\lambda_*},-\frac{M_*^{(1)}}{m_*},\hat{s}\right\},$$
and $s_0\le s_*$ be such that
\begin{equation*}
\frac{e^{-\lambda_*s_0}}{3|s_0|M_*}\ge\delta_1,\quad
n_0:=\frac{e^{-\epsilon_*s_0}m_*}{M_{\epsilon_*}}\ge\delta_1,\quad m_0:=3|s_0|.
\end{equation*}
Noting that $n_0e^{\epsilon_*s_0}\phi_{k,\epsilon_*}\le\phi_{k,*}\le|s_0|\phi_{k,*}$, and
$$\left||s|\phi_{k,*}-m_0\phi_{k,*}-\phi_{k,*}^{(1)}
+n_0e^{\epsilon_*s}\phi_{k,\epsilon_*}\right|
\le6|s|\phi_{k,*},\quad\forall\; s\le s_0,\ \ k\in I,$$
and
$$|s|^2e^{(\lambda_*-\epsilon_*)s}\le e^{\frac{1}{2}(\lambda_*-\epsilon_*)s},
\quad\forall\; s\le s_0.$$
By a direct calculation, we have
\begin{align*}
\mathcal N_1(x,\underline{\bm u})&=\frac{\partial\underline{u}_1(t,x)}{\partial t}-
d_1(x)\Delta\underline{u}_1-q_1(x)\cdot\nabla\underline{u}_1-f_1(x,\underline{\bm u})\\
&=\delta_1e^{\lambda_*s}\left\{\left(\kappa_1(\lambda_*)-L_{1,\lambda_*}\right)
\left(|s|\phi_{1,*}-m_0\phi_{1,*}-\phi_{1,*}^{(1)}\right) +\left(L_{1,\lambda_*}^{(1)}-c_*\right)\phi_{1,*}\right .\\
&\qquad\qquad\quad\left .+h_1(x,\bm 0)\left(|s|\phi_{1,*}-m_0\phi_{1,*}-\phi_{1,*}^{(1)}
+n_0e^{\epsilon_*s}\phi_{1,\epsilon_*}\right)\right\}\\
&\quad+n_0\delta_1e^{(\lambda_*+\epsilon_*)s}
\left[c_*(\lambda_*+\epsilon_*)-L_{1,\lambda_*+\epsilon_*}\right]\phi_{1,\epsilon_*}
-\underline{u}_1h_1(x,\underline{\bm u})\\
&=\delta_1e^{\lambda_*s}\left[\left(L_{1,\lambda_*}-\kappa_1(\lambda_*)\right)\phi_{1,*}^{(1)}
+\left(L_{1,\lambda_*}^{(1)}-\kappa_1^{\prime}(\lambda_*)\right)\phi_{1,*}\right]\\
&\quad+h_1(x,\bm 0)\underline{u}_1-\underline{u}_1h_1(x,\underline{\bm u})
+\sigma_*n_0\delta_1e^{(\lambda_*+\epsilon_*)s}\phi_{1,\epsilon_*}\\
&=[h_1(x,\bm 0)-h_1(x,\underline{\bm u})]\underline{u}_1
+\sigma_*n_0\delta_1e^{(\lambda_*+\epsilon_*)s}\phi_{1,\epsilon_*}\\
&\le\gamma_0|\underline{\bm u}||\underline{u}_1|
-|\sigma_*|n_0\delta_1e^{(\lambda_*+\epsilon_*)s}\phi_{1,\epsilon_*}\\
&\le\gamma_0\delta_1^2e^{2\lambda_*s}\sum_{k=1}^m(6|s|)^2\phi_{k,*}\phi_{1,*}
-|\sigma_*|n_0\delta_1e^{(\lambda_*+\epsilon_*)s}\phi_{1,\epsilon_*}\\
&\le\gamma_0\delta_1^2e^{2\lambda_*s}(6|s|)^2M_*|\bm{\Phi}_{\lambda_*}|
-|\sigma_*|n_0\delta_1e^{(\lambda_*+\epsilon_*)s}m_{\epsilon_*}\\
&\le n_0\delta_1e^{(\lambda_*+\epsilon_*)s}
\left\{6^2\gamma_0M_*|\bm{\Phi}_{\lambda_*}||s|^2e^{(\lambda_*-\epsilon_*)s}
-|\sigma_*|m_{\epsilon_*}\right\}\\
&\le n_0\delta_1e^{(\lambda_*+\epsilon_*)s}
\left\{6^2\gamma_0M_*|\bm{\Phi}_{\lambda_*}|
e^{\frac{1}{2}(\lambda_*-\epsilon_*)s}-|\sigma_*|m_{\epsilon_*}\right\}\\
&\le0,
\end{align*}
and similarly,
\begin{align*}
\mathcal N_i(x,\underline{\bm u})&=\frac{\partial\underline{u}_i(t,x)}{\partial t}-
d_i(x)\Delta\underline{u}_i-q_i(x)\cdot\nabla\underline{u}_i-f_i(x,\underline{\bm u})\\
&=\delta_2e^{\lambda_*s}\left\{\sum_{j=1}^{i-1}a_{ij}\phi_{j,*}^{(1)}
+\left[\left(L_{i,\lambda_*}-\kappa_1(\lambda_*)\right)\phi_{i,*}^{(1)}
+\left(L_{i,\lambda_*}^{(1)}-\kappa_1^{\prime}(\lambda_*)\right)\phi_{i,*}\right]\right\}\\
&\quad+h_i(x,\bm 0)\underline{u}_1-\underline{u}_i h_i(x,\underline{\bm u})
+\sigma_*n_0\delta_1e^{(\lambda_*+\epsilon_*)s}\phi_{i,\epsilon_*}\\
&=[h_i(x,\bm 0)-h_i(x,\underline{\bm u})]\underline{u}_i
-|\sigma_*|n_0\delta_1e^{(\lambda_*+\epsilon_*)s}\phi_{i,\epsilon_*}\\
&\le\gamma_0|\underline{\bm u}||\underline{u}_i|
-|\sigma_*|n_0\delta_1e^{(\lambda_*+\epsilon_*)s}\phi_{i,\epsilon_*}\\
&\le n_0\delta_1e^{(\lambda_*+\epsilon_*)s}
\left\{6^2\gamma_0M_*|\bm{\Phi}_{\lambda_*}|
e^{\frac{1}{2}(\lambda_*-\epsilon_*)s}-|\sigma_*|m_{\epsilon_*}\right\}\\
&\le0,\quad i=2,3,\cdots,m.
\end{align*}
Moreover, since $|s|e^{\lambda_*s}$ is nondecreasing in $s\in(-\infty,s_0]$, it follows that
\begin{align*}
\mathop{\sup}\limits_{(x,s)\in\mathbb R^N\times(-\infty,s_0]}\underline{U}_i(x,s)
&<\mathop{\sup}\limits_{(x,s)\in\mathbb R^N\times(-\infty,s_0]}
3\delta_1|s|e^{\lambda_*s}\phi_{i,*}(x)\\
&\le\mathop{\sup}\limits_{x\in\mathbb R^N}
3\delta_1|s_0|e^{\lambda_*s_0}\phi_{i,*}(x)\\
&\le3\delta_1|s_0|e^{\lambda_*s_0}M_*\\
&\le 1,\ \ \ i\in I,
\end{align*}
and for each $i$, it follows from the definition of $m_0$ that
$$\underline{U}_i(x,s_0)\le\delta_1e^{\lambda_*s_0}
\left(|s_0|\phi_{i,*}(x)-m_0\phi_{i,*}(x)+|\phi_{i,*}^{(1)}(x)|+\phi_{i,*}(x)\right)\le 0,
\quad\forall\;x\in\mathbb R^N.$$
Therefore $\mathop{\sup}\limits_{(x,s)\in\mathbb R^N\times(-\infty,s_0]}
\underline{U}(x,s)\ll\bm 1$ and $\sup_{x\in\mathbb R^N}\underline{U}(x,s_0)\le\bm 0$.
The proof is complete.
\end{proof}

\begin{lemma}\label{super*}
Assume (H1)-(H7). Then for any constants $k>0$ and $n>0$,
there exists $s^*=s^*(n)<0$ such that for any $s^0\le s^*$,
the function $\overline{\bm u}(t,x)=\min\{\bm w_c(t,x),\bm 1\}$ is an
irregular supersolution of \eqref{m-system} in
$\Omega_{s^0,*}^-=\{\mathbb R\times\mathbb R^N: c_*t-x\cdot e\le s^0\}$,
where
\begin{equation*}\label{critical-super}
\begin{aligned}
\bm{w}_c(t,x)=W(x,c_*t-x\cdot e)=ke^{\lambda_*(c_*t-x\cdot e)}
\left(|c_*t-x\cdot e|\bm{\Phi}_{\lambda_*}(x)+n\bm{\Phi}_{\lambda_*}(x)
-\bm{\Phi}_{\lambda_*}^{(1)}(x)\right).
\end{aligned}
\end{equation*}
Moreover, $\bm{w}_c(t,x)\textcolor{red}{>}\textcolor{blue}{\gg}\bm 0$ for all $(t,x)\in\Omega_{s^0,*}^-$,
$W(x,s)$ is periodic in $x$ and nondecreasing in $s$ for any $s\le s^0$,
and there exists $k^*=k^*(n)>0$ such that
$\mathop{\inf}\limits_{x\in\mathbb R^N}W(x,2s^0)\ge\bm 1$ for any $k\ge k^*$.
\end{lemma}
\begin{proof}
We only prove that $\bm{w}_c$ is a (regular) supersolution of
\eqref{m-system} in $\Omega_{s^0,*}^-$. For any $k,n>0$, let
\begin{align*}
s^0\le s^*:=\min\left\{-1,\ n-\frac{1}{\lambda_*}-\frac{M_*^{(1)}}{m_*}\right\}.
\end{align*}
Then $\bm{w}_c(t,x)\textcolor{red}{>}\textcolor{blue}{\gg}\bm 0$ for all $(t,x)\in\Omega_{s^0,*}^-$,
$W(x,s)$ is periodic in $x$ and nondecreasing in $s$.
By a direct calculation and in view of (H7), we have
\begin{align*}
\frac{\partial\bm{w}_c(t,x)}{\partial t}
&=D(x)\Delta\bm{w}_c+q(x)\cdot\nabla\bm{w}_c+D_{\bm u}\bm F(x,\bm{0})\bm{w}_c\\
&\ge D(x)\Delta\bm{w}_c+q(x)\cdot\nabla\bm{w}_c+\bm F(x,\bm{w}_c),
\end{align*}
that is, $\bm{w}_c$ is a (regular) supersolution of \eqref{m-system} in $\Omega_{s^0,*}^-$.
Let
$$k^*=\frac{e^{-2\lambda_*s^0}}{(2|s^0|+n)m_*-M_*^{(1)}}>0,$$
then $\mathop{\inf}\limits_{x\in\mathbb R^N}W(x,2s^0)\ge\bm 1$ for any $k\ge k^*$.
The proof is complete.
\end{proof}

\begin{lemma}\label{low*-upp*}
Assume (H1)-(H7). Let
$\bm u(t,x)=U(x,c_*t-x\cdot e)=(U_1(x,c_*t-x\cdot e),\cdots,U_m(x,c_*t-x\cdot e))$
be the critical pulsating traveling front of \eqref{m-system}, then
\begin{equation*}
\mathop{\limsup}\limits_{s\to-\infty}\left\{\sup_{x\in\mathbb R^N}
\frac{U_1(x,s)}{|s|e^{\lambda_*s}\phi_{1,*}(x)}\right\}<+\infty
\quad\text{and}\quad
\mathop{\liminf}\limits_{s\to-\infty}\left\{\inf_{x\in\mathbb R^N}
\frac{U_1(x,s)}{|s|e^{\lambda_*s}\phi_{1,*}(x)}\right\}>0.
\end{equation*}
\end{lemma}
\begin{proof}
Firstly, similar to Step 1 in the proof of Lemma \ref{low-upp}, one can prove that
\begin{equation}\label{limsup*}
\mathop{\limsup}\limits_{s\to-\infty}\left\{\sup_{x\in\mathbb R^N}
\frac{U_1(x,s)}{|s|e^{\lambda_*s}\phi_{1,*}(x)}\right\}<+\infty.
\end{equation}
Hence there exists $B_*>0$ such that
$U_1(x,s)\le B_*|s|e^{\lambda_*s}$ for any $(x,s)\in\mathbb R^N\times\mathbb R$.

Next we prove that for $B_*>0$ large enough, there hold
\begin{equation}\label{Ui*-induc}
U_i(x,s)\le B_*|s|e^{\lambda_*s},\ \ \ \forall\;(x,s)\in\mathbb R^N\times\mathbb R,
\ \ i\in I.
\end{equation}
Choose
$$0<\varepsilon<\min_{i=2,\cdots,m}\left\{\frac{\min_{x\in\mathbb R^N}
\left(\sum_{j=1}^{i-1}a_{ij}(x)\right)}{3\theta_*}, \
\frac{\min_{x\in\mathbb R^N}{|h_i(x,\bm 0)|}}{2}\right\}.$$
Since $\mathop{\lim}\limits_{c_*t-x\cdot e\to-\infty}\bm u(t,x)=\bm 0$,
there exists $Z>0$ such that
$$|h_i(x,\bm u)-h_i(x,\bm 0)|\le\varepsilon,\quad\forall\;(t,x)\in\Omega_{-Z,*}^-,
\quad \forall\;i=2,3,\cdots,m.$$
Define
$$w_i(t,x)=Ke^{\lambda_*(c_*t-x\cdot e)}
\left(|c_*t-x\cdot e|\phi_{i,*}(x)-\phi_{i,*}^{(1)}(x)\right),\quad  i=2,3,\cdots,m,$$
where $(t,x)\in\Omega_{\check{s},*}^-=\{\mathbb R\times\mathbb R^N:\ c_*t-x\cdot e\le\check{s}\}$,
with
$$\check{s}:=\min\left\{-1,-Z,-\frac{2M_*^{(1)}}{m_*}\right\},$$
and $K\ge\frac{2B_*}{m_*}$ is such that
$$Ke^{\lambda_*\check{s}}\left(|\check{s}|\phi_{i,*}(x)-\phi_{i,*}^{(1)}(x)\right)
\ge U_i(x,\check{s}),\quad \forall\; x\in\mathbb R^N, \ \ i=2,3,\cdots,m.$$
We prove firstly that \eqref{Ui*-induc} holds for $i=2$. Noting that
\begin{align*}
\frac{\partial w_2(t,x)}{\partial t}-d_2\Delta w_2-q_2\cdot\nabla w_2
&=h_2(x,\bm 0)w_2+a_{21}K|s|e^{\lambda_*s}\phi_{1,*}\\
&=(h_2(x,\bm 0)+\varepsilon)w_2+a_{21}K|s|e^{\lambda_*s}\phi_{1,*}
-\varepsilon Ke^{\lambda_*s}\left(|s|\phi_{2,*}-\phi_{2,*}^{(1)}\right)\\
&\ge(h_2(x,\bm 0)+\varepsilon)w_2+a_{21}K|s|e^{\lambda_*s}\phi_{1,*}
-\varepsilon Ke^{\lambda_*s}\left(\frac{3}{2}|s|\phi_{2,*}\right)\\
&\ge(h_2(x,\bm 0)+\varepsilon)w_2+a_{21}B_*|s|e^{\lambda_*s}\\
&\ge(h_2(x,\bm 0)+\varepsilon)w_2+a_{21}u_1,\\
\frac{\partial u_2(t,x)}{\partial t}-d_2\Delta u_2-q_2\cdot\nabla u_2
&=a_{21}u_1+h_2(x,\bm 0)u_2+(h_2(x,\bm u)-h_2(x,\bm 0))u_2\\
&\le a_{21}u_1+(h_2(x,\bm 0)+\varepsilon)u_2,
\quad\forall\;(t,x)\in\Omega_{\check{s},*}^-,
\end{align*}
and $h_2(x,\bm 0)+\varepsilon<0$ for any $x\in\mathbb R^N$.
The maximum principle then implies that
$$U_2(x,s)\le Ke^{\lambda_*s}\left(|s|\phi_{2,*}(x)-\phi_{2,*}^{(1)}(x)\right)
\le\frac{3}{2}K|s|e^{\lambda_*s}\phi_{2,*}(x)\le B_*|s|e^{\lambda_*s}$$
for some $B_*>0$ and $(x,s)\in\mathbb R^N\times(-\infty,\check{s}]$.
Due to the boundedness of $U_2$ in $\mathbb R^N\times\mathbb R$, there exists
$B_*$ large enough such that $U_2(x,s)\le B_*|s|e^{\lambda_*s}$ for any
$(x,s)\in\mathbb R^N\times\mathbb R$.
By using an induction argument, and notice that
\begin{align*}
\frac{\partial w_i(t,x)}{\partial t}-d_i\Delta w_i-q_i\cdot\nabla w_i
&=h_i(x,\bm 0)w_i+\sum_{j=1}^{i-1}a_{ij}K|s|e^{\lambda_*s}\phi_{j,*}\\
&=(h_i(x,\bm 0)+\varepsilon)w_i+\sum_{j=1}^{i-1}a_{ij}K|s|e^{\lambda_*s}\phi_{j,*}
-\varepsilon Ke^{\lambda_*s}\left(|s|\phi_{i,*}-\phi_{i,*}^{(1)}\right)\\
&\ge(h_i(x,\bm 0)+\varepsilon)w_i+\sum_{j=1}^{i-1}a_{ij}B_*|s|e^{\lambda_*s}\\
&\ge(h_i(x,\bm 0)+\varepsilon)w_i+\sum_{j=1}^{i-1}a_{ij}u_j,\\
\frac{\partial u_i(t,x)}{\partial t}-d_i\Delta u_i-q_i\cdot\nabla u_i
&=\sum_{j=1}^{i-1}a_{ij}u_j+h_i(x,\bm 0)u_i+(h_i(x,\bm u)-h_i(x,\bm 0))u_i\\
&\le(h_i(x,\bm 0)+\varepsilon)u_i+\sum_{j=1}^{i-1}a_{ij}u_j,
\quad\forall\;(t,x)\in\Omega_{\check{s},*}^-,
\end{align*}
one can prove that \eqref{Ui*-induc} hold for all $i\in I$.

Finally, we prove that
$$\mathop{\liminf}\limits_{s\to-\infty}\left\{\inf_{x\in\mathbb R^N}
\frac{U_1(x,s)}{|s|e^{\lambda_*s}\phi_{1,*}(x)}\right\}>0.$$
If this is not true,
then there exists $\{(y_n,z_n)\}_{n\in\mathbb{N}}$ with $y_n\in\overline{\mathcal D}$ such that
\begin{equation*}
z_n\to-\infty,\quad y_n\to y^* \quad (n\to\infty),
\quad\mathop{\lim}\limits_{n\to\infty}\frac{U_1(y_n,z_n)}{|z_n|e^{\lambda_*z_n}\phi_{1,*}(y_n)}=0.
\end{equation*}
For each $i\in I$, let
\begin{align*}
u_i^n(t,x)=U_i^n(x,c_*t-x\cdot e):&=\frac{U_i(x,c_*t-x\cdot e+z_n)}
{e^{\lambda_*(c_*t-x\cdot e+z_n)}
\left(|c_*t-x\cdot e+z_n|\psi_{i,*}(x)-\psi_{i,*}^{(1)}(x)\right)}\\
&=\frac{u_i(t+\frac{z_n}{c_*},x)}
{e^{\lambda_*(c_*t-x\cdot e+z_n)}\left(|c_*t-x\cdot e+z_n|\psi_{i,*}(x)
-\psi_{i,*}^{(1)}(x)\right)},
\end{align*}
where $\psi_{i,*}(x)>0$ is the periodic eigenfunction associated with
$\kappa_i(\lambda_*)$, and $\psi_{i,*}^{(1)}(x)$ is the first order derivative of $\psi_{i,*}$
with respect to $\lambda$ at $\lambda_*$. That is,
\begin{align*}
L_{i,\lambda_*}\psi_{i,*}=\kappa_i(\lambda_*)\psi_{i,*},
\qquad (L_{i,\lambda_*}-\kappa_i(\lambda_*))\psi_{i,*}^{(1)}
+\left(L_{i,\lambda_*}^{(1)}-\kappa_i^\prime(\lambda_*)\right)\psi_{i,*}=0.
\end{align*}
It then follows from \eqref{Ui*-induc} that $\{u_i^n\}_{n\in\mathbb{N}}$ is uniformly bounded.
By a direct calculation,
\begin{equation*}
\begin{aligned}
\frac{\partial u_i^n(t,x)}{\partial t}&=d_i\Delta u_i^n
+\left(q_i+2d_i\left(\frac{\nabla\left(|s_n|\psi_{i,*}-\psi_{i,*}^{(1)}\right)}
{|s_n|\psi_{i,*}-\psi_{i,*}^{(1)}}-\lambda_*e\right)\right)
\cdot\nabla u_i^n-\sigma_{i,*}u_i^n-h_i(x,\bm 0)u_i^n\\
&\quad-\frac{(\kappa_i^\prime(\lambda_*)-c_*)\psi_{i,*}}{|s_n|\psi_{i,*}-\psi_{i,*}^{(1)}}u_i^n
+\frac{f_i(x,\bm u(t+\frac{z_n}{c_*},x))}{u_i(t+\frac{z_n}{c_*},x)}u_i^n,\\
u_i^n(t,x)&=u_i^n\left(t+\frac{p\cdot e}{c_*},x+p\right),
\quad\forall\;p\in \mathcal L,\ \ i\in I,
\end{aligned}
\end{equation*}
where $s_n:=c_*t-x\cdot e+z_n$, $\sigma_{1,*}=0$ and $\sigma_{i,*}:=c\lambda_*-\kappa_i(\lambda_*)=\kappa_1(\lambda_*)-\kappa_i(\lambda_*)>0$
for $i=2,3,\cdots,m$.
By using an induction argument and similar to Step 3 in the proof of Lemma \ref{low-upp},
up to a subsequence,
$\{u_1^n\}_{n\in\mathbb{N}}$ converges in $C^{1,2}_{loc}(\mathbb R\times\mathbb R^N)$ to a function $u_1^*\equiv0$, and $\{u_i^n\}_{n\in\mathbb{N}}$ converges in $C^{1,2}_{loc}(\mathbb R\times\mathbb R^N)$ to a function $u_i^*\ge0$ for each $i=2,3,\cdots,m$, which satisfies
\begin{equation*}
\begin{aligned}
\frac{\partial u_i^*(t,x)}{\partial t}&=d_i\Delta u_i^*
+\left(q_i+2d_i\left(\frac{\nabla\psi_{i,*}}{\psi_{i,*}}-\lambda_*e\right)\right)
\cdot\nabla u_i^*-\sigma_{i,*}u_i^*,\\
u_i^*(t,x)&=u_i^*\left(t+\frac{p\cdot e}{c_*},x+p\right),
\quad \forall\;p\in \mathcal L.
\end{aligned}
\end{equation*}
Since $\sigma_{i,*}>0$, the maximum principle then yields that $u_i^*\equiv0$
in $\mathbb R\times\mathbb R^N$, and hence
$$\mathop{\lim}\limits_{n\to\infty}\frac{U_i(y_n,z_n)}{|z_n|e^{\lambda_*z_n}\phi_{i,*}(y_n)}=0,
\quad i\in I.$$
The remaining of the proof is similar to that of Step 3 in the proof of Lemma \ref{low-upp},
we omit it here. The proof is complete.
\end{proof}

The main result of this subsection is stated as follows.

\begin{theorem}\label{main-2}
Assume (H1)-(H7). Let $U(x,c_*t-x\cdot e)$ be the critical pulsating traveling front of \eqref{m-system}. Then there exists $\rho>0$ such that
\begin{equation*}
\mathop{\lim}\limits_{s\to-\infty}\frac{U(x,s)}
{\rho|s|e^{\lambda_*s}\bm{\Phi}_{\lambda_*}(x)}
=\bm 1\quad\text{uniformly in}\ x\in\mathbb R^N.
\end{equation*}
\end{theorem}
\begin{proof}
In view of Lemma \ref{low*-upp*},
\begin{equation*}
0<\rho_*:=\mathop{\liminf}\limits_{s\to-\infty}\left\{\inf_{x\in\mathbb R^N}
\frac{U_1(x,s)}{|s|e^{\lambda_*s}\phi_{1,*}(x)}\right\}
\le\mathop{\limsup}\limits_{s\to-\infty}\left\{\sup_{x\in\mathbb R^N}
\frac{U_1(x,s)}{|s|e^{\lambda_*s}\phi_{1,*}(x)}\right\}=:\rho^*<+\infty.
\end{equation*}
Next we divide the proof into three steps.

{\bf Step 1}.  We prove that
\begin{equation}\label{11*}
\rho_*=\mathop{\lim}\limits_{s\to-\infty}\left\{\inf_{x\in\mathbb R^N}
\frac{U_1(x,s)}{|s|e^{\lambda_*s}\phi_{1,*}(x)}\right\}.
\end{equation}
Assume this is not true, then there exist $\epsilon>0$ and a sequence $\{s_n\}$ such that
\begin{equation}\label{U-1-rho*}
s_n\to-\infty\ \ (n\to\infty),\quad \left\{\inf_{x\in\mathbb R^N}
\frac{U_1(x,s_n)}{|s|e^{\lambda_*s}\phi_{1,*}(x)}\right\}\ge\rho_*(1+2\epsilon).
\end{equation}
Let $\delta_1=\rho_*(1+\frac{3}{2}\epsilon)$ and $\delta_2=\delta_1\min\left\{1,\frac{1}{\theta_*K_c}\right\}$,
where $K_c$ is given by Lemma \ref{UcpV}. Define
\begin{align*}
\underline{u}_1(t,x)&=\underline{U}_1(x,c_*t-x\cdot e)\\
&=\delta_1e^{\lambda_*(c_*t-x\cdot e)}
\left(|c_*t-x\cdot e|\phi_{1,*}(x)-m_0\phi_{1,*}(x)-\phi_{1,*}^{(1)}(x)
+n_0e^{\epsilon_*(c_*t-x\cdot e)}\phi_{1,\epsilon_*}(x)\right),\\
\underline{u}_i(t,x)&=\underline{U}_i(x,c_*t-x\cdot e)\\
&=\delta_2e^{\lambda_*(c_*t-x\cdot e)}
\left(|c_*t-x\cdot e|\phi_{i,*}(x)-\frac{m_0\delta_1}{\delta_2}\phi_{i,*}(x)-\phi_{i,*}^{(1)}(x)
+\frac{n_0\delta_1}{\delta_2}e^{\epsilon_*(c_*t-x\cdot e)}\phi_{i,\epsilon_*}(x)\right),\\
i&=2,3,\cdots,m
\end{align*}
where $(t,x)\in\Omega_{s_0,*}^-$, $s_0$ and $m_0$ and $n_0$ are given in Lemma \ref{sub-star}.
Note that  $U_i(x,s_n)\ge\frac{1}{K_c}U_1(x,s_n)$ for each $i$,
it follows from \eqref{U-1-rho*} that
\begin{equation}\label{uuvv*}
\mathop{\lim}\limits_{n\to\infty}\frac{U_i(x,s_n)}{\underline{U}_i(x,s_n)}>1,
\quad\forall\; x\in\mathbb R^N,\ \ i\in I.
\end{equation}
By the definition of $\rho_*$, there exists $\{(x_n,z_n)\}_{n\in\mathbb{N}}$ such that
$$z_n\to-\infty\ \ \text{as}\ \ n\to\infty,\quad\mathop{\lim}\limits_{n\to\infty}
\frac{U_1(x_n,z_n)}{|z_n|e^{\lambda_* z_n}\phi_{1,*}(x_n)}=\rho_*.$$
Hence there exists $n^*\in\mathbb N_+$ such that
\begin{equation}\label{contr*}
z_{n^*}<s_0,\quad U_1(x_{n^*},z_{n^*})\le
\rho_*\left(1+\frac{1}{2}\epsilon\right)|z_{n^*}|e^{\lambda_* z_{n^*}}\phi_{1,*}(x_{n^*})\le\underline{U}_1(x_{n^*},z_{n^*}).
\end{equation}
Furthermore, it follows from \eqref{uuvv*} that there exists
$n^\prime$ such that
$$s_{n^\prime}<z_{n^*},\quad \underline{U}(x,s_{n^\prime})\ll U(x,s_{n^\prime}),
\quad\forall\;x\in\mathbb R^N.$$
Lemma \ref{CP-+} then implies that
$$\underline{U}(x,s)\ll U(x,s),\quad\forall\;(x,s)\in\mathbb R^N\times[s_{n^\prime},s_0],$$
which contradicts \eqref{contr*}, and thus \eqref{11*} holds.

{\bf Step 2}.  We prove that $\rho_*=\rho^*$.
Let $\{(x_n^\prime,s_n^\prime)\}_{n\in\mathbb{N}}$ be the sequence such that $x_n^\prime\in\overline{\mathcal D}$, and
$$s_n^\prime\to-\infty,\ \ x_n^\prime\to x^*\in\overline{\mathcal D}\ \ (n\to\infty),
\quad\mathop{\lim}\limits_{n\to\infty}
\frac{U_1(x_n^\prime,s_n^\prime)}{|s_n^\prime|e^{\lambda_*s_n^\prime}\phi_{1,*}(x_n^\prime)}=\rho^*.$$
Define
\begin{align*}
u_1^n(t,x)&=\frac{u_1(t+\frac{s_n^\prime}{c_*},x)}
{e^{\lambda_*(c_*t-x\cdot e+s_n^\prime)}
\left(|c_*t-x\cdot e+s_n^\prime|\phi_{1,*}-\phi_{1,*}^{(1)}\right)}\\
&=\frac{U(x,c_*t-x\cdot e+s_n^\prime)}
{e^{\lambda_*(c_*t-x\cdot e+s_n^\prime)}
\left(|c_*t-x\cdot e+s_n^\prime|\phi_{1,*}-\phi_{1,*}^{(1)}\right)},
\end{align*}
where
$(t,x)\in\Omega_n:=\left\{c_*t-x\cdot e<-s_n^\prime-\frac{M_*^{(1)}}{m_*}\right\}$.
Noting that $\{u_1^n\}_{n\in\mathbb{N}}$ is uniformly bounded, and
\begin{align*}
\frac{\partial u_1^n(t,x)}{\partial t}&=d_1(x)\Delta u_1^n
+\left[q_1+2d_1\left(\frac{\nabla\left(|c_*t-x\cdot e+s_n^\prime|\phi_{1,*}-\phi_{1,*}^{(1)}\right)}
{|c_*t-x\cdot e+s_n^\prime|\phi_{1,*}-\phi_{1,*}^{(1)}}
-\lambda_*e\right)\right]\cdot\nabla u^n\\
&\quad-h_1(x,\bm 0)u_1^n+\frac{f_1\left(x,\bm u(t+\frac{s_n^\prime}{c_*},x)\right)}
{u_1(t+\frac{s_n^\prime}{c_*},x)}u_1^n.
\end{align*}
It then follows that $\{u_1^n\}_{n\in\mathbb{N}}$ converges in $C^{1,2}_{loc}(\Omega_n)$,
up to a subsequence, to a function $u_1^*\ge0$, which satisfies
\begin{equation*}
\frac{\partial u_1^*(t,x)}{\partial t}=d_1(x)\Delta u_1^*
+\left(q_1+2d_1\left(\frac{\nabla\phi_{1,*}}{\phi_{1,*}}-\lambda_*e\right)\right)\cdot\nabla u_1^*,
\quad (t,x)\in\mathbb R\times\mathbb R^N.
\end{equation*}
Notice that $u_1^*(\frac{x^*\cdot e}{c_*},x^*)=\rho^*$ and $u_1^*\le\rho^*$
from the definition of $\rho^*$,
the maximum principle then shows that
$u_1^*\equiv\rho^*$ for any $(t,x)\in\mathbb R\times\mathbb R^N$.
By similar arguments to Step 2 in the proof of Theorem \ref{main-1},
we have $\rho_*=\rho^*$, and therefore
\begin{equation}\label{U-asy*}
\mathop{\lim}\limits_{s\to-\infty}\frac{U_1(x,s)}{\rho|s|e^{\lambda_* s}\phi_{1,*}(x)}=1
\ \ \text{uniformly in}\ \ x\in\mathbb R^N.
\end{equation}

{\bf Step 3}.  We prove that
\begin{equation}\label{Ui-asymp*}
\mathop{\lim}\limits_{s\to-\infty}\frac{U_i(x,s)}{\rho|s|e^{\lambda_*s}\phi_{i,*}(x)}=1
\ \ \text{uniformly in}\  x\in\mathbb R^N,\quad \ i=2,3,\cdots,m.
\end{equation}
Let
\begin{align*}
\eta_i(t,x)&=u_i(t,x)-\rho e^{\lambda_*(c_*t-x\cdot e)}
\left(|c_*t-x\cdot e|\phi_{i,*}(x)-\phi_{i,*}^{(1)}(x)\right)\\
&=U_i(x,c_*t-x\cdot e)-\rho e^{\lambda_*(c_*t-x\cdot e)}
\left(|c_*t-x\cdot e|\phi_{i,*}(x)-\phi_{i,*}^{(1)}(x)\right)\\
:&=\xi_i(x,c_*t-x\cdot e), \quad i=1,2,\cdots,m,
\end{align*}
and define
$$\underline{\tau}_i:=\liminf_{s\to-\infty}
\left\{\inf_{x\in\mathbb R^N}\frac{\xi_i(x,s)}{\rho|s|e^{\lambda_*s}\phi_{i,*}(x)}\right\}
\le\limsup_{s\to-\infty}\left\{\sup_{x\in\mathbb R^N}
\frac{\xi_i(x,s)}{\rho|s|e^{\lambda_*s}\phi_{i,*}(x)}\right\}=:\overline{\tau}_i,
\ \ i=2,3,\cdots,m.$$
By using an induction argument and similar arguments to those of Step 3 in the proof of Theorem \ref{main-1}, one can infer that $\underline{\tau}_i=\overline{\tau}_i=0$ for each $i$, and hence \eqref{Ui-asymp*} holds. The proof is complete.
\end{proof}

\begin{corollary}\label{U-s-estimate}
Assume (H1)-(H7).
Let $\bm u(t,x)=U(x,ct-x\cdot e)$ be a pulsating traveling front of \eqref{m-system}
with $c\ge c_+^0$, then for each $i=1,2,\cdots,m$,
\begin{equation*}
0<\underline{\lambda}_i:=\liminf_{s\to-\infty}\left\{\inf_{x\in\mathbb R^N}
\frac{\partial U_i(x,s)/\partial s}{U_i(x,s)}\right\}\le
\limsup_{s\to-\infty}\left\{\sup_{x\in\mathbb R^N}
\frac{\partial U_i(x,s)/\partial s}{U_i(x,s)}\right\}:=\bar{\lambda}_i<\infty.
\end{equation*}
\end{corollary}
\begin{proof}
Noting that
\begin{align*}
\frac{\partial u_1(t,x)}{\partial t}=d_1(x)\Delta u_1+q_1(x)\cdot\nabla u_1+u_1h_1(x,\bm u),
\end{align*}
by the standard interior estimates for parabolic equations and Lemma \ref{Harnack},
there exist $C_1,C_2>0$ such that for any $(t,x)\in\mathbb R\times\mathbb R^N$,
\begin{equation}\label{inter-estimate}
\begin{aligned}
\left|\frac{\partial u_1(t,x)}{\partial t}(t,x)\right|+|\Delta u_1(t,x)|+|\nabla u_1(t,x)|
\le C_1\sup_{t-1\le t_1\le t,\ |x_1-x|\le 1}|u_1(t_1,x_1)|\le C_2|u_1(t,x)|.
\end{aligned}
\end{equation}
Since $\partial U_1(x,ct-x\cdot e)/\partial s=\frac{1}{c}\partial u_1(t,x)/\partial t$,
it follows from \eqref{inter-estimate} that
$\frac{\partial U_1(x,s)}{\partial s}/U_1(x,s)$ is globally bounded in $\mathbb R^N\times\mathbb R$,
and hence $\underline{\lambda}_1$ and $\bar{\lambda}_1$ are real numbers.
Next we prove that $\underline{\lambda}_1>0$.

Let $\{(x_n,s_n)\}_{n\in\mathbb{N}}$ be the sequence with $x_n\in\overline{\mathcal D}$, and
$$s_n\to-\infty\ \ (n\to\infty),\qquad \lim_{n\to\infty}\left\{
\frac{\partial U_1(x_n,s_n)/\partial s}{U_1(x_n,s_n)}\right\}=\underline{\lambda}_1.$$
Up to extraction of a subsequence, $x_n\to x_{\infty}\in\overline{\mathcal D}$ as
$n\to\infty$. Define
\begin{align*}
u_1^n(t,x)=\frac{u_1(t+t_n,x)}{u_1(t_n,x_n)}
=\frac{U_1(x,ct-x\cdot e+ct_n)}{U_1(x_n,s_n)},
\end{align*}
where
$$t_n:=\frac{s_n+x_n\cdot e}{c},\qquad
ct_n=s_n+x_n\cdot e\to-\infty\ \ \ (n\to\infty).$$
By Lemma \ref{Harnack}, the sequence $\{u_1^n\}_{n\in\mathbb{N}}$ is locally uniformly bounded in
$\mathbb R\times\mathbb R^N$, which satisfies
\begin{align*}
\frac{\partial u_1^n(t,x)}{\partial t}&=d_1\Delta u_1^n
+q_1\cdot\nabla u_1^n+h_1(x,\bm u(t+t_n,x))u_1^n,\\
u_1^n(t,x)&=u_1^n\left(t+\frac{p\cdot e}{c},x+p\right),
\quad\forall\;(t,x)\in\mathbb R\times\mathbb R^N,\ \forall\;p\in \mathcal L,
\end{align*}
and in particular, $u_1^n(0,x_n)=1$. Noting that
$\lim_{n\to\infty}\bm u(t+t_n,x)=\bm 0$ locally uniformly in $(t,x)\in\mathbb R\times\mathbb R^N$,
by the standard parabolic estimates, the sequence $\{u_1^n\}_{n\in\mathbb{N}}$ converges
up to extraction of a subsequence in $C^{1,2}_{loc}(\mathbb R\times\mathbb R^N)$ to a function $u_1^\infty\ge0$, which satisfies
\begin{align*}
\frac{\partial u_1^\infty(t,x)}{\partial t}&=d_1\Delta u_1^\infty
+q_1\cdot\nabla u_1^\infty+h_1(x,\bm 0)u_1^\infty,\\
u_1^\infty(t,x)&=u_1^\infty\left(t+\frac{p\cdot e}{c},x+p\right),
\quad\forall\;(t,x)\in\mathbb R\times\mathbb R^N,\ \forall\;p\in \mathcal L,
\end{align*}
and in particular, $u_1^\infty(0,x_\infty)=1$.
It then follows from the maximum principle that $u_1^{\infty}(t,x)>0$ for any $(t,x)\in\mathbb{R}\times\mathbb{R}^N$ .
Since
$$\frac{\partial u_1^n(t,x)/\partial t}{u_1^n(t,x)}
=\frac{\partial u_1(t+t_n,x)/\partial t}{u_1(t+t_n,x)}
=c\frac{\partial U_1(x,c(t+t_n)-x\cdot e)/\partial s}{U_1(x,c(t+t_n)-x\cdot e)},$$
by passing the limits and in view of the definition of $\underline{\lambda}_1$, we have
\begin{align*}
w_1(t,x):=\frac{\partial u_1^\infty(t,x)/\partial t}{u_1^\infty(t,x)}
\ge c\underline{\lambda}_1\ \ (c>0)\ \ \text{or}\ \ \le c\underline{\lambda}_1\ \ (c<0),
\end{align*}
and in particular, $w_1(0,x_\infty)=c\underline{\lambda}_1$.
Noting that
\begin{align*}
\frac{\partial w_1(t,x)}{\partial t}
&=d_1\Delta w_1+\left(q_1+2d_1\frac{\nabla u_1^\infty}{u_1^\infty}\right)\cdot\nabla w_1,\\
w_1(t,x)&=w_1\left(t+\frac{p\cdot e}{c},x+p\right),
\quad\forall\;(t,x)\in\mathbb R\times\mathbb R^N,\ \forall\;p\in \mathcal L,
\end{align*}
the maximum principle then implies that $w_1(t,x)\equiv c\underline{\lambda}_1$
in $\mathbb R\times\mathbb R^N$, that is,
$\frac{\partial u_1^\infty(t,x)}{\partial t}\equiv c\underline{\lambda}_1u_1^\infty$.
Hence $\frac{\partial(u_1^\infty e^{-\underline{\lambda}_1ct})}{\partial t}\equiv0$,
which shows that $u_1^\infty(t,x)=e^{\underline{\lambda}_1ct}v(x)$.
On the other hand, $u_1^\infty(t,x)=u_1^\infty\left(t+\frac{p\cdot e}{c},x+p\right)$ for any
$(t,x)\in\mathbb R\times\mathbb R^N$ and $p\in \mathcal L$, then
$u_1^\infty(t,x)=e^{\underline{\lambda}_1(ct-x\cdot e)}\phi_1(x)$, where $\phi_1(x)>0$ satisfies
\begin{align*}
c\underline{\lambda}_1\phi_1&=d_1\Delta\phi_1+(q_1-2\underline{\lambda}_1d_1e)\cdot\nabla\phi_1
+(d_1\underline{\lambda}_1^2-\underline{\lambda}_1q_1\cdot e+h_1(x,\bm 0))\phi_1,\\
\phi_1(x)&=\phi_1(x+p),\quad\forall\;x\in\mathbb R^N,\ \ p\in\mathcal{L}.
\end{align*}
Therefore $c\underline{\lambda}_1=\kappa_1(\underline{\lambda}_1)$. Similarly,
one can obtain $c\overline{\lambda}_1=\kappa_1(\overline{\lambda}_1)$.
Observing that $\kappa_1(0)=\lambda_0(d_1,q_1,\zeta^1)>0$, $U_1(x,s)>0$ and $\lim_{s\to-\infty}U_1(x,s)=0$, the quantities
$\underline{\lambda}_1$ and $\overline{\lambda}_1$ are nonzero with the same sign
and cannot be negative. Consequently,
$\underline{\lambda}_1=\hat{\lambda}>0$,
where $\hat{\lambda}:=\lambda_c$ if $c>c_+^0$ and
$\hat{\lambda}:=\lambda_+^0$ if $c=c_+^0$, in terms of \eqref{c1}.

Noticing that
\begin{align*}
\frac{\partial u_2(t,x)}{\partial t}=d_2(x)\Delta u_2+q_2(x)\cdot\nabla u_2
+\left(a_{21}(x)\frac{u_1(t,x)}{u_2(t,x)}+h_2(x,\bm{u})\right)u_2,
\end{align*}
where
$a_{21}(x)\frac{u_1(t,x)}{u_2(t,x)}+h_2(x,\bm{u})$ is uniformly bounded
in $\mathbb R\times\mathbb R^N$ in terms of Theorems \ref{main-1} and \ref{main-2}
and (H1).
Using a similar argument as above, one can prove that there exists $\phi_2(x)>0$ such that
\begin{equation}\label{2-phi2}
\begin{aligned}
c\underline{\lambda}_2\phi_2&=d_2\Delta\phi_2+(q_2-2\underline{\lambda}_2d_2e)\cdot\nabla\phi_2
+\left(d_2\underline{\lambda}_2^2-\underline{\lambda}_2q_2\cdot e
+a_{21}\frac{\phi_1^{\lambda}(x)}{\phi_2^{\lambda}(x)}+h_2(x,\bm 0)\right)\phi_2,\\
\phi_2(x)&=\phi_2(x+p),\quad\forall\;x\in\mathbb R^N,\ \ p\in\mathcal{L}.
\end{aligned}
\end{equation}
Since $\phi_2^{\lambda}(x)>0$ satisfies \eqref{2-phi2}
with $\underline{\lambda}_2=\hat{\lambda}$, the uniqueness of the principal eigenvalue
then implies that $\underline{\lambda}_2>0$.
Since for each $i=3,\cdots,m$,
\begin{align*}
\frac{\partial u_i(t,x)}{\partial t}=d_i(x)\Delta u_i+q_i(x)\cdot\nabla u_i
+\left(\sum_{j=1}^{i-1}a_{ij}(x)\frac{u_j(t,x)}{u_i(t,x)}+h_i(x,\bm{u})\right)u_i,
\end{align*}
where
$\sum_{j=1}^{i-1}a_{ij}(x)\frac{u_j(t,x)}{u_i(t,x)}+h_i(x,\bm{u})$ is uniformly bounded
in $(t,x)\in\mathbb R\times\mathbb R^N$, a similar argument as above shows that $\underline{\lambda}_i>0$ for each $i\in I$.
The proof is then complete.
\end{proof}

To this end, we give the proof of Theorem \ref{Asy} as follows.

\begin{proof}[{\bf Proof of Theorem \ref{Asy}}]
The proof follows from Theorems \ref{main-1} and \ref{main-2}.
The proof is complete.
\end{proof}


\section{Uniqueness of pulsating traveling fronts}\label{UNI}

In this section, we always assume that (H1)-(H8) are satisfied,
and we prove the uniqueness of pulsating traveling fronts.
For each $i\in I$, let
\begin{equation}\label{varrho}
\begin{aligned}
\varrho_i&:=\sup\left\{\varrho\ge0:\ \sum_{k=1}^m
\left|\frac{\partial f_i(x,\bm u)}{\partial u_k}
-\frac{\partial f_i(x,\bm 1)}{\partial u_k}\right|
\le\frac{\alpha_*|\mu^-|}{2},
\ \ \forall\;(x,\bm u)\in\mathbb R^N\times[(1-\varrho)\bm 1,(1+\varrho)\bm 1]\right\},
\end{aligned}
\end{equation}
where
$$\alpha_*:=\frac{\min_{i\in I}\{\min_x{\psi_i(x)}\}}
{\max_{i\in I}\{\max_x{\psi_i(x)}\}},$$
and $\mu^-<0$ and $\bm{\Psi}(x)=(\psi_1(x),\psi_2(x),\cdots,\psi_m(x))$ are given in (H8).

Firstly, we establish a comparison principle in the region where the fronts are close
to the stable periodic solution.

\begin{lemma}\label{1-CP}
Assume (H1)-(H8). If $\underline{\bm u}(t,x)=\underline{U}(x,ct-x\cdot e)$
and $\overline{\bm u}(t,x)=\overline{U}(x,ct-x\cdot e)$
are sub- and supersolutions of \eqref{m-system} in $C^{1,2}_b(\mathbb R\times\mathbb R^N)$, respectively, $\underline{U}(x,s)$ and $\overline{U}(x,s) $ are periodic in $x$,
and there exists $s^*\in\mathbb R$ such that
\begin{equation*}
\begin{cases}
\underline{U}(x,s),\overline{U}(x,s)\in[(1-\varrho^*)\bm 1,\bm 1],
\quad \forall\;(x,s)\in\mathbb R^N\times[s^*,+\infty),\\
\mathop{\liminf}\limits_{s\to+\infty}\left\{
\mathop{\inf}\limits_{x}\left\{\overline{U}(x,s)-\underline{U}(x,s)\right\}\right\}\ge\bm 0,\\
\overline{U}(x,s^*)\ge\underline{U}(x,s^*),\quad \forall\; x\in\mathbb R^N,
\end{cases}
\end{equation*}
where $\varrho^*=\min\{1,\min_{i\in I}\varrho_i\}$.
Then
$$\overline{U}(x,s)\ge\underline{U}(x,s),\quad\forall\;(x,s)\in\mathbb R^N\times[s^*,+\infty).$$
\end{lemma}
\begin{proof}
The proof follows a similar argument to that of \cite[Lemma 3.1]{duJDE2019},
we omit the details here.
\end{proof}

\begin{theorem}\label{uniqu}
Assume (H1)-(H8).
If $\bm{u}(t,x)=U(x,ct-x\cdot e)$ and $\bm{v}(t,x)=V(x,ct-x\cdot e)$ are
two pulsating traveling fronts of \eqref{m-system} with $c\not=0$.
Then there exists $z_0\in\mathbb R$ such that
\begin{equation*}
U(x,s+z_0)=V(x,s),\quad\forall\;(x,s)\in\mathbb R^N\times\mathbb R,
\end{equation*}
that is, there exists $\sigma\in\mathbb R$ $(\sigma=z_0/c)$ such that
\begin{equation*}
\bm u(t+\sigma,x)=\bm v(t,x),\quad\forall\;(t,x)\in\mathbb R\times\mathbb R^N.
\end{equation*}
\end{theorem}
\begin{proof}
In view of Theorem \ref{Asy}, there exist $\rho_i>0$, $i=1,2$ such that
\begin{equation}\label{rho-12}
\begin{aligned}
\mathop{\lim}\limits_{s\to-\infty}\frac{U(x,s)}
{\rho_1|s|^\tau e^{\lambda_c s}\bm{\Phi}_{\lambda_c}(x)}=\bm 1
\quad\text{and}\quad
\mathop{\lim}\limits_{s\to-\infty}\frac{V(x,s)}
{\rho_2|s|^\tau e^{\lambda_c s}\bm{\Phi}_{\lambda_c}(x)}=\bm 1
\quad\text{uniformly\ in}\ x\in\mathbb R^N,
\end{aligned}
\end{equation}
where $\tau=0$ if $c>c_*$ and $\tau=1$ if $c=c_*$.
Next we divide the proof into three steps.

{\bf Step 1}. We prove that there exists $\bar{z}\in\mathbb R$ such that
$$U(x,s+\bar{z})\ge V(x,s),\quad\forall\;(x,s)\in\mathbb R^N\times\mathbb R.$$
Let $z_0\in\mathbb R$ be such that $\rho_1e^{\lambda_cz_0}>\rho_2$.
By \eqref{rho-12}, there exists $M>0$ such that
$$U(x,s+z_0)\ge V(x,s),\quad\forall\; (x,s)\in\mathbb R^N\times(-\infty,-M],$$
and that
$$|U(x,s+z_0)-\bm 1|+|V(x,s)-\bm 1|\le\varrho^*,
\quad\forall\;(x,s)\in\mathbb R^N\times[M,+\infty),$$
where $\varrho^*>0$ is given in Lemma \ref{1-CP}.
By the boundedness of $V(x,s)$ in $\mathbb R^N\times[-2M,2M]$, and note that
$\mathop{\lim}\limits_{s\to+\infty}U(x,s)=\bm 1$ uniformly in $x\in\mathbb R^N$,
there exits $\bar{z}\ge z_0$ such that
$$U(x,s+\bar{z})\ge V(x,s),\quad\forall\;(x,s)\in\mathbb R^N\times[-2M,2M],$$
and hence
\begin{equation}\label{2M}
U(x,s+\bar{z})\ge V(x,s),\quad\forall\;(x,s)\in\mathbb R^N\times[-\infty,2M].
\end{equation}
Lemma \ref{1-CP} applied to
\begin{align*}
\underline{\bm u}(t,x)&:=\bm v(t,x)=V(x,ct-x\cdot e),\\
\overline{\bm u}(t,x)&:=\bm u\left(t+\frac{\bar{z}}{c},x)\right)=U(x,ct-x\cdot e+\bar{z})
\end{align*}
and $s^*=M$ shows that
$U(x,s+\bar{z})\ge V(x,s)$ for all $(x,s)\in\mathbb R^N\times[M,\infty)$,
which together with \eqref{2M} yields that
$$U(x,s+\bar{z})\ge V(x,s),\quad\forall\;(x,s)\in\mathbb R^N\times\mathbb R.$$

{\bf Step 2}. Let
$$z_*=\inf\left\{z\in\mathbb R\ | \ U(x,s+z)\ge V(x,s),
\ \forall\;(x,s)\in\mathbb R^N\times\mathbb R\right\}.$$
Observe that $-\infty<z_*\le\bar{z}$, and it follows from \eqref{rho-12} that $\rho_1e^{\lambda_cz_*}\ge\rho_2$
(otherwise, there exist $i_0$ and $(\tilde{x},\tilde{s})$ such that $U_{i_0}(\tilde{x},\tilde{s}+z_*)<V_{i_0}(\tilde{x},\tilde{s})$,
which contradicts the definition of $z_*$).
Assume that $\rho_1e^{\lambda_cz_*}>\rho_2$. Define
$$\bm w(t,x)=\bm u\left(t+\frac{z_*}{c},x\right)-\bm v(t,x)=U(x,ct-x\cdot e+z_*)-V(x,ct-x\cdot e),$$
then $\bm w\ge\bm 0$, and for each $i\in I$, we have
\begin{align*}
\frac{\partial w_i(t,x)}{\partial t}-d_i(x)\Delta w_i-q_i(x)\cdot\nabla w_i
&=f_i\left(x,\bm u\left(t+\frac{z_*}{c},x\right)\right)-f_i(x,\bm v(t,x))\\
&\ge\left(\int_0^1{\frac{\partial f_i}{\partial u_i}(x,s\bm u+(1-s)\bm v)}ds\right)w_i
\end{align*}
by (H3). If there exist $i_0$ and $(\hat{x},\hat{s})\in\mathbb R^N\times\mathbb R$ such that
$w_{i_0}(\hat{t},\hat{x})=0$, where $\hat{t}:=\frac{\hat{s}+\hat{x}\cdot e}{c}$,
then it follows from the maximum principle that
$w_{i_0}(t,x)\equiv0$ for any $(t,x)\in\{\mathbb R\times\mathbb R^N: t\le\hat{t}\}$.
Noting that $w_{i_0}(t,x)=w_{i_0}\left(t+\frac{p\cdot e}{c},x+p\right)$
in $\mathbb R\times\mathbb R^N$ for all $p\in\mathcal L$,
then $w_{i_0}(t,x)\equiv0$ for any $\mathbb R\times\mathbb R^N$, that is,
$U_{i_0}(x,s+z_*)\equiv V_{i_0}(x,s)$ for any $(x,s)\in\mathbb R^N\times\mathbb R$,
which contradicts $\rho_1e^{\lambda_cz_*}>\rho_2$ in terms of \eqref{rho-12}.
Therefore $\bm w(t,x)\gg\bm 0$ for all $(t,x)\in\mathbb R\times\mathbb R^N$,
that is,
\begin{equation}\label{z*}
U(x,s+z_*)\gg V(x,s),\quad\forall\;(x,s)\in\mathbb R^N\times\mathbb R.
\end{equation}
On the other hand, one gets from $\rho_1e^{\lambda_cz_*}>\rho_2$ that
$\rho_1e^{\lambda_c(z_*-l)}>\rho_2$ for any $l\in\left(0,z_*-\frac{1}{\lambda_c}\ln{\frac{\rho_2}{\rho_1}}\right)$.
Now fix $l_0\in\left(0,z_*-\frac{1}{\lambda_c}\ln{\frac{\rho_2}{\rho_1}}\right)$, and let $\theta\in\left(\frac{\rho_2}{\rho_1e^{\lambda_c(z_*-l_0)}},1\right)$.
Since $\mathop{\lim}\limits_{s\to-\infty}\frac{|s+z_*-l_0|}{|s|}=1$, there exists
$K_\theta>0$ such that $\frac{|s+z_*-l_0|}{|s|}\ge\theta$ for any $s\le-K_\theta$.
Let
$$0<\epsilon<\frac{\theta\rho_1e^{\lambda_c(z_*-l_0)}-\rho_2}{2\rho_1e^{\lambda_c(z_*-l_0)}+\rho_2}.$$
In view of \eqref{rho-12}, there exits $K_\epsilon>0$ such that
$$\left|\frac{U(x,s+z_*-l_0)}
{\rho_1e^{\lambda_c(z_*-l_0)}|s+z_*-l_0|^\tau e^{\lambda_c s}
{\bm\Phi}_{\lambda_c}(x)}-\bm 1\right|\le\epsilon\ \quad \text{and}\ \quad
\left|\frac{V(x,s)}{\rho_2|s|^\tau e^{\lambda_cs}
{\bm\Phi}_{\lambda_c}(x)}-\bm 1\right|\le\epsilon$$
for any $(x,s)\in\mathbb R^N\times(-\infty,-K_\epsilon]$.
Therefore $U(x,s+z_*-l_0)\ge V(x,s)$ for any
$(x,s)\in\mathbb R^N\times(-\infty,-K_\epsilon-K_\theta]$.
Furthermore, for any $l\in(0,l_0]$, we have
$$U(x,s+z_*-l)\ge V(x,s),\quad\forall\;(x,s)\in\mathbb R^N\times(-\infty,-K_\epsilon-K_\theta].$$
Let $M\ge K_\epsilon+K_\theta$. By \eqref{z*},
there exists $0<l_M\le l_0$ such that for any $0<l\le l_M$,
$$U(x,s+z_*-l)\ge V(x,s),\quad\forall\;(x,s)\in\mathbb R^N\times[-M,M].$$
Now let $M>0$ be large enough such that
$$|U(x,s+z_*-l_0)-\bm 1|+|V(x,s)-\bm 1|\le\varrho^*,
\quad\forall\;(x,s)\in\mathbb R^N\times[M,+\infty).$$
Observe that
\begin{equation}\label{l-M}
U(x,s+z_*-l_M)\ge V(x,s),\quad\forall\;(x,s)\in\mathbb R^N\times(-\infty,M].
\end{equation}
Lemma \ref{1-CP} then applied to
\begin{align*}
\underline{\bm u}(t,x)&=\bm v(t,x)=V(x,ct-x\cdot e),\\
\overline{\bm u}(t,x)&=\bm u\left(t+\frac{z_*-l_M}{c},x\right)=U(x,ct-x\cdot e+z_*-l_M)
\end{align*}
and $s^*=M$, together with \eqref{l-M}, yields that
$$U(x,s+z_*-l_M)\ge V(x,s),\quad\forall\;(x,s)\in\mathbb R^N\times\mathbb R,$$
which contradicts the definition of $z_*$. Therefore, $\rho_1e^{\lambda_cz_*}=\rho_2$.

{\bf Step 3}. Define
$$z^*=\sup\left\{z\in\mathbb R\ |\ U(x,s+z)\le V(x,s),
\ \forall\;(x,s)\in\mathbb R^N\times\mathbb R\right\}.$$
Similar to Step 2, one can prove that $z^*$ is bounded,
and that $\rho_1e^{\lambda_cz^*}\le\rho_2$. Noting that
$$-z^*=\inf\left\{-z\in\mathbb R\ |\ V(x,s-z)\ge U(x,s),
\ \forall\;(x,s)\in\mathbb R^N\times\mathbb R\right\}.$$
By changing the roles of $U$ and $V$, and following similar arguments as in Step 2,
we conclude that $\rho_2e^{-\lambda_cz^*}=\rho_1$, that is,
$\rho_1e^{\lambda_cz^*}=\rho_2$. Therefore, $z^*=z_*:=z_0$, and consequently,
$$U(x,s+z_0)=V(x,s), \quad\forall\;(x,s)\in\mathbb R^N\times\mathbb R,$$
which is equivalent to
$$\bm u(t+\sigma,x)=\bm v(t,x),\quad\forall\;(t,x)\in\mathbb R\times\mathbb R^N.$$
The proof of Theorem \ref{uniqueness} is then complete.
\end{proof}


\section{ Stability of pulsating traveling fronts }\label{STA}

This section is devoted to the study of asymptotic stability of pulsating traveling fronts for solutions of the Cauchy problem
\begin{equation}\label{initial-P}
\begin{cases}
\frac{\partial\bm{u}(t,x)}{\partial t}=D(x)\Delta\bm{u}+q(x)\cdot\nabla\bm{u}
+\bm{F}(x,\bm{u}),\quad t>0,\ \  x\in\mathbb R^N,\\
\bm u(0,x)=\bm{u}_0(x),\quad x\in\mathbb R^N,
\end{cases}
\end{equation}
where $\bm u_0$ is a uniformly continuous function
from $\mathbb R^N$ to $\mathbb R^m$, and $\bm 0<\bm u_0<\bm 1$.
We shall use
$$\bm u(t,x;\bm u_0)=(u_1(t,x;\bm u_0),u_2(t,x;\bm u_0),\cdots,u_m(t,x;\bm u_0))\textcolor{blue}{^T}$$
to denote the classical solution of \eqref{initial-P} with initial data
$\bm u(0,\cdot;\bm u_0)=\bm u_0$.
Observe that $\bm 0\le\bm u(t,x;\bm u_0)\le\bm 1$ for any
$(t,x)\in(0,\infty)\times\mathbb R^N$ by the maximum principle.
We first state a comparison principle as follows.

\begin{lemma}\label{CP}
Let $D=\{(t,x)\in\mathbb R\times\mathbb R^N: t>t_0,\ ct-x\cdot e<s_0\}$, where
$t_0\ge0$ and $s_0\in\mathbb R$.
Assume that $\underline{\bm u}$,
$\overline{\bm u}\in C_b^{1+\theta/2,2+\theta}(D)\cap C_b(\overline{D})$
are sub-and supersolutions of \eqref{initial-P} in $D$, respectively,
and $\underline{\bm u}\le\bm 1$ and $\overline{\bm u}\ge\bm 0$ for all $(t,x)\in\overline{D}$.
If $\underline{\bm u}(t,x)\le\bm u(t,x;\bm u_0)\le\overline{\bm u}(t,x)$ for
all $(t,x)\in\partial D:=\{\mathbb R\times\mathbb R^N: t=t_0,\ ct-x\cdot e<s_0\}
\cup\{\mathbb R\times\mathbb R^N: t>t_0,\ ct-x\cdot e=s_0\}$, then
\begin{equation*}
\underline{\bm u}(t,x)\le\bm u(t,x;\bm u_0)\le\overline{\bm u}(t,x),
\quad\forall\;(t,x)\in\overline{D}.
\end{equation*}
\end{lemma}
\begin{proof}
The proof is similar to that of \cite[Proposition 4.1]{zhao2011}, we omit the details here.
\end{proof}

In this section, the initial data $\bm 0<\bm u_0<\bm 1$ is assumed to be close to the pulsating traveling front at $t=0$ at both ends, in the sense that
\begin{equation}\label{u0v0}
\liminf_{\varsigma\to+\infty}\left\{\inf_{x\in\mathbb R^N,\ -x\cdot e\ge\varsigma}\bm u_0(x)\right\}
\ge(1-\varepsilon_0)\bm 1
\end{equation}
for some $\varepsilon_0\in(0,\frac{\hat{\delta}}{2\delta_M})$,
where $\hat{\delta}\in(0,\delta_m]$ is some constant, and
$$
\delta_m=\min_{k=1,2,\cdots,m}\left\{\min_{x\in\mathbb R^N}\frac{1}{\psi_k(x)}\right\},
\qquad
\delta_M=\max_{k=1,2,\cdots,m}\left\{\max_{x\in\mathbb R^N}\frac{1}{\psi_k(x)}\right\},$$
with $\bm{\Psi}=(\psi_1,\cdots,\psi_m)$ given in (H8).
Moreover, there exists $k>0$ such that
\begin{equation}\label{u0v0-behavior}
\limsup_{\varsigma\to-\infty}\left\{\sup_{\substack{x\in\mathbb R^N\\-x\cdot e\le\varsigma}}
\left|\frac{\bm u_0(x)}{k|x\cdot e|^\tau e^{-\lambda_c(x\cdot e)}
\bm{\Phi}_{\lambda_c}(x)}-\bm 1\right|\right\}=0,
\end{equation}
where $\tau=0$ if $c>c_+^0(e)$ and $\tau=1$ if $c=c_+^0(e)$.

Using a very similar argument as in \cite[Proposition A.4]{zhao2014}, we have the following result.

\begin{lemma}\label{asy}
Assume (H1)-(H8), and that there exists $k>0$ such that \eqref{u0v0-behavior} holds.
Let $\textcolor{red}{I}\textcolor{blue}{J}\subset[0,+\infty)$ be any compact subset, then there exists $s_0\in\mathbb R$ such that
\begin{align*}
&\limsup_{\varsigma\to-\infty}\left\{\sup_{\substack{x\in\mathbb R^N\\-x\cdot e\le\varsigma}}
\left|\frac{\bm u(t,x;\bm u_0)-U(x,ct-x\cdot e+s_0)}{U(x,ct-x\cdot e+s_0)}\right|\right\}=0
\quad\text{uniformly in}\ \ t\in \textcolor{red}{I}\textcolor{blue}{J}.
\end{align*}
\end{lemma}
\begin{proof}
In view of Theorem \ref{Asy}, there exists $s_0\in\mathbb R$ such that
\begin{align*}
&\limsup_{\varsigma\to-\infty}\left\{\sup_{\substack{x\in\mathbb R^N\\-x\cdot e\le\varsigma}}
\left|\frac{U(x,-x\cdot e+s_0)}
{k|x\cdot e|^\tau e^{-\lambda_c(x\cdot e)}\bm{\Phi}_{\lambda_c}(x)}-\bm 1\right|
\right\}=0,
\end{align*}
where $s_0$ is uniquely determined by $k$. The remaining of the proof is similar to that of
\cite[Proposition A.4]{zhao2014}, we omit the details here.
\end{proof}

In the following, we study the global stability properties of
pulsating traveling fronts, in the case $c>c_+^0(e)$ and $c=c_+^0(e)$, respectively.

\subsection{The super-critical case $c>c_+^0(e)$}

In this subsection, we consider the super-critical case $c>c_+^0(e)$.
Let $0<\lambda_c<\lambda_+^0$ be such that $\kappa_1(\lambda_c)=c\lambda_c$, and
$0<\epsilon<\min\left\{\frac{\lambda_+^0-\lambda_c}{2},\frac{\lambda_c}{2}\right\}$.
It is easy to see that
$$\sigma_\epsilon:=\kappa_1(\lambda_c+\epsilon)-c(\lambda_c+\epsilon)<0,$$
and there exists $\epsilon_0>0$ such that $|\sigma_\epsilon|\le|\mu^-|$
for any $0<\epsilon\le\epsilon_0$, where
$\mu^-<0$ is the principal eigenvalue associated with positive periodic eigenfunction
$\bm{\Psi}=(\psi_1,\cdots,\psi_m)$ given in (H8).

Let
\begin{equation*}
0<\epsilon<\min\left\{\frac{\lambda_+^0-\lambda_c}{2}, \frac{\lambda_c}{2}, \epsilon_0\right\},
\qquad \beta=\frac{|\sigma_\epsilon|}{2},
\end{equation*}
and $\chi(s)$ be a smooth function such that
\begin{equation}\label{chi-s}
\begin{cases}
\chi(s)=0,\quad \forall s\ge\bar{s},\\
\chi(s)=1,\quad \forall s\le\underline{s},\\
|\chi^\prime|+|\chi^{\prime\prime}|\le1,\\
\chi^\prime\le 0,
\end{cases}
\end{equation}
where $\underline{s}<\bar{s}$ are certain constants. Define
\begin{align*}
\bm\xi(x,s)&=\chi(s)e^{(\lambda_c+\epsilon)s}\bm{\Phi}_{\lambda_c+\epsilon}(x)
+(1-\chi(s))\bm{\Psi}(x),
\end{align*}
where $\bm{\Phi}_{\lambda_c+\epsilon}
=\left(\phi_1^{\epsilon},\phi_2^{\epsilon},\cdots,\phi_m^{\epsilon}\right)\textcolor{blue}{^{\top}}$
is the positive periodic eigenfunction of \eqref{u1um-linear-PE} with
$\lambda=\lambda_c+\epsilon$ associated with principal eigenvalue $\kappa_1(\lambda_c+\epsilon)$.

\begin{lemma}\label{z-0}
Assume (H1)-(H8). Let $U(x,ct-x\cdot e)$ be a pulsating traveling front of \eqref{m-system}
with $c>c_+^0$, then there exists $z_0\in\mathbb R$ such that
\begin{equation}\label{pre}
\begin{aligned}
\sup_{(x,s)\in\mathbb R^N\times\mathbb R}
\frac{U(x,s)-\delta\bm\xi(x,s+z_0)-\bm 1}{\bm{\Psi}(x)}&\le-\frac{\delta}{2}\bm 1,
\end{aligned}
\quad\forall\;\delta\in(0,\delta_m].
\end{equation}
\end{lemma}
\begin{proof}
We first prove that
\begin{equation}\label{U-1}
\limsup_{z\to+\infty}\left\{\sup_{(x,s)\in\mathbb R^N\times\mathbb R,\ \delta\in(0,\delta_m]}
\frac{U(x,s)-\delta\bm\xi(x,s+z)-\bm 1}{\delta\bm{\Psi}(x)}\right\}\le-\bm 1.
\end{equation}
If \eqref{U-1} is not true, then there exist $\{(x_n,s_n)\}_{n\in\mathbb N}$,
$\{\delta_n\}_{n\in\mathbb N}$ and $\{z_n\}_{n\in\mathbb N}$ such that
$$z_n\to+\infty\ (n\to\infty),\quad\delta_n\in(0,\delta_m],\qquad \frac{U_i(x_n,s_n)-\delta_n\xi_i(x_n,s_n+z_n)-1}{\delta_n\psi_i(x_n)}\ge -1+\tau$$
for some $i\in I$ and $\tau\in(0,1)$.
Observe that $U_i$, $\xi_i$ and $\psi_i$ are periodic in $x$,
one may assume without loss of generality that $x_n\in\overline{\mathcal D}$,
and hence $x_n\to x_*\in\overline{\mathcal D}$ as $n\to\infty$ up to a subsequence.
Since $z_n\to+\infty$ as $n\to\infty$, we have either $s_n+z_n\to\infty$ as $n\to\infty$
or $\{s_n+z_n\}_{n\in\mathbb N}$ is bounded from above.
If $s_n+z_n\to\infty$ as $n\to\infty$, then by the definition of $\xi_i$, we have
$$-1=\lim_{n\to\infty}\frac{-\delta_n\xi_i(x_n,s_n+z_n)}{\delta_n\psi_i(x_n)}
\ge\lim_{n\to\infty}\frac{U_i(x_n,s_n)-\delta_n\xi_i(x_n,s_n+z_n)-1}{\delta_n\psi_i(x_n)}
\ge -1+\tau,$$
which is a contradiction.
Therefore $\{s_n+z_n\}_{n\in\mathbb{N}}$ is bounded from above, and thus $s_n\to-\infty$ as $n\to\infty$.
Noting that $\mathop{\lim}\limits_{s\to-\infty}U_i(x,s)=0$ uniformly in $x\in\mathbb R$ and $\xi_i\ge0$, then
$$-\frac{1}{\psi_i(x_*)}=\lim_{n\to\infty}\frac{U_i(x_n,s_n)-1}{\psi_i(x_n)}
\ge\lim_{n\to\infty}(-1+\tau)\delta_n\ge(-1+\tau)\delta_m,$$
which contradicts the definition of $\delta_m$.
Hence \eqref{U-1} holds, and it follows from \eqref{U-1} that
there exists $z_0\in\mathbb R$ such that \eqref{pre} holds.
The proof is complete.
\end{proof}

\begin{lemma}\label{ss-c}
Assume (H1)-(H8).
Let $\bm u(t,x)=U(x,ct-x\cdot e)$ be a pulsating traveling front of \eqref{m-system}
with $c>c_+^0$.
Then there exists $\delta_c\in(0,\delta_m]$ such that for any $s_0\in\mathbb R$, $\delta\in(0,\delta_c]$ and $\sigma\ge\frac{1}{\beta}$,
the functions $\bm u^\pm(t,x)$ defined by
\begin{equation*}
\begin{aligned}
\bm u^\pm(t,x)&=U(x,ct-x\cdot e+s_0\pm\sigma(1-e^{-\beta t}))
\pm\delta\bm\xi(x,ct-x\cdot e+s_0+z_0\pm\sigma(1-e^{-\beta t}))e^{-\beta t}
\end{aligned}
\end{equation*}
are super- and subsolutions of \eqref{m-system} in $(0,\infty)\times\mathbb R^N$,
respectively, where $z_0$ is given by Lemma \ref{z-0}.
\end{lemma}
\begin{proof}
We only prove that $\bm u^-$ is a subsolution since the other one can be proved similarly.
Let $\hat{s}=ct-x\cdot e+s_0-\sigma(1-e^{-\beta t})$ and $\check{s}=\hat{s}+z_0$.
Then
\begin{equation*}
\bm u^-(t,x)=U(x,\hat{s})-\delta\bm\xi(x,\check{s})e^{-\beta t}
=\bm u\left(\frac{\hat{s}+x\cdot e}{c},x\right)-\delta\bm\xi(x,\check{s})e^{-\beta t}.
\end{equation*}
For each $i$, by a direct calculation, we have
\begin{align*}
&\quad\frac{\partial u_i^-(t,x)}{\partial t}-d_i(x)\Delta u_i^--q_i(x)\cdot\nabla u_i^--f_i(x,\bm u^-)\\
&=f_i(x,\bm u)-f_i(x,\bm u^-)+\delta\beta\xi_ie^{-\beta t}
-\frac{\sigma\beta}{c}\frac{\partial u_i}{\partial t}e^{-\beta t}\\
&\quad-\delta e^{-\beta t}\left\{\chi e^{(\lambda_c+\epsilon)\check{s}}
[-d_i\Delta\phi_i^{\epsilon}-(q_i-2d_i(\lambda_c+\epsilon)e)\cdot\nabla\phi_i^{\epsilon}\right .\\
&\left .\hspace{3.8cm}-\left(d_i(\lambda_c+\epsilon)^2+(\lambda_c+\epsilon)q_i\cdot e
+c(\lambda_c+\epsilon)\right)\phi_i^{\epsilon}]\right .\\
&\left .\hspace{2cm}-(1-\chi)[d_i\Delta\psi_i+q_i\cdot\nabla\psi_i]\right .\\
&\left .\hspace{2cm}+e^{(\lambda_c+\epsilon)\check{s}}
[\chi^\prime(c-\sigma\beta e^{-\beta t})\phi_i^{\epsilon}
-\chi(\lambda_c+\epsilon)\sigma\beta e^{-\beta t}\phi_i^{\epsilon} \right .\\
&\left .\hspace{3.8cm}+2d_i\chi^\prime\nabla\phi_i^{\epsilon}\cdot e
-d_i\chi^{\prime\prime}\phi_i^{\epsilon}
-2d_i\chi^{\prime}(\lambda_c+\epsilon)\phi_i^{\epsilon}
+q_i\cdot e\chi^\prime\phi_i^{\epsilon}] \right .\\
&\left .\hspace{2cm}-\chi^\prime(c-\sigma\beta e^{-\beta t})\psi_i
-2d_i\chi^\prime\nabla\psi_i\cdot e
+d_i\chi^{\prime\prime}\psi_i-q_i\cdot e\chi^\prime\psi_i\right\}\\
&=\sum_{j=1}^{i-1}a_{ij}(\delta\xi_je^{-\beta t})+(h_i(x,\bm u)-h_i(x,\bm u^-))u_i
+\delta\xi_ie^{-\beta t}h_i(x,\bm u^-)+\delta\beta\xi_ie^{-\beta t}
-\frac{\sigma\beta}{c}\frac{\partial u_i}{\partial t}e^{-\beta t}\\
&\quad-\delta e^{-\beta t}\left\{\chi e^{(\lambda_c+\epsilon)\check{s}}
\left[\left(c(\lambda_c+\epsilon)-\kappa_1(\lambda_c+\epsilon)\right)\phi_i^{\epsilon}
+\sum_{j=1}^{i-1}a_{ij}\phi_j^{\epsilon}+h_i(x,\bm 0)\phi_i^{\epsilon}\right]\right .\\
&\left .\hspace{2.2cm}-(1-\chi)\left[\mu^-\psi_i
-\sum_{k=1}^m\frac{\partial f_i}{\partial u_k}(x,\bm 1)\psi_k\right]+R(x,\check{s})\right\}\\
&=-\delta e^{-\beta t}\left\{-\sum_{j=1}^{i-1}a_{ij}\xi_j
-u_i\left[\sum_{k=1}^m{\left(\int_0^1{\frac{\partial h_i}{\partial u_k}
(x,s\bm u+(1-s)\bm u^-)}ds\right)}\xi_k\right]-\xi_ih_i(x,\bm u^-)\right .\\
&\left .\hspace{2cm}+\frac{\sigma\beta}{\delta}\frac{\partial U_i}{\partial s}
+\chi e^{(\lambda_c+\epsilon)\check{s}}
\left[|\sigma_{\epsilon}|\phi_i^{\epsilon}+\sum_{j=1}^{i-1}a_{ij}\phi_j^{\epsilon}
+h_i(x,\bm 0)\phi_i^{\epsilon}-\beta\phi_i^{\epsilon}\right] \right .\\
&\left .\hspace{2cm}-(1-\chi)\left[\mu^-\psi_i-\sum_{k=1}^m
\frac{\partial f_i}{\partial u_k}(x,\bm 1)\psi_k+\beta\psi_i\right]
+R(x,\check{s})\right\}\\
&=-\delta e^{-\beta t}\left\{\frac{\sigma\beta}{\delta}\frac{\partial U_i}{\partial s}
+\chi e^{(\lambda_c+\epsilon)\check{s}}
\left[\beta\phi_i^{\epsilon}+(h_i(x,\bm 0)-h_i(x,\bm u^-))\phi_i^{\epsilon}
-u_i\sum_{k=1}^m{h_{i,k}(x,\hat{s};\delta)\phi_k^{\epsilon}}\right]\right .\\
&\left .\hspace{2cm}+(1-\chi)\left[-(\mu^-+\beta)\psi_i
+\sum_{k=1}^m\frac{\partial f_i}{\partial u_k}(x,\bm 1)\psi_k
-u_i\sum_{k=1}^m{h_{i,k}(x,\hat{s};\delta)\psi_k}\right . \right .\\
&\left .\left .\hspace{3.8cm}-\sum_{j=1}^{i-1}a_{ij}\psi_j
-h_i(x,\bm u^-)\psi_i\right]+R(x,\check{s})\right\}\\
&=-\delta e^{-\beta t}\left\{\frac{\sigma\beta}{\delta}\frac{\partial U_i}{\partial s}
+\chi e^{(\lambda_c+\epsilon)\check{s}}
\left[\beta\phi_i^{\epsilon}+(h_i(x,\bm 0)-h_i(x,\bm u^-))\phi_i^{\epsilon}
-u_i\sum_{k=1}^m{h_{i,k}(x,\hat{s};\delta)\phi_k^{\epsilon}}\right]\right .\\
&\left .\hspace{2cm}+(1-\chi)\Bigg[-(\mu^-+\beta)\psi_i
+\sum_{k=1}^m\left(\frac{\partial h_i}{\partial u_k}(x,\bm 1)
-u_ih_{i,k}(x,\hat{s};\delta)\right)\psi_k \right .\\
&\left .\hspace{3.9cm}+(h_i(x,\bm 1)-h_i(x,\bm u^-))\psi_i\Bigg]+R(x,\check{s})\right\},
\end{align*}
where $a_{ij}\equiv0$ if $i=1$, and
$$h_{i,k}(x,\hat{s};\delta):=\int_0^1{\frac{\partial h_i}{\partial u_k}(x,s\bm u+(1-s)\bm u^-)}ds
=\int_0^1{\frac{\partial h_i}{\partial u_k}\left(x,\bm u-(1-s)\delta\bm\xi e^{-\beta t}\right)}ds,$$
and
\begin{align*}
R(x,\check{s})&:=e^{(\lambda_c+\epsilon)\check{s}}
[\chi^\prime(c-\sigma\beta e^{-\beta t})\phi_i^{\epsilon}
-\chi(\lambda_c+\epsilon)\sigma\beta e^{-\beta t}\phi_i^{\epsilon}\\
&\hspace{2cm}+2d_i\chi^\prime\nabla\phi_i^{\epsilon}\cdot e
-d_i\chi^{\prime\prime}\phi_i^{\epsilon}
-2d_i\chi^{\prime}(\lambda_c+\epsilon)\phi_i^{\epsilon}+q_i\cdot e\chi^\prime\phi_i^{\epsilon}]\\
&\qquad-\chi^\prime(c-\sigma\beta e^{-\beta t})\psi_i
-2d_i\chi^\prime\nabla\psi_i\cdot e+d_i\chi^{\prime\prime}\psi_i-q_i\cdot e\chi^\prime\psi_i.
\end{align*}
Let
\begin{align*}
\Gamma_0^i(x,\hat{s};\delta)&=|h_i(x,\bm 0)-h_i(x,\bm u^-)|
+\sum_{k=1}^m|u_ih_{i,k}(x,\hat{s};\delta)|,\\
\Gamma_1^i(x,\hat{s};\delta)&=|h_i(x,\bm 1)-h_i(x,\bm u^-)|
+\sum_{k=1}^m\left|\frac{\partial h_i}{\partial u_k}(x,\bm 1)-u_ih_{i,k}(x,\hat{s};\delta)\right|.
\end{align*}
Noting that
$\mathop{\lim}\limits_{\hat{s}\to-\infty}\Gamma_0^i(x,\hat{s};\delta)=0$ uniformly
in $x\in\mathbb R^N$.
In view of Theorem \ref{Asy} and Corollary \ref{U-s-estimate}, there exists $M_0>0$ such that
\begin{align*}
&\frac{\partial U_i(x,s)}{\partial s}\ge\frac{\underline{\lambda}_i}{2}U_i(x,s)
\ge\frac{\underline{\lambda}_i\rho}{4}e^{\lambda_c s}\phi_i^c(x),
\quad\forall\;(x,s)\in\mathbb R^N\times(-\infty,-M_0],
\end{align*}
and
\begin{equation}\label{m0}
e^{-\epsilon M_0}\le\min_{i\in I}\left\{\frac{\underline{\lambda}_i\rho\min_x\phi_i^c(x)}
{4\delta_m(\lambda_c+\epsilon)\max_x\phi_i^{\epsilon}(x)}\right\}.
\end{equation}
Moreover,
since $\mathop{\lim}\limits_{\hat{s}\to+\infty}U(x,\hat{s})=\bm 1$ uniformly
in $x\in\mathbb R^N$, there exists $M_1\ge M_0$ such that
\begin{align*}
0\le\Gamma_1^i(x,\hat{s};\delta)&\le\sum_{k=1}^m
\left|\frac{\partial h_i}{\partial u_k}(x,s\bm 1+(1-s)\bm u^-)\right|
|1-u_k+\delta\psi_k e^{-\beta t}|\\
&\quad+\sum_{k=1}^m\left|\int_0^1{\frac{\partial h_i}{\partial u_k}(x,\bm 1)
-\frac{\partial h_i}{\partial u_k}
\left(x,\bm u-(1-s)\delta\bm\xi e^{-\beta t}\right)}ds\right|\\
&\quad+|1-u_i|\sum_{k=1}^m\int_0^1{\left|\frac{\partial h_i}{\partial u_k}
\left(x,\bm u-(1-s)\delta\bm\xi e^{-\beta t}\right)\right|}ds\\
&\le K_1\delta(1+|\bm{\Psi}|)+K_2\delta(1+2|\bm{\Psi}|)+\delta mK_1\\
&\le\delta K
\end{align*}
for all $(x,\hat{s})\in\mathbb R^N\times[M_1,\infty)$,
where $K=K_1(1+m+|\bm{\Psi}|)+K_2(1+2|\bm{\Psi}|)$, with
\begin{align*}
K_1&=\max_{k\in I}\left\{\max_{(x,\bm u)\in\mathbb R^N\times[-\theta\bm 1,\theta\bm 1]}
\frac{\partial h_i}{\partial u_k}(x,\bm u)\right\},\\
K_2&=\max_{k,l\in  I}\left\{\max_{(x,\bm u)\in\mathbb R^N\times[-\theta\bm 1,\theta\bm 1]}
\frac{\partial^2h_i}{\partial u_k\partial u_l}(x,\bm u)\right\},\quad
\theta=1+\delta_m|\bm{\Psi}|.
\end{align*}
Therefore, there exist $M\ge M_1$ with $M>\bar{s}$ and $-M<\underline{s}$,
and $\delta_0\le\delta_m$ such that for any $0<\delta\le\delta_0$, we have
\begin{align*}
\Gamma_0^i(x,\hat{s};\delta)|\bm{\Phi}_{\lambda_c+\epsilon}|
&\le\frac{\beta}{2}\min_{k\in I}\left\{\min_{x\in\mathbb R^N}\phi_k^{\epsilon}(x)\right\},
\qquad\forall\;(x,\hat{s})\in\mathbb R^N\times(-\infty,-M],\\
\Gamma_1^i(x,\hat{s};\delta)|\bm{\Psi}|
&\le\frac{\beta}{2}\min_{k\in I}\left\{\min_{x\in\mathbb R^N}\psi_k(x)\right\},
\qquad\forall\;(x,\hat{s})\in\mathbb R^N\times[M,+\infty).
\end{align*}
Consequently:

(i) For any $(x,\hat{s})\in\mathbb R^N\times(-\infty,-M]$, we have
\begin{align*}
&\quad\frac{\partial u_i^-(t,x)}{\partial t}-d_i(x)\Delta u_i^--q_i(x)\cdot\nabla u_i^--f_i(x,\bm u^-)\\
&\le-\delta e^{-\beta t}\left\{\frac{\sigma\beta}{\delta}
\frac{\underline{\lambda}_i\rho}{4}e^{\lambda_c s}\phi_i^c
+e^{(\lambda_c+\epsilon)\check{s}}\frac{\beta}{2}\phi_i^{\epsilon}
-(\lambda_c+\epsilon)\sigma\beta e^{-\beta t}
e^{(\lambda_c+\epsilon)\check{s}}\phi_i^{\epsilon}\right\}\\
&\le-\sigma\beta e^{-\beta t}e^{\lambda_c\check{s}}
\left\{\frac{\underline{\lambda}_i\rho}{4}\phi_i^c
-\delta(\lambda_c+\epsilon)e^{-\beta t}e^{\epsilon\check{s}}\phi_i^{\epsilon}\right\}\\
&\le0,
\end{align*}
where the last inequality follows from \eqref{m0}.

(ii) For any $(x,\hat{s})\in\mathbb R^N\times[M,+\infty)$, we have
\begin{align*}
&\quad\frac{\partial u_i^-(t,x)}{\partial t}-d_i(x)\Delta u_i^--q_i(x)\cdot\nabla u_i^--f_i(x,\bm u^-)\\
&\le-\delta e^{-\beta t}
\left[-(\mu^-+\beta)\psi_i-\Gamma_1^i(x,\hat{s};\delta)|\bm{\Psi}|\right]\\
&\le-\delta e^{-\beta t}\psi_i
\left[-(\mu^-+\beta)-\frac{\beta}{2}\right]\\
&\le0,
\end{align*}
where we used the fact that $\beta=\frac{|\sigma_\epsilon|}{2}\le\frac{|\mu^-|}{2}$.

(iii) For any $(x,\hat{s})\in\mathbb R^N\times[-M,M]$, let
$$\triangle_i(x,\hat{s})=e^{(\lambda_c+\epsilon)\check{s}}
\Gamma_0^i(x,\hat{s};\delta)|\bm{\Phi}_{\lambda_c+\epsilon}|
+\Gamma_1^i(x,\hat{s};\delta)|\bm{\Psi}|+|R(x,\check{s})|,$$
and define
\begin{align*}
\alpha_i&=\left\{\inf_{(x,s)\in\mathbb R^N\times[-M,M]}
\frac{\partial U_i(x,s)}{\partial s}\right\}>0,\\
\delta_c&=\min\left\{\delta_m, \delta_0,
\frac{\alpha_i}{\mathop{\sup}\limits_{(x,\hat{s})
\in\mathbb R^N\times[-M,M]}|\triangle_i(x,\hat{s})|}\right\}>0.
\end{align*}
Noting that $\sigma\beta\ge 1$, then
\begin{align*}
&\quad\frac{\partial u_i^-(t,x)}{\partial t}-d_i(x)\Delta u_i^--q_i(x)\cdot\nabla u_i^--f_i(x,\bm u^-)\\
&\le-\delta e^{-\beta t}\left\{\frac{\sigma\beta}{\delta}\alpha_i
-e^{(\lambda_c+\epsilon)\check{s}}\Gamma_0^i(x,\hat{s};\delta)|\bm{\Phi}_{\lambda_c+\epsilon}|
-\Gamma_1^i(x,\hat{s};\delta)|\bm{\Psi}|-|R(x,\check{s})|\right\}\\
&\le-\sigma\beta e^{-\beta t}\left(\alpha_i-\delta_c\triangle_i\right)\\
&\le0.
\end{align*}
By (i)-(iii), we conclude that $\bm u^-$ is a subsolution of \eqref{m-system} in $(0,\infty)\times\mathbb R^N$.
Using a similarly argument, one can prove that $\bm u^+$ is a subsolution in $(0,\infty)\times\mathbb R^N$. The proof is complete.
\end{proof}


In the following of this subsection, for any $s_0\in\mathbb R$, we denote
\begin{align*}
\bm u^\pm_\sigma(t,x,s_0)&=U(x,ct-x\cdot e+s_0\pm\sigma(1-e^{-\beta t}))
\pm\delta_c\bm{\xi}(x,ct-x\cdot e+s_0+z_0\pm\sigma(1-e^{-\beta t}))e^{-\beta t}.
\end{align*}

\begin{lemma}\label{sss}
Assume (H1)-(H8). Let $\bm 0<\bm u_0<\bm 1$ satisfy
\eqref{u0v0} for some $\varepsilon_0\in(0,\frac{\delta_c}{2\delta_M})$
and \eqref{u0v0-behavior} with $\tau=0$, where $\delta_c>0$ is given in Lemma \ref{ss-c}.
Then there exist $s_0\in\mathbb R$, $\sigma_c\ge1$ and $t_c>0$ such that for any $\sigma\ge\sigma_c$,
\begin{equation*}
\bm u^-_\sigma(t,x,s_0)\le\bm u(t,x;\bm u_0)\le\bm u^+_\sigma(t,x,s_0),
\quad\forall\;(t,x)\in[t_c,\infty)\times\mathbb R^N.
\end{equation*}
\end{lemma}
\begin{proof}
Let $T(t)={\rm diag}(T_i(t))_{i\in I}$ be the operator defined on $Y$,
and $T_i(t)$ is the linear semigroup generated by
$w_t=d_i(x)\Delta w+q_i(x)\cdot\nabla w$, $i\in I$.
Then
\begin{equation*}
\bm u(t,x;\bm u_0)=T(t)\bm u_0+\int_0^t{T(t-s)\bm F(x,\bm u(s,x;\bm u_0))ds},\quad t>0.
\end{equation*}
Noting that $\lim_{t\to0}|\bm u(t,\cdot;\bm u_0)-\bm u_0|=0$, and for any $t>0$, there hold
\begin{align*}
\liminf_{\varsigma\to+\infty}\inf_{\substack{x\in\mathbb R^N\\-x\cdot e\ge\varsigma}}
(\bm u(t,x;\bm u_0)-\bm 1)
\ge\liminf_{\varsigma\to+\infty}\inf_{\substack{x\in\mathbb R^N\\-x\cdot e\ge\varsigma}}
(\bm u(t,x;\bm u_0)-\bm u_0(x))
+\liminf_{\varsigma\to+\infty}\inf_{\substack{x\in\mathbb R^N\\-x\cdot e\ge\varsigma}}
(\bm u_0(x)-\bm 1).
\end{align*}
Since there exists $\gamma>1$ such that
$0\le\gamma\varepsilon_0\le\frac{\delta_c}{2\delta_M}$,
it follows from \eqref{u0v0} that there exists $t_c>0$ such that
\begin{align*}
\liminf_{\varsigma\to+\infty}\left\{\inf_{\substack{x\in\mathbb R^N\\-x\cdot e\ge\varsigma}}
\frac{\bm u(t_c,x;\bm u_0)-\bm 1}{\bm{\Psi}(x)}\right\}
\gg-\delta_M\gamma\varepsilon_0 e^{-\beta t_c}\bm 1.
\end{align*}
It then follows from Lemma \ref{z-0} that
\begin{equation}\label{t-c}
\begin{aligned}
\sup_{(x,s)\in\mathbb R^N\times\mathbb R}
\frac{U(x,s)-\delta_c\bm\xi(x,s+z_0)e^{-\beta t_c}-\bm 1}{\bm{\Psi}(x)}
&\ll\liminf_{\varsigma\to+\infty}\left\{\inf_{\substack{x\in\mathbb R^N\\-x\cdot e\ge\varsigma}}
\frac{\bm u(t_c,x;\bm u_0)-\bm 1}{\bm{\psi}(x)}\right\}.
\end{aligned}
\end{equation}
By Theorem \ref{Asy}, there exists $s_0=s_0(k)$ such that
\begin{equation}\label{s0-c}
\limsup_{\varsigma\to-\infty}\left\{\sup_{\substack{x\in\mathbb R^N\\-x\cdot e\le\varsigma}}
\left|\frac{U(x,-x\cdot e+s_0)}{ke^{-\lambda_c(x\cdot e)}
\bm{\Phi}_{\lambda_c}(x)}-\bm 1\right|\right\}=0.
\end{equation}

Next, we prove that there exists $\sigma_1\ge1$ such that
$$\bm u^-_\sigma(t_c,x,s_0)\le\bm u(t_c,x;\bm u_0),
\quad\forall\;x\in\mathbb R^N,\quad\forall\;\sigma\ge\sigma_1.$$
Assume this is not true, then there exist $\{x_n\}_{n\in\mathbb N}$
and $\{\sigma_n\}_{n\in\mathbb N}$ such that
$$\sigma_n\to+\infty \ (n\to\infty),
\quad u^-_{i,\sigma_n}(t_c,x_n,s_0)>u_i(t_c,x_n;\bm u_0)$$
for some $i\in I$.
Let $s_n=ct_c-x_n\cdot e+s_0-\sigma_n(1-e^{-\beta t_c})$.
If $\{s_n\}_{n\in\mathbb{N}}$ is bounded from below,
then $-x_n\cdot e\to\infty$ as $n\to\infty$,
and it follows from \eqref{t-c} that
\begin{align*}
\sup_{n\in\mathbb N}\frac{u^-_{i,\sigma_n}(t_c,x_n,s_0)-1}{\psi_i(x_n)}
&\le\sup_{n\in\mathbb N}\frac{U_i(x_n,s_n)
-\delta_c\xi_i(x_n,s_n+z_0)e^{-\beta t_c}-1}{\psi_i(x_n)}\\
&<\liminf_{n\to\infty}\left\{\frac{u_i(t_c,x_n;\bm u_0)-1}{\psi_i(x_n)}\right\},
\end{align*}
which is a contradiction. Therefore $s_n\to-\infty$ as $n\to\infty$.
Noting that
$$0=\mathop{\lim}\limits_{n\to\infty}u^-_{i,\sigma_n}(t_c,x_n,s_0)
\ge\mathop{\lim}\limits_{n\to\infty}u_i(t_c,x_n;\bm u_0)\ge0,$$
then $\mathop{\lim}\limits_{n\to\infty}u_i(t_c,x_n;\bm u_0)=0$.
Now if $\{x_n\cdot e\}_{n\in\mathbb{N}}$ is bounded, we write $x_n=x_n^\prime+x_n^{\prime\prime}$, where $x_n^\prime\in \mathcal L$ and $x_n^{\prime\prime}\in\overline{\mathcal D}$ with
$x_n^{\prime\prime}\to x_\infty\in\overline{\mathcal D}$ as $n\to\infty$.
Let
$$\bm u_n(t,x)=\bm u(t,x+x_n^\prime;\bm u_0).$$
Then $\bm u_n$ solves \eqref{initial-P}, and $\bm 0\le\bm u_n\le\bm 1$ for any $(t,x)\in[0,\infty)\times\mathbb R^N$.
Up to an extraction of subsequence, we assume that $\bm u_n$ converges locally uniformly in $(0,\infty)\times\mathbb R^N$ to $\bm u_\infty\ge\bm 0$, which solves \eqref{initial-P},
and in particular, $u_{i,\infty}$ satisfies
$$\frac{\partial u_{i,\infty}(t,x)}{\partial t}-d_i(x)\Delta u_{i,\infty}
-q_i(x)\cdot\nabla u_{i,\infty}-
\left(\int_0^1{\frac{\partial f_i}{\partial u_i}(x,\tau \bm u_\infty)}d\tau\right)u_{i,\infty}
\ge 0.$$
Observe that $u_{i,\infty}(t_c,x_\infty)=\mathop{\lim}\limits_{n\to\infty}u_{i,n}(t_c,x_n^{\prime\prime})
=\mathop{\lim}\limits_{n\to\infty}u_i(t_c,x_n;\bm u_0)=0$.
It then follows from the maximum principle that $u_{i,\infty}(t,x)\equiv0$ for any $(t,x)\in(0,t_c]\times\mathbb R^N$.
On the other hand, since $\{x_n\cdot e\}_{n\in\mathbb{N}}$ is bounded, so is $\{x_n^\prime\cdot e\}_{n\in\mathbb{N}}$,
and then we obtain from \eqref{u0v0} that
$$\liminf_{\varsigma\to+\infty}\left\{\inf_{x\in\mathbb R^N,\ -x\cdot e\ge\varsigma}
\bm u_\infty(0,x)\right\}\ge(1-\varepsilon_0)\bm 1.$$
Hence there exists
$\hat{x}\in B_{\frac{R}{2}}(0):=\left\{x\in\mathbb R^N: |x|\le\frac{R}{2}\right\}$
with $R>0$ such that
$$\bm u_\infty(0,\hat{x})
\ge\frac{(1-\varepsilon_0)\bm 1}{2}\gg\bm 0.$$
Notice that $\bm u_\infty\ge\bm 0$ for any $(t,x)\in(0,t_c+1)\times\mathring{B}_{R}(0)$,
and $\bm u_\infty(0,x)>\bm 0$ for any $x\in B_{R}(0)$,
where $\mathring{B}_{R}(0):=\left\{x\in\mathbb R^N: |x|<R\right\}$.
Then $\bm u_\infty(t_c,x)\gg\bm 0$ for any $x\in B_{\frac{R}{2}}(0)$
due to the maximum principle, which is a contradiction since $u_{i,\infty}\equiv0$ in $(0,t_c]\times\mathbb R^N$.
As a result, $-x_n\cdot e\to-\infty$ as $n\to\infty$.
Consequently,
$$0=\limsup_{n\to\infty}\frac{u^-_{i,\sigma_n}(t_c,x_n,s_0)}{U_i(x_n,ct_c-x_n\cdot e+s_0)}
\ge\liminf_{n\to\infty}\frac{u_i(t_c,x_n;\bm u_0)}{U_i(x_n,ct_c-x_n\cdot e+s_0)}=1,$$
where the left-hand equality follows from Theorem \ref{Asy}
and the right-hand equality follows from Lemma \ref{asy}.
This contradiction shows that there exists $\sigma_1\ge1$ such that
$\bm u^-_\sigma(t_c,x,s_0)\le \bm u(t_c,x;\bm u_0)$ for any $x\in\mathbb R^N$
and $\sigma\ge\sigma_1$.

Now, we prove that there exists $\sigma_2\ge1$ such that
$\bm u(t_c,x;\bm u_0)\le \bm u^+_\sigma(t_c,x,s_0)$ for any $x\in\mathbb R^N$ and $\sigma\ge\sigma_2$.
Again we argue by a contradiction. If this is not true, then there exist $\{x_n\}_{n\in\mathbb N}$ and $\{\sigma_n\}_{n\in\mathbb N}$ such that
$\sigma_n\to+\infty$ as $n\to\infty$ and
$u^+_{j,\sigma_n}(t_c,x_n,s_0)<u_j(t_c,x_n;\bm u_0)$ for some $j\in I$.
Denote $z_n=ct_c-x_n\cdot e+s_0+\sigma_n(1-e^{-\beta t_c})$.
If $z_n\to+\infty$ as $n\to\infty$, then
\begin{equation*}
1\ge\limsup_{n\to\infty}u_j(t_c,x_n;\bm u_0)
\ge\liminf_{n\to\infty}u^+_{j,\sigma_n}(t_c,x_n,s_0)
\ge1+\delta_c e^{-\beta t_c}\min_{x}\psi_j(x)>1.
\end{equation*}
This contradiction shows that $\{z_n\}_{n\in\mathbb{N}}$ is bounded from above.
Hence $-x_n\cdot e\to-\infty$ as $n\to\infty$, and it follows from Lemma \ref{asy} that
$$0\le\liminf_{n\to\infty}u^+_{j,\sigma_n}(t_c,x_n,s_0)
\le\limsup_{n\to\infty}u_j(t_c,x_n;\bm u_0)
=\lim_{n\to\infty}U_j(x_n,ct_c-x_n\cdot e+s_0)=0,$$
which together with the fact that $\xi_j\ge0$ yields that
$\mathop{\lim}\limits_{n\to\infty}z_n=-\infty$.
Observe that
$$\infty=\liminf_{n\to\infty}\frac{u^+_{j,\sigma_n}(t_c,x_n,s_0)}{U_j(x_n,ct_c-x_n\cdot e+s_0)}
\le\limsup_{n\to\infty}\frac{u_j(t_c,x_n;\bm u_0)}{U(x_n,ct_c-x_n\cdot e+s_0)}=1,$$
this contradiction shows that $\bm u(t_c,x;\bm u_0)\le \bm u^+_\sigma(t_c,x,s_0)$
for any $x\in\mathbb R^N$ and $\sigma\ge\sigma_2$.

To this end, let $\sigma_c=\max\{\sigma_1,\sigma_2\}$,
then for any $\sigma\ge\sigma_c$,
\begin{equation*}
\bm u^-_\sigma(t_c,x,s_0)\le\bm u(t_c,x;\bm u_0)
\le\bm u^+_\sigma(t_c,x,s_0),\quad\forall\;x\in\mathbb R^N.
\end{equation*}
Noting that $\bm u^-_\sigma(t,x,s_0)\le\bm 1$ and
$\bm u^+_\sigma(t,x,s_0)>\bm 0$, and
$\bm 0\le\bm u(t,x;\bm u_0)\le\bm 1$ for all $(t,x)\in[t_c,\infty)\times\mathbb R^N$.
It then follows from the maximum principle that
\begin{equation*}
\bm u^-_\sigma(t,x,s_0)\le\bm u(t,x;\bm u_0)\le\bm u^+_\sigma(t,x,s_0),
\quad\forall\;(t,x)\in[t_c,\infty)\times\mathbb R^N.
\end{equation*}
The proof is complete.
\end{proof}

\begin{lemma}\label{c-s-s}
Suppose that all the assumptions in Lemma \ref{sss} are satisfied.
Let $\bm u(t,x)=U(x,ct-x\cdot e)$ be a pulsating traveling front of \eqref{m-system}
with $c>c_+^0$.
Then for any $\eta>0$, there exist $D_\eta>0$ and $s_\eta\in\mathbb R$ such that
\begin{align*}
U(x,ct-x\cdot e-\eta)-D_\eta e^{(\lambda_c+\epsilon)(ct-x\cdot e)}\bm{\Phi}_{\lambda_c+\epsilon}(x)
\le\bm u(t,x;\bm u_0)
\end{align*}
and
\begin{align*}
\bm u(t,x;\bm u_0)\le U(x,ct-x\cdot e+\eta)
+D_\eta e^{(\lambda_c+\epsilon)(ct-x\cdot e)}\bm{\Phi}_{\lambda_c+\epsilon}(x)
\end{align*}
for any $(t,x)\in\{\mathbb R^+\times\mathbb R^N: \ ct-x\cdot e\le s_\eta\}$.
\end{lemma}
\begin{proof}
Assume without loss of generality that $s_0=0$, where $s_0$ is given by \eqref{s0-c}.
For any $\eta>0$, it follows from \eqref{u0v0} that
\begin{equation*}
\limsup_{\varsigma\to-\infty}\sup_{\substack{x\in\mathbb R^N\\-x\cdot e\le\varsigma}}
\left|\frac{U(x,-x\cdot e-\eta)}{\bm u_0(x)}\right|\le\bm 1.
\end{equation*}
Hence there exists $M>0$ such that $U(x,-x\cdot e-\eta)\le\bm u_0(x)$ for any
$x\in\{\mathbb R^N:\ -x\cdot e\le-M\}$.
Since $\mathop{\inf}\limits_{-x\cdot e\ge-M}\bm u_0(x)\ge\bm 0$,
there exists $D_0=D_0(\eta)>0$ such that for any $D\ge D_0$,
\begin{align*}
U(x,-x\cdot e-\eta)-De^{(\lambda_c+\epsilon)(-x\cdot e)}\bm{\Phi}_{\lambda_c+\epsilon}(x)
\le\bm u_0(x)
\end{align*}
for any $x\in\{\mathbb R^N: -x\cdot e\ge-M\}$. Consequently,
\begin{align*}
U(x,-x\cdot e-\eta)-De^{(\lambda_c+\epsilon)(-x\cdot e)}\bm{\Phi}_{\lambda_c+\epsilon}(x)
\le\bm u_0(x),\quad\forall\;x\in\mathbb R^N,\ \ \forall\;D\ge D_0.
\end{align*}
Let
$$
m_0=\min_{i\in I}\left\{\min_{x\in\mathbb R^N}
\frac{\phi_i^{\epsilon}(x)}{\phi_i^c(x)}\right\},\qquad
m_{\epsilon}=\min_{i\in I}\left\{\min_{x\in\mathbb R^N}\phi_i^{\epsilon}(x)\right\}$$
and
$$K_1=\max_{i,k=1,2,\cdots,m}\left\{\max_{(x,\bm u)\in\mathbb R^N\times[-\bm 1,\bm 1]}
\frac{\partial h_i}{\partial u_k}(x,\bm u)\right\}.$$
Observe that there exists $\tau\in(0,1)$ such that for any $|\bm u|\le\tau$,
\begin{align*}
|\bm u|\le\frac{|\sigma_\epsilon|m_{\epsilon}}{2(K_1+1)|\bm{\Phi}_{\lambda_c+\epsilon}|},
\qquad |h_i(x,\bm u)-h_i(x,\bm 0)|
\le\frac{|\sigma_\epsilon|m_{\epsilon}}{2(K_1+1)|\bm{\Phi}_{\lambda_c+\epsilon}|}.
\end{align*}
By Theorem \ref{Asy}, there exists $s_0=s_0(\eta)<0$ such that
$$U(x,s-\eta)\le\frac{3\rho}{2}e^{\lambda_cs}\bm{\Phi}_{\lambda_c}(x),
\quad\forall\;(x,s)\in\mathbb R^N\times(-\infty,s_0].$$
Let $s_1\le s_0$ be such that
$$3\rho e^{\lambda_cs_1}|\bm{\Phi}_{\lambda_c}|\le\tau,\quad
\frac{3\rho}{m_0}e^{\lambda_cs_1}|\bm{\Phi}_{\lambda_c+\epsilon}|\le\tau,$$
and set
$D_\eta^-=\max\left\{D_0,e^{-(\lambda_c+\epsilon)s_1}\right\}$.
Then $\underline{s}_\eta:=\frac{1}{\epsilon}\ln\frac{3\rho}{2m_0D_\eta^-}\le s_1$
for $s_1$ small enough. Define
\begin{align*}
\underline{\bm u}_\eta(t,x)=\bm u\left(t-\frac{\eta}{c},x\right)
-D_\eta^-e^{(\lambda_c+\epsilon)(ct-x\cdot e)}\bm{\Phi}_{\lambda_c+\epsilon}(x).
\end{align*}
It is easy to see that for all $(t,x)\in\{\mathbb R^+\times\mathbb R^N: ct-x\cdot e\le\underline{s}_\eta\}$,
$$\left|\bm u\left(t-\frac{\eta}{c},x\right)\right|
=|U(x,s-\eta)|\le\frac{3\rho}{2}e^{\lambda_c s}|\bm{\Phi}_{\lambda_c}|\le\frac{\tau}{2}$$
and
$$|\underline{\bm u}_\eta(t,x)|\le\left|\bm u\left(t-\frac{\eta}{c},x\right)\right|
+D_\eta^-e^{(\lambda_c+\epsilon)(ct-x\cdot e)}|\bm{\Phi}_{\lambda_c+\epsilon}|
\le\frac{\tau}{2}+\frac{3\rho}{2m_0}e^{\lambda_c\underline{s}_\eta}|\bm{\Phi}_{\lambda_c+\epsilon}|
\le\tau.$$
Moreover, for any
$(t,x)\in\left\{\mathbb R^+\times\mathbb R^N: ct-x\cdot e=\underline{s}_\eta\right\}$,
\begin{align*}
\underline{\bm u}_\eta(t,x)&=U(x,\underline{s}_\eta-\eta)
-D_\eta^-e^{(\lambda_c+\epsilon)\underline{s}_\eta}\bm{\Phi}_{\lambda_c+\epsilon}(x)\\
&\le\rho e^{\lambda_c\underline{s}_\eta}
\left(\frac{3}{2}\bm{\Phi}_{\lambda_c}(x)-\frac{D_\eta^-}{\rho}e^{\epsilon\underline{s}_\eta}
\bm{\Phi}_{\lambda_c+\epsilon}(x)\right)
\le\bm 0.
\end{align*}
Observe that
$\underline{\bm u}_\eta=(\underline{u}_{1,\eta},\underline{u}_{2,\eta},\cdots,
\underline{u}_{m,\eta})$ satisfies
\begin{align*}
&\quad\frac{\partial\underline{u}_{i,\eta}(t,x)}{\partial t}-d_i(x)\Delta\underline{u}_{i,\eta}
-q_i(x)\cdot\nabla\underline{u}_{i,\eta}-f_i(x,\underline{\bm u}_\eta)\\
&=f_i(x,\bm u)-f_i(x,\underline{\bm u}_\eta)
-D_\eta^-e^{(\lambda_c+\epsilon)s}\left[\sum_{j=1}^{i-1}a_{ij}\phi_j^{\epsilon}
+(h_i(x,\bm 0)+|\sigma_\epsilon|)\phi_i^{\epsilon}\right]\\
&=u_i[h_i(x,\bm u)-h_i(x,\underline{\bm u}_\eta)]
+D_\eta^-e^{(\lambda_c+\epsilon)s}\left(h_i(x,\underline{\bm u}_\eta)
-h_i(x,\bm 0)-|\sigma_\epsilon|\right)\phi_i^{\epsilon}\\
&=D_\eta^-e^{(\lambda_c+\epsilon)s}\left\{u_i\sum_{k=1}^{m}
\left(\int_0^1{\frac{\partial h_i}{\partial u_k}(x,s\bm u+(1-s)
\underline{\bm u}_\eta)ds}\right)\phi_k^{\epsilon}
+(h_i(x,\underline{\bm u}_\eta)-h_i(x,\bm 0))\phi_i^{\epsilon}
-|\sigma_\epsilon|\phi_i^{\epsilon}\right\}\\
&\le D_\eta^-e^{(\lambda_c+\epsilon)s}
\left\{(K_1|\bm u|+|h_i(x,\underline{\bm u}_\eta)-h_i(x,\bm 0)|)|\bm{\Phi}_{\lambda_c+\epsilon}|
-|\sigma_\epsilon| m_{\epsilon}\right\}\\
&\le0, \quad\forall\;(t,x)\in\{\mathbb R^+\times\mathbb R^N: ct-x\cdot e\le \underline{s}_\eta\},
\ \quad i\in I,
\end{align*}
where $a_{ij}\equiv0$ if $i=1$.
Consequently, $\underline{\bm u}_\eta$ is a subsolution of \eqref{m-system}
in $\{\mathbb R^+\times\mathbb R^N: ct-x\cdot e\le \underline{s}_\eta\}$.
It then follows from Lemma \ref{CP} that
$$U(x,s-\eta)-D_\eta^-e^{(\lambda_c+\epsilon)s}\bm{\Phi}_{\lambda_c+\epsilon}(x)
\le\bm u(t,x;\bm u_0)$$
for any $(t,x)\in\{\mathbb R^+\times\mathbb R^N: ct-x\cdot e\le \underline{s}_\eta\}$.

Similarly, it can be shown that there exists $D_1>0$ such that for any $D\ge D_1$, there holds
\begin{align*}
\bm u_0(x)\le U(x,-x\cdot e+\eta)
+De^{(\lambda_c+\epsilon)(-x\cdot e)}\bm{\Phi}_{\lambda_c+\epsilon}(x),
\quad\forall\;x\in\mathbb R^N.
\end{align*}
Observe that
$\mathop{\lim}\limits_{-x\cdot e\to-\infty}\bm u(t,x;\bm u_0)=\bm 0$
uniformly in $t\in[0,t_c]$ by Lemma \ref{asy}, and
\begin{align*}
\bm u(t,x;\bm u_0)&\le U(x,ct-x\cdot e+\sigma_c(1-e^{-\beta t}))+\delta_c e^{(\lambda_c+\epsilon)
(ct-x\cdot e+z_0+\sigma_c(1-e^{-\beta t}))}\bm{\Phi}_{\lambda_c+\epsilon}(x)e^{-\beta t}
\end{align*}
for any
$(t,x)\in\left\{[t_c,\infty)\times\mathbb R^N: ct-x\cdot e\le\underline{s}-z_0-\sigma_c\right\}$
in terms of Lemma \ref{sss},
we only need to consider for each $\eta\le\sigma_c$.
Let
$$M_{\epsilon}=\max_{i\in I}\left\{\max_{x\in\mathbb R^N}\phi_i^{\epsilon}(x)\right\},
\qquad\theta_{\epsilon}=\frac{m_{\epsilon}}{M_{\epsilon}}.$$
In view of Lemma \ref{asy}, there exists $s_2\le0$ such that
$$|\bm u(t,x;\bm u_0)|\le\frac{\tau\theta_{\epsilon}}{2},
\quad |U(x,ct-x\cdot e+\eta)|\le\frac{\tau\theta_{\epsilon}}{2},
\quad\forall\;(t,x)\in\{\mathbb R^+\times\mathbb R^N: ct-x\cdot e\le s_2\}.$$

Let
$$\bar{s}_\eta:=\min\left\{s_2,\frac{1}{\lambda_c+\epsilon}\ln\frac{\tau}{2D_1M_{\epsilon}}\right\},
\qquad D_\eta^+:=
\max\left\{D_1,\frac{\tau e^{-(\lambda_c+\epsilon)\bar{s}_\eta}}{2M_{\epsilon}}\right\},$$
and define
\begin{align*}
\overline{\bm u}_\eta(t,x)=\bm u\left(t+\frac{\eta}{c},x\right)
+D_\eta^+e^{(\lambda_c+\epsilon)(ct-x\cdot e)}\bm{\Phi}_{\lambda_c+\epsilon}(x).
\end{align*}
One can easily obtain that
$|\overline{\bm u}_\eta(t,x)|\le\tau$
for all $(t,x)\in\{\mathbb R^+\times\mathbb R^N: ct-x\cdot e\le\bar{s}_\eta\}$,
and
$$\bm u(t,x;\bm u_0)\le\frac{\tau\theta_\epsilon}{2}\le\overline{\bm u}_\eta(t,x),
\quad\forall\;(t,x)\in\{\mathbb R^+\times\mathbb R^N: ct-x\cdot e=\bar{s}_\eta\}.$$
Furthermore, a similar argument as above shows that
$\overline{\bm u}_\eta$ is a supersolution of \eqref{m-system} in
$\{\mathbb R^+\times\mathbb R^N: ct-x\cdot e\le\bar{s}_\eta\}$.
Lemma \ref{CP} again implies that
$$\bm u(t,x;\bm u_0)\le U(x,s+\eta)+D_\eta^+e^{(\lambda_c+\epsilon)s}\bm{\Phi}_{\lambda_c+\epsilon}(x)$$
for any $(t,x)\in\{\mathbb R^+\times\mathbb R^N: ct-x\cdot e\le\bar{s}_\eta\}$.
Let $s_\eta=\min\{\underline{s}_\eta,\bar{s}_\eta\}$ and $D_\eta=\max\{D_\eta^-,D_\eta^+\}$,
then the proof is complete.
\end{proof}


\subsection{The critical case $c=c_+^0(e)$}

In this subsection, we consider the critical case $c_*=c_+^0(e)$.
Recall that $\lambda_*>0$ is such that $\kappa_1(\lambda_*)=c_*\lambda_*$, and
$$\bm{\Phi}_{\lambda_*}(x)
=\left(\phi_{1,*}(x),\phi_{2,*}(x),\cdots,\phi_{m,*}(x)\right)$$
is the positive periodic eigenfunction of \eqref{Ep} with $\lambda=\lambda_*$
associated with the principal eigenvalue $\kappa=\kappa_1(\lambda_*)$.
Let $\epsilon_*>0$ be a fixed constant satisfying \eqref{e} and
$|\sigma_*|\le\frac{|\mu^-|}{2}$, and
$$\bm{\Phi}_{\lambda_*+\epsilon_*}(x)
=\left(\phi_{1,\epsilon_*}(x),\phi_{2,\epsilon_*}(x),\cdots,\phi_{m,\epsilon_*}(x)\right)$$
be the positive periodic eigenfunction of \eqref{Ep} with $\lambda=\lambda_*+\epsilon_*$
associated with $\kappa=\kappa_1(\lambda_*+\epsilon_*)$,
where
$$\sigma_*=c_*(\lambda_*+\epsilon_*)-\kappa_1(\lambda_*+\epsilon_*)<0.$$

Let $\chi(s)$ be defined by \eqref{chi-s}, and define
\begin{align*}
\bm\xi_*(x,s)&=\chi(s)e^{\lambda_*s}\left(\bm{\Phi}_{\lambda_*}(x)
-e^{\epsilon_*s}\bm{\Phi}_{\lambda_*+\epsilon_*}(x)\right)+(1-\chi(s))\bm{\Psi}(x).
\end{align*}

\begin{lemma}\label{z-0-*}
Assume (H1)-(H8). Let $U(x,c_*t-x\cdot e)$ be the critical pulsating traveling front of \eqref{m-system}. Then there exists $z_0\in\mathbb R$ such that
\begin{equation*}
\sup_{(x,s)\in\mathbb R^N\times\mathbb R}
\frac{U(x,s)-\delta\bm\xi_*(x,s+z_0)-\bm 1}{\bm{\Psi}(x)}\le-\frac{\delta}{2}\bm 1,
\quad\forall\;\delta\in(0,\delta_m].
\end{equation*}
\end{lemma}
\begin{proof}
The proof is similar to that of Lemma \ref{z-0}, we omit it here.
\end{proof}

\begin{lemma}\label{ss-c-*}
Assume (H1)-(H8).
Let $\bm u(t,x)=U(x,c_*t-x\cdot e)$ be the critical pulsating traveling front of \eqref{m-system}.
Then there exists $\delta_*\in(0,\delta_m]$ such that for any $s_0\in\mathbb R$, $\delta\in(0,\delta_*]$ and $\sigma\ge\frac{1}{\beta}$, where $\beta:=|\sigma_*|$,
the functions $\bm u^\pm(t,x)$ defined by
\begin{equation*}
\begin{aligned}
\bm u^\pm(t,x)&=U(x,c_*t-x\cdot e+s_0\pm\sigma(1-e^{-\beta t}))
\pm\delta\bm\xi_*(x,c_*t-x\cdot e+s_0+z_0\pm\sigma(1-e^{-\beta t}))e^{-\beta t}
\end{aligned}
\end{equation*}
are super- and subsolutions of \eqref{m-system} in $(0,\infty)\times\mathbb R^N$,
where $z_0$ is given by Lemma \ref{z-0-*}.
\end{lemma}
\begin{proof}
We only give a sketch here since the proof is similar to that of Lemma \ref{ss-c}.
Let $\hat{s}=c_*t-x\cdot e+s_0-\sigma(1-e^{-\beta t})$ and $\check{s}=\hat{s}+z_0$,
then
$\bm u^-(t,x)=U(x,\hat{s})-\delta\bm\xi_*(x,\hat{s}+z_0)e^{-\beta t}$,
and for each $i$, a direct calculation shows that
\begin{align*}
&\quad\frac{\partial u_i^-(t,x)}{\partial t}-d_i(x)\Delta u_i^--q_i(x)\cdot\nabla u_i^--f_i(x,\bm u^-)\\
&=f_i(x,\bm u)-f_i(x,\bm u^-)-\frac{\sigma\beta}{c}\frac{\partial u_i}{\partial t}e^{-\beta t}\\
&\quad-\delta e^{-\beta t}\left\{\chi e^{\lambda_*\check{s}}
\left[\sum_{j=1}^{i-1}a_{ij}\phi_{j,*}+(h_i(x,\bm 0)-\beta)\phi_{i,*}\right] \right .\\
&\left .\hspace{2.2cm}-\chi e^{(\lambda_*+\epsilon_*)\check{s}}
\left[\sigma_*\phi_{i,\epsilon_*}+\sum_{j=1}^{i-1}a_{ij}\phi_{j,\epsilon_*}
+(h_i(x,\bm 0)-\beta)\phi_{i,\epsilon_*}\right]\right .\\
&\left .\hspace{2.2cm}-(1-\chi)\left[\mu^-\psi_i+\beta\psi_i
-\sum_{k=1}^m\frac{\partial f_i}{\partial u_k}(x,\bm 1)\psi_k\right]+R_*(x,\check{s})\right\}\\
&=-\delta e^{-\beta t}\left\{\frac{\sigma\beta}{\delta}\frac{\partial U_i}{\partial s}
+\chi e^{\lambda_*\check{s}}\left[(h_i(x,\bm 0)-h_i(x,\bm u^-))\phi_{i,*}
-u_i\sum_{k=1}^{m}h_{i,k}(x,\hat{s};\delta)\phi_{k,*}-\beta\phi_{i,*}\right] \right .\\
&\left .\hspace{2cm}-\chi e^{(\lambda_*+\epsilon_*)\check{s}}
\left[(h_i(x,\bm 0)-h_i(x,\bm u^-))\phi_{i,\epsilon_*}
-u_i\sum_{k=1}^{m}h_{i,k}(x,\hat{s};\delta)\phi_{k,\epsilon_*}
-2\beta\phi_{i,\epsilon_*}\right] \right .\\
&\left .\hspace{2cm}+(1-\chi)\Bigg[-(\mu^-+\beta)\psi_i
+\sum_{k=1}^m\left(\frac{\partial f_i}{\partial u_k}(x,\bm 1)
-u_ih_{i,k}(x,\hat{s};\delta)\right)\psi_k\right .\\
&\left .\hspace{3.9cm}+(h_i(x,\bm 1)-h_i(x,\bm u^-))\psi_i\Bigg]+R_*(x,\check{s})\right\},
\end{align*}
where $a_{ij}\equiv0$ if $i=1$, and
$$h_{i,k}(x,\hat{s};\delta)
=\int_0^1{\frac{\partial h_i}{\partial u_k}(x,s\bm u+(1-s)\bm u^-)}ds
=\int_0^1{\frac{\partial h_i}{\partial u_k}\left(x,\bm u-(1-s)\delta\bm\xi_*e^{-\beta t}\right)}ds,
$$
and
\begin{align*}
R_*(x,\check{s})=&\chi^{\prime}e^{\lambda_*\check{s}}
\left[(c-\sigma\beta e^{-\beta t}+q_i\cdot e)(\phi_{i,*}-e^{\epsilon_*\check{s}}\phi_{i,\epsilon_*})
+2d_i(\nabla\phi_{i,*}-e^{\epsilon_*\check{s}}\nabla\phi_{i,\epsilon_*})\cdot e\right]\\
&-\chi^\prime[(c-\sigma\beta e^{-\beta t})\psi_i+2d_i\nabla\psi_i\cdot e+q_i\cdot e\psi_i]+d_i\chi^{\prime\prime}[\psi_i-e^{\lambda_*\check{s}}(\phi_{i,*}
-e^{\epsilon_*\check{s}}\phi_{i,\epsilon_*})]\\
&-\chi e^{\lambda_*\check{s}}\sigma\beta e^{-\beta t}
[\lambda_*\phi_{i,*}-(\lambda_*+\epsilon_*)e^{\epsilon_*\check{s}}\phi_{i,\epsilon_*}].
\end{align*}
Noting from Theorem \ref{Asy} and Corollary \ref{U-s-estimate} that
for each $i$, there exists $M>0$ such that
\begin{align*}
&\frac{\partial U_i(x,s)}{\partial s}\ge\frac{\underline{\lambda}_i}{2}U_i(x,s)
\ge\frac{\underline{\lambda}_i\rho}{4}|s|e^{\lambda_c s}\phi_i^c(x),
\quad\forall\;(x,s)\in\mathbb R^N\times(-\infty,-M].
\end{align*}
By following the same line to the proof of Lemma \ref{ss-c}, one can prove that
$\bm u^-$ is a subsolution of \eqref{m-system} in $(0,\infty)\times\mathbb R^N$.
Similarly, $\bm u^+$ is a supersolution of \eqref{m-system}.
The proof is complete.
\end{proof}

In the following of this subsection, denote
\begin{align*}
\bm u^\pm_\sigma(t,x,s_0)&=U(x,c_*t-x\cdot e+s_0\pm\sigma(1-e^{-\beta t}))
\pm\delta_*\bm\xi_*(x,c_*t-x\cdot e+s_0+z_0\pm\sigma(1-e^{-\beta t}))e^{-\beta t}.
\end{align*}

\begin{lemma}\label{sss-*}
Assume (H1)-(H8). Let $\bm 0<\bm u_0<\bm 1$ satisfy
\eqref{u0v0} for some $\varepsilon_0\in(0,\frac{\delta_*}{2\delta_M})$
and \eqref{u0v0-behavior} with $\tau=1$, where $\delta_*>0$ is given in Lemma \ref{ss-c-*}.
Then there exist $s_0\in\mathbb R$, $\sigma_*\ge1$ and $t_*>0$ such that for any $\sigma\ge\sigma_*$,
\begin{equation*}
\bm u^-_\sigma(t,x,s_0)\le\bm u(t,x;\bm u_0)\le\bm u^+_\sigma(t,x,s_0),
\quad\forall\;(t,x)\in[t_*,\infty)\times\mathbb R^N.
\end{equation*}
\end{lemma}
\begin{proof}
By Theorem \ref{Asy}, there exists $s_0=s_0(k)$ such that
\begin{equation}\label{s0-*}
\limsup_{\varsigma\to-\infty}\left\{\sup_{\substack{x\in\mathbb R^N\\-x\cdot e\le\varsigma}}
\left|\frac{U(x,-x\cdot e+s_0)}{k|x\cdot e|e^{-\lambda_*(x\cdot e)}
\bm{\Phi}_{\lambda_*}(x)}-\bm 1\right|\right\}=0.
\end{equation}
The remaining of the proof is similar to that of Lemma \ref{sss}, we omit the details here.
\end{proof}

\begin{lemma}\label{s-eta-D-eta}
Suppose that all the assumptions in Lemma \ref{sss-*} are satisfied.
Let $\bm u(t,x)=U(x,c_*t-x\cdot e)$ be the critical pulsating traveling front of \eqref{m-system}.
Then for any $\eta>0$, there exist $D_\eta>0$ and $s_\eta\le\bar{s}$ such that
\begin{align*}
U(x,c_*t-x\cdot e-\eta)-D_\eta e^{\lambda_*(c_*t-x\cdot e)}
\left(\bm{\Phi}_{\lambda_*}(x)
-e^{\epsilon_*(c_*t-x\cdot e)}\bm{\Phi}_{\lambda_*+\epsilon_*}(x)\right)
\le\bm u(t,x;\bm u_0)
\end{align*}
and
\begin{align*}
\bm u(t,x;\bm u_0)\le U(x,c_*t-x\cdot e+\eta)+D_\eta e^{\lambda_*(c_*t-x\cdot e)}
\left(\bm{\Phi}_{\lambda_*}(x)
-e^{\epsilon_*(c_*t-x\cdot e)}\bm{\Phi}_{\lambda_*+\epsilon_*}(x)\right)
\end{align*}
for any $(t,x)\in\{\mathbb R^+\times\mathbb R^N\ | \ c_*t-x\cdot e\le s_\eta\}$.
\end{lemma}
\begin{proof}
Assume without loss of generality that $s_0=0$, where $s_0$ is given by \eqref{s0-*}.
Let $s_*\le-1$ be such that $\theta^*e^{\epsilon_*s}\le\frac{1}{2}$ for all $s\le s_*$,
where $\theta^*=\max_{i\in I}\left\{\max_{x\in\mathbb R^N}
\frac{\phi_{i,\epsilon_*}(x)}{\phi_{i,*}(x)}\right\}$.
By following similar arguments to those of Lemma \ref{c-s-s},
one can prove that there exists $D_0(\eta)>0$ such that
\begin{align*}
&U(x,-x\cdot e-\eta)-D_0e^{-\lambda_*(x\cdot e)}
\left(\bm{\Phi}_{\lambda_*}(x)
-e^{-\epsilon_*(x\cdot e)}\bm{\Phi}_{\lambda_*+\epsilon_*}(x)\right)
\le\bm u_0(x)
\end{align*}
for any $x\in\{\mathbb R^N: -x\cdot e\le s_*\}$.
In view of Theorem \ref{Asy}, there exists $s_1\le s_*$ such that
$$\bm 0<U(x,s)\le\frac{3}{2}\rho|s|e^{\lambda_*s}\bm{\Phi}_{\lambda_*}(x),
\quad\forall\;(x,s)\in\mathbb R^N\times(-\infty,s_1].$$
Denote
$$m_{\epsilon_*}=\min_{i\in I}\left\{\min_{x\in\mathbb R^N}\phi_{i,\epsilon_*}(x)\right\},
\quad K_1=\max_{i,k\in I}\left\{\max_{(x,\bm u)\in\mathbb R^N\times[-\bm 1,\bm 1]}
\frac{\partial h_i}{\partial u_k}(x,\bm u)\right\}.$$
Let $\underline{s}_\eta\le s_1$ be such that
$$6\rho K_1|\bm{\Phi}_{\lambda_*}|^2|s|e^{(\lambda_*-\epsilon_*)s}\le|\sigma_*|m_{\epsilon_*},
\quad\forall\; s\le\underline{s}_\eta,$$
and $D_\eta^-:=\max\{D_0, 3\rho|\underline{s}_\eta|\}$.
Define
\begin{align*}
\underline{\bm u}_\eta(t,x)=\bm u\left(t-\frac{\eta}{c_*},x\right)
-D_\eta^-e^{\lambda_*(c_*t-x\cdot e)}
\left(\bm{\Phi}_{\lambda_*}(x)
-e^{\epsilon_*(c_*t-x\cdot e)}\bm{\Phi}_{\lambda_*+\epsilon_*}(x)\right).
\end{align*}
It is easy to see that
$\underline{\bm u}_\eta(t,x)\le\bm 0$ for any
$(t,x)\in\{\mathbb R^+\times\mathbb R^N: c_*t-x\cdot e=\underline{s}_\eta\}$,
and
$\underline{\bm u}_\eta(t,x)\le\bm 1$ for any
$(t,x)\in\{\mathbb R^+\times\mathbb R^N: c_*t-x\cdot e\le\underline{s}_\eta\}$.
Observe that
$\underline{\bm u}_\eta=(\underline{u}_{1,\eta},\underline{u}_{2,\eta},\cdots,
\underline{u}_{m,\eta})\textcolor{blue}{^T}$ satisfies
\begin{align*}
&\quad\frac{\partial\underline{u}_{i,\eta}(t,x)}{\partial t}-d_i(x)\Delta\underline{u}_{i,\eta}
-q_i(x)\cdot\nabla\underline{u}_{i,\eta}-f_i(x,\underline{\bm u}_\eta)\\
&=f_i(x,\bm u)-f_i(x,\underline{\bm u}_\eta)-D_\eta^-
\left[e^{\lambda_*s}\left(\sum_{j=1}^{i-1}a_{ij}\phi_{j,*}+h_i(x,\bm 0)\phi_{i,*}\right)\right .\\
&\left .\hspace{5.2cm}-e^{(\lambda_*+\epsilon_*)s}\left(\sum_{j=1}^{i-1}a_{ij}\phi_{j,\epsilon_*}
+(h_i(x,\bm 0)-|\sigma_*|)\phi_{i,\epsilon_*}\right)\right]\\
&=u_i(h_i(x,\bm u)-h_i(x,\underline{\bm u}_{\eta}))
+D_\eta^-e^{\lambda_*s}(\phi_{i,*}-e^{\lambda_*s}\phi_{i,\epsilon_*})
(h_i(x,\underline{\bm u}_{\eta})-h_i(x,\bm 0))
-D_\eta^-|\sigma_*|e^{(\lambda_*+\epsilon_*)s}\phi_{i,\epsilon_*}\\
&=D_\eta^-e^{\lambda_*s}\left\{u_i\sum_{k=1}^m{\left(\int_0^1{\frac{\partial h_i}{\partial u_k}
(x,s\bm u+(1-s)\underline{\bm u}_{\eta})ds}\right)
(\phi_{k,*}-e^{\lambda_*s}\phi_{k,\epsilon_*})}
-|\sigma_*|e^{\epsilon_*s}\phi_{i,\epsilon_*} \right .\\
&\left .\hspace{2cm}+(\phi_{i,*}-e^{\lambda_*s}\phi_{i,\epsilon_*})
\sum_{k=1}^m{\left(\int_0^1{\frac{\partial h_i}{\partial u_k}
(x,s\underline{\bm u}_{\eta})ds}\right)
(u_k-D_\eta^-e^{\lambda_*s}(\phi_{k,*}-e^{\lambda_*s}\phi_{k,\epsilon_*}))}\right\}\\
&\le D_\eta^-e^{\lambda_*s}\left\{K_1|\bm u||\bm{\Phi}_{\lambda_*}|
-|\sigma_*|m_{\epsilon_*}e^{\epsilon_*s}
+K_1|\bm{\Phi}_{\lambda_*}|(|\bm u|+D_\eta^-|\bm{\Phi}_{\lambda_*}|e^{\lambda_*s})\right\}\\
&\le D_\eta^-e^{(\lambda_*+\epsilon_*)s}
\left\{(3\rho|s|+D_\eta^-)K_1|\bm{\Phi}_{\lambda_*}|^2e^{(\lambda_*-\epsilon_*)s}
-|\sigma_*|m_{\epsilon_*}\right\}\\
&\le  D_\eta^-e^{(\lambda_*+\epsilon_*)s}
\left\{6\rho K_1|\bm{\Phi}_{\lambda_*}|^2|s|e^{(\lambda_*-\epsilon_*)s}
-|\sigma_*|m_{\epsilon_*}\right\}\\
&\le 0, \quad\forall\;(t,x)\in\{\mathbb R^+\times\mathbb R^N: c_*t-x\cdot e\le\underline{s}_\eta\},
\ \quad i=1,2,\cdots,m,
\end{align*}
where $a_{ij}\equiv0$ if $i=1$. It then follows from Lemma \ref{CP} that
\begin{align*}
&U(x,c_*t-x\cdot e-\eta)-D_\eta^-e^{\lambda_*(c_*t-x\cdot e)}
\left(\bm{\Phi}_{\lambda_*}(x)
-e^{\epsilon_*(c_*t-x\cdot e)}\bm{\Phi}_{\lambda_*+\epsilon_*}(x)\right)
\le\bm u(t,x;\bm u_0)
\end{align*}
for any $(t,x)\in\{\mathbb R^+\times\mathbb R^N: \ c_*t-x\cdot e\le \underline{s}_\eta\}$.
Similarly, one can prove that there exist $D_\eta^+>0$ and $\bar{s}_\eta\in\mathbb R$ such that
\begin{align*}
&\bm u(t,x;\bm u_0)\le U(x,c_*t-x\cdot e+\eta)+D_\eta^+e^{\lambda_*(c_*t-x\cdot e)}
\left(\bm{\Phi}_{\lambda_*}(x)
-e^{\epsilon_*(c_*t-x\cdot e)}\bm{\Phi}_{\lambda_*+\epsilon_*}(x)\right)
\end{align*}
for any $(t,x)\in\{\mathbb R^+\times\mathbb R^N: \ c_*t-x\cdot e\le\bar{s}_\eta\}$.
Set $s_\eta=\min\{\underline{s}_\eta,\bar{s}_\eta\}$ and $D_\eta=\max\{D_\eta^-,D_\eta^+\}$,
the proof is then complete.
\end{proof}

\subsection{Stability of pulsating traveling fronts}

We prove the stability of pulsating traveling fronts in this subsection,
which is induced by the following lemma.

\begin{lemma}\label{louville}
Assume (H1)-(H8).
Let $U(x,ct-x\cdot e)$ be a pulsating traveling front of \eqref{m-system} with $c\ge c_+^0$,
and $\bm u(t,x)$ be a solution of \eqref{m-system} such that
$$U(x,ct-x\cdot e+s_0+\underline{s})\le\bm u(t,x)\le U(x,ct-x\cdot e+s_0+\overline{s}),
\quad\forall\;(t,x)\in\mathbb R\times\mathbb R^N$$
for some $\underline{s}\le0\le\overline{s}$ and $s_0\in\mathbb R$.
Moreover, for each $\eta>0$, there exist $D_\eta>0$ and $s_\eta\in\mathbb R$ such that
for any $(t,x)\in\{\mathbb R^+\times\mathbb R^N: \ ct-x\cdot e\le s_\eta\}$,
\begin{align*}
&\quad U(x,ct-x\cdot e-\eta)-D_\eta e^{(\lambda_c+\epsilon)(ct-x\cdot e)}
\bm{\Phi}_{\lambda_c+\epsilon}(x)\\
&\le\bm u(t,x)\\
&\le U(x,ct-x\cdot e+\eta)+D_\eta e^{(\lambda_c+\epsilon)(ct-x\cdot e)}
\bm{\Phi}_{\lambda_c+\epsilon}(x),\quad\text{if}\ \ c>c_*(e),
\end{align*}
and
\begin{align*}
&\quad U(x,c_*t-x\cdot e-\eta)-D_\eta e^{\lambda_*(c_*t-x\cdot e)}
\left(\bm{\Phi}_{\lambda_*}(x)
-e^{\epsilon_*(c_*t-x\cdot e)}\bm{\Phi}_{\lambda_*+\epsilon_*}(x)\right)\\
&\le\bm u(t,x)\\
&\le U(x,c_*t-x\cdot e+\eta)+D_\eta e^{\lambda_*(c_*t-x\cdot e)}
\left(\bm{\Phi}_{\lambda_*}(x)
-e^{\epsilon_*(c_*t-x\cdot e)}\bm{\Phi}_{\lambda_*+\epsilon_*}(x)\right),
\quad\text{if}\ \ c=c_*(e).
\end{align*}
Then
$$\bm u(t,x)=U(x,ct-x\cdot e+s_0),\quad\forall\;(t,x)\in\mathbb R\times\mathbb R^N.$$
\end{lemma}

\begin{proof}
We only prove for the case $c=c_*(e)$ since the other one can be proved similarly.
Assume without loss of generality that $s_0=0$. Define
$$
\bar{\eta}=\inf\left\{\eta\ge0: \ U(x,c_*t-x\cdot e+\eta)\ge\bm u(t,x),
\ \forall\;(t,x)\in\mathbb R\times\mathbb R^N\right\}.
$$
Then $0\le\bar{\eta}\le\overline{s}$.
Assume that $\bar{\eta}>0$, next we argue by a contradiction, which shows that $\bar{\eta}=0$.

{\bf Step 1}. We claim that there exists $\underline{z}\in(-\infty,s_{\bar{\eta}/2}]$ such that
\begin{equation}\label{UV-uv}
U(x,c_*t-x\cdot e+\bar{\eta}/2)\ge\bm u(t,x),
\quad\forall\;(t,x)\in\{\mathbb R\times\mathbb R^N: c_*t-x\cdot e\le\underline{z}\}.
\end{equation}
Indeed, if this is not true, then there exists
$\{(t_n,x_n)\}_{n\in\mathbb N}$ such that
$$s_n:=c_*t_n-x_n\cdot e\to-\infty\ (n\to\infty),
\quad u_i(t_n,x_n)>U_i(x_n,s_n+\bar{\eta}/2)$$
for some $i\in I$. In view of Theorem \ref{Asy},
\begin{align*}
\limsup_{n\to\infty}\frac{U_i(x_n,s_n+\bar{\eta}/4)
+D_{\bar{\eta}/4}e^{\lambda_*s_n}[\phi_{i,*}(x_n)-e^{\epsilon_*s_n}\phi_{i,\epsilon_*}(x_n)]}
{U_i(x_n,s_n+\bar{\eta}/2)}<1,
\end{align*}
which together with the assumption shows that there exists $N$
such that for all $n\ge N$, $s_n=c_*t_n-x_n\cdot e\le s_{\bar{\eta}/4}$, and
\begin{align*}
u_i(t_n,x_n)&\le U_i(x_n,s_n+\bar{\eta}/4)+D_{\bar{\eta}/4}e^{\lambda_*s_n}
\left(\phi_{i,*}(x_n)-e^{\epsilon_*s_n}\phi_{i,\epsilon_*}(x_n)\right)
{\le U_i(x_n,s_n+\bar{\eta}/2)}
\end{align*}
which is a contradiction, and hence \eqref{UV-uv} holds.

{\bf Step 2}. We prove that
\begin{equation}\label{St-1}
U(x,c_*t-x\cdot e+\bar{\eta})\gg\bm u(t,x),
\quad\forall\;(t,x)\in\{\mathbb R\times\mathbb R^N:\ \underline{z}\le c_*t-x\cdot e\le z,
\ \forall\;z\ge\underline{z}\}.
\end{equation}
For any $z\ge\underline{z}$, assume to the contrary that there exists $i\in I$ such that
$$\mathop{\inf}\limits_{\underline{z}\le c_*t-x\cdot e\le z}
\{U_i(x,c_*t-x\cdot e+\bar{\eta})-u_i(t,x)\}=0.$$
Then there exists $\{(t_k,x_k)\}_{k\in\mathbb N}$ such that
$$\underline{z}\le s_k:=c_*t_k-x_k\cdot e\le z,\qquad
\mathop{\lim}\limits_{k\to\infty}\{U_i(x_k,c_*t_k-x_k\cdot e+\bar{\eta})-u_i(t_k,x_k)\}=0.$$
Let $x_k=x_k^\prime+x_k^{\prime\prime}$ with $x_k^\prime\in \mathcal L$ and $x_k^{\prime\prime}\in\overline{\mathcal D}$,
and assume (up to a subsequence) that $s_k\to s_\infty\in[\underline{z},z]$ and
$x_k^{\prime\prime}\to x_\infty\in\overline{\mathcal D}$ as $k\to\infty$.
Let
$$\bm u_k(t,x)=\bm u(t+t_k,x+x_k^\prime),$$
then $\bm u_k$ solves \eqref{m-system} for any $k\in\mathbb N$.
Up to an extraction of a subsequence, $\{\bm u_k\}_{k\in\mathbb{N}}$ converges
uniformly in any compact subset of $\mathbb R\times\mathbb R^N$
to a solution $\bm u_\infty\in[\bm 0,\bm 1]$ of \eqref{m-system}.
Noting that
$$\bm u_\infty(t,x)\le U(x,c_*t-x\cdot e+s_\infty+x_\infty\cdot e+\bar{\eta}),
\quad\forall\;(t,x)\in\mathbb R\times\mathbb R^N$$
by the definition of $\bar{\eta}$. In particular, $u_{i,\infty}(0,x_\infty)=\mathop{\lim}\limits_{k\to\infty}u_i(t_k,x_k)
=U_i(x_\infty,s_\infty+\bar{\eta})$.
Let
$$\bm\omega(t,x)=\bm u^{\bar{\eta}}(t,x)-\bm u_\infty(t,x),
\qquad\bm u^{\bar{\eta}}(t,x):=U(x,c_*t-x\cdot e+s_\infty+x_\infty\cdot e+\bar{\eta}).$$
Then $\bm\omega\ge\bm 0$, and $\omega_i(0,x_\infty)=0$.
Observe that
\begin{align*}
(\omega_i)_t-d_i\Delta\omega_i-q_i\cdot\nabla\omega_i
\ge\left(\int_0^1{\frac{\partial f_i}{\partial u_i}
\left(x,\tau\bm u^{\bar{\eta}}+(1-\tau)\bm u_\infty\right)d\tau}\right)\omega_i.
\end{align*}
It then follows from the maximum principle that
\begin{equation}\label{ui-smp}
u_{i,\infty}(t,x)=U_i(x,c_*t-x\cdot e+s_\infty+x_\infty\cdot e+\bar{\eta}),
\quad\forall\; (t,x)\in(-\infty,0]\times\mathbb R^N.
\end{equation}
On the other hand, it follows from \eqref{UV-uv} that
$$\bm u_k(t,x)\le U(x,c_*t-x\cdot e+s_k+x_k^{\prime\prime}\cdot e+\bar{\eta}/2)$$
for all $(t,x)\in\{\mathbb R\times\mathbb R^N: c_*t-x\cdot e\le\underline{z}-s_k-x_k^{\prime\prime}\cdot e\}$.
By passing limits, one obtains that
$$\bm u_\infty(t,x)\le U(x,c_*t-x\cdot e+s_\infty+x_\infty\cdot e+\bar{\eta}/2)$$
for all $(t,x)\in\{\mathbb R\times\mathbb R^N: c_*t-x\cdot e
\le\underline{z}-s_\infty-x_\infty\cdot e\}$,
which contradicts \eqref{ui-smp}. Hence \eqref{St-1} holds.

{\bf Step 3}. Noting from the assumptions that
$$\limsup_{\varsigma\to\infty}
\left\{\sup_{c_*t-x\cdot e\ge\varsigma}|\bm u(t,x)-\bm 1|\right\}=0.$$
Then there exists $\bar{z}>\underline{z}$ such that
$$(1-\varrho^*)\bm 1\le\bm u(t,x)\ll\bm 1,\qquad
(1-\varrho^*)\bm 1\le U(x,c_*t-x\cdot e)\ll\bm 1$$
for any $(t,x)\in\{\mathbb R\times\mathbb R^N: c_*t-x\cdot e\ge\bar{z}\}$,
where $\varrho^*=\min\{1,\min_{i\in I}\varrho_i\}$ with $\varrho_i$ given by \eqref{varrho}.
In view of \eqref{St-1}, there exists $\eta_0\in[\frac{\bar{\eta}}{2},\bar{\eta})$ such that
\begin{equation}\label{St-2}
U(x,c_*t-x\cdot e+\eta_0)\ge\bm u(t,x),
\quad\forall\;(t,x)\in\{\mathbb R\times\mathbb R^N: \underline{z}\le c_*t-x\cdot e\le\bar{z}\}.
\end{equation}
Next we prove that
\begin{equation}\label{infinities}
U(x,c_*t-x\cdot e+\eta_0)\ge\bm u(t,x),\quad
\forall\;(t,x)\in\{\mathbb R\times\mathbb R^N: c_*t-x\cdot e\ge\bar{z}\}.
\end{equation}
Let
\begin{align*}
\bm u^\theta(t,x)=U(x,c_*t-x\cdot e+\eta_0)+\theta\bm{\Psi}(x)-\bm u(t,x),
\end{align*}
and define
$$\bar{\theta}:=\inf\left\{\theta\ge0\ |\ \bm u^\theta(t,x)\ge\bm 0,
\ \forall\;(t,x)\in\{\mathbb R\times\mathbb R^N: c_*t-x\cdot e\ge\bar{z}\}\right\}.$$
Observe that $\bar{\theta}\ge0$ is well defined, since
$\min_{i\in I}\{\min_{x\in\mathbb R^N}\psi_i(x)\}>0$.
Suppose that $\bar{\theta}>0$, and
$\mathop{\inf}\limits_{c_*t-x\cdot e\ge\bar{z}}u_i^{\bar{\theta}}(t,x)=0$ for some $i\in I$.
Then there exists $\{(t_k,x_k)\}_{k\in\mathbb N}$ such that
$$s_k:=c_*t_k-x_k\cdot e\ge\bar{z},\quad \mathop{\lim}\limits_{k\to\infty}u_i^{\bar{\theta}}(t_k,x_k)=0.$$
Noting that $\mathop{\lim}\limits_{\varsigma\to\infty}\mathop{\inf}
\limits_{c_*t-x\cdot e\ge\varsigma}u_i^{\bar{\theta}}(t,x)
\ge\bar{\theta}\mathop{\min}\limits_{x\in\mathbb R^N}\psi_i(x)>0$,
the sequence $\{s_k\}_{k\in\mathbb{N}}$ is bounded from above.
Writing $x_k=x_k^{\prime}+x_k^{\prime\prime}$ with $x_k^{\prime}\in \mathcal L$ and $x_k^{\prime\prime}\in\overline{\mathcal D}$,
and assume up to an extraction of a subsequence that
$s_k\to s_\infty\ge\bar{z}$ and $x_k^{\prime\prime}\to x_\infty\in\overline{\mathcal D}$
as $k\to\infty$.
Let
\begin{align*}
\bm u_k^{\bar{\theta}}(t,x)&=\bm u^{\bar{\theta}}(t+t_k,x+x_k^{\prime})
=U(x,c_*t-x\cdot e+s_k+x_k^{\prime\prime}\cdot e+\eta_0)
+\bar{\theta}\bm{\Psi}(x)-\bm u(t+t_k,x+x_k^{\prime}).
\end{align*}
Observing that $\bm u_k^{\bar{\theta}}$ is uniformly bounded and nonnegative in
$\{\mathbb R\times\mathbb R^N:\ c_*t-x\cdot e\ge\bar{z}-s_k-x_k^{\prime\prime}\cdot e\}$.
Up to an extraction of a subsequence,
$\{\bm u(\cdot+t_k,\cdot+x_k^{\prime})\}_{k\in\mathbb{N}}$ converges to a solution $\bm u_\infty$
of \eqref{m-system} and $\{\bm u^{\bar{\theta}}_k\}_{k\in\mathbb{N}}$ converges to a function
$\bm u_\infty^{\bar{\theta}}$,
uniformly in any compact subset of $\mathbb R\times\mathbb R^N$. Moreover,
\begin{align*}
\bm u_\infty^{\bar{\theta}}(t,x)&
=U(x,c_*t-x\cdot e+s_\infty+x_\infty\cdot e+\eta_0)+\bar{\theta}\bm{\Psi}(x)-\bm u_\infty(t,x),
\end{align*}
and $\bm u_\infty^{\bar{\theta}}(t,x)\ge\bm 0$ for any
$(t,x)\in\{\mathbb R\times\mathbb R^N: c_*t-x\cdot e\ge\bar{z}-s_\infty-x_\infty\cdot e\}$.
In particular,
$u_{i,\infty}^{\bar{\theta}}(0,x_\infty)
=\mathop{\lim}\limits_{k\to\infty}u_i^{\bar{\theta}}(t_k,x_k)=0$.
In view of \eqref{UV-uv} and \eqref{St-2},
$$\bm u_k^{\bar{\theta}}(t,x)\ge\bar{\theta}\mathop{\min}\limits_{x}\bm{\Psi}(x),
\quad\forall\;(t,x)\in\{\mathbb R\times\mathbb R^N: c_*t-x\cdot e\le\bar{z}-s_k-x_k^{\prime\prime}\cdot e\},$$
and hence $\bm u_\infty^{\bar{\theta}}(t,x)\gg\bm 0$ for any
$(t,x)\in\{\mathbb R\times\mathbb R^N: c_*t-x\cdot e\le\bar{z}-s_\infty-x_\infty\cdot e\}$
by passing the limits.
Therefore $s_\infty>\bar{z}$ since $u_{i,\infty}^{\bar{\theta}}(0,x_\infty)=0$.
Furthermore, it is easy to see that
$$(1-\varrho^*)\bm 1\le\bm u_\infty(t,x)\le\bm 1,
\quad (1-\varrho^*)\bm 1\le U(x,c_*t-x\cdot e+s_\infty+x_\infty\cdot e+\eta_0)\le\bm 1$$
for any $(t,x)\in\{\mathbb R\times\mathbb R^N: c_*t-x\cdot e\ge\bar{z}-s_\infty-x_\infty\cdot e\}$.
By a direct calculation, we have
\begin{align*}
&\quad\frac{\partial u_{i,\infty}^{\bar{\theta}}(t,x)}{\partial t}
-d_i(x)\Delta u_{i,\infty}^{\bar{\theta}}-q_i(x)\cdot\nabla u_{i,\infty}^{\bar{\theta}}\\
&=f_i(x,U)-f_i(x,\bm u_\infty)+\bar{\theta}
\left(\sum_{k=1}^{m}\frac{\partial f_i}{\partial u_k}(x,\bm 1)\psi_k-\mu^-\psi_i\right)\\
&\ge\sum_{k=1}^{m}{
\left(\int_0^1{\frac{\partial f_i}{\partial u_k}(x,\tau U+(1-\tau)\bm u_\infty)d\tau}\right)
(u_{k,\infty}^{\bar{\theta}}-\bar{\theta}\psi_k)}
+\bar{\theta}
\left(\sum_{k=1}^{m}\frac{\partial f_i}{\partial u_k}(x,\bm 1)\psi_k-\mu^-\psi_i\right)\\
&\ge\left(\int_0^1{\frac{\partial f_i}{\partial u_i}(x,\tau U+(1-\tau)\bm u_\infty)d\tau}\right)
u_{i,\infty}^{\bar{\theta}}\\
&\qquad+\bar{\theta}\left\{\sum_{k=1}^{m}
{\left[\int_0^1\left(\frac{\partial f_i}{\partial u_k}(x,\bm 1)
-\frac{\partial f_i}{\partial u_k}(x,\tau U+(1-\tau)\bm u_\infty)\right)d\tau\right]\psi_k}
-\mu^-\psi_i\right\}\\
&\ge\left(\int_0^1{\frac{\partial f_i}{\partial u_i}(x,\tau U+(1-\tau)\bm u_\infty)d\tau}\right)
u_{i,\infty}^{\bar{\theta}}
\end{align*}
for any $(t,x)\in\{\mathbb R\times\mathbb R^N: c_*t-x\cdot e>\bar{z}-s_\infty-x_\infty\cdot e\}$.
The maximum principle then implies that
$u_{i,\infty}^{\bar{\theta}}\equiv0$ for any
$(t,x)\in\{\mathbb R\times\mathbb R^N:\
c_*t-x\cdot e\ge\bar{z}-s_\infty-x_\infty\cdot e\}\cap\{t\le0\}$,
which contradicts the fact that $\bm u_\infty^{\bar{\theta}}(t,x)>\bm 0$
for any $(t,x)\in\{\mathbb R\times\mathbb R^N:\ c_*t-x\cdot e=\bar{z}-s_\infty-x_\infty\cdot e\}$.
Therefore $\bar{\theta}=0$, and hence \eqref{infinities} holds.

To this end, we conclude from \eqref{UV-uv}, \eqref{St-2} and \eqref{infinities} that
$$U(x,c_*t-x\cdot e+\eta_0)\ge\bm u(t,x),\quad\forall\;(t,x)\in\mathbb R\times\mathbb R^N.$$
Recall that $\eta_0\in[\frac{\bar{\eta}}{2},\bar{\eta})$, this contradicts the definition of $\bar{\eta}$. Therefore $\bar{\eta}=0$, and consequently,
$$U(x,c_*t-x\cdot e)\ge\bm u(t,x),\quad\forall\;(t,x)\in\mathbb R\times\mathbb R^N.$$
Similarly, if we define
$$\underline{\eta}=\inf\left\{\eta\ge0:\ \bm u(t,x)\ge U(x,c_*t-x\cdot e-\eta),
\ \forall\;(t,x)\in\mathbb R\times\mathbb R^N\right\},$$
then $0\le\bar{\eta}\le-\underline{s}$.
One can prove by using analogous arguments as above that $\underline{\eta}=0$.
Therefore $\bm u(t,x)\ge U(x,c_*t-x\cdot e)$ for any $(t,x)\in\mathbb R\times\mathbb R^N$.
The proof is complete.
\end{proof}

\begin{theorem}
Assume (H1)-(H8).
Let $U(x,ct-x\cdot e)$ be a pulsating traveling front of \eqref{m-system} with $c\ge c_+^0$,
and $\bm u(t,x;\bm u_0)$ be a solution of \eqref{m-system} with
initial value $\bm{u}(0,\cdot;\bm{u}_0)=\bm{u}_0\in Y_+$.
Assume that $\bm 0<\bm u_0<\bm 1$ satisfies \eqref{u0v0} and \eqref{u0v0-behavior}.
Then there exists $s_0\in\mathbb R$ such that
\begin{equation*}
\lim_{t\to\infty}\sup_{x\in\mathbb R^N}|\bm u(t,x;\bm u_0)-U(x,ct-x\cdot e+s_0)|=0.
\end{equation*}
\end{theorem}
\begin{proof}
We only prove the case $c=c_*$ since the other one can be proved similarly.
Let $s_0\in\mathbb R$ be such that
\begin{align*}
\limsup_{\varsigma\to-\infty}\left\{\sup_{\substack{x\in\mathbb R^N\\-x\cdot e\le\varsigma}}
\left|\frac{U(x,-x\cdot e+s_0)}{k|x\cdot e|e^{-\lambda_*(x\cdot e)}
\bm{\Phi}_{\lambda_*}(x)}-\bm 1\right|\right\}=0.
\end{align*}
Assume without loss of generality that $s_0=0$.
If the statement is not true, then there exist $\alpha>0$ and a sequence
$\{(t_n,x_n)\}_{n\in\mathbb N}$ such that
\begin{equation}\label{alpha}
t_n\to+\infty\ (n\to\infty),\quad
\lim_{n\to\infty}|u_i(t_n,x_n;\bm u_0)-U_i(x_n,c_*t_n-x_n\cdot e)|\ge\alpha
\end{equation}
for some $i\in I$. Denote $s_n=c_*t_n-x_n\cdot e$. If $\{s_n\}_{n\in\mathbb{N}}$ is bounded,
we write $x_n=x_n^{\prime}+x_n^{\prime\prime}$ with $x_n^{\prime}\in \mathcal L$ and
$x_n^{\prime\prime}\in\overline{\mathcal D}$.
Assume up to a subsequence that
$s_n\to s_\infty$ and $x_n^{\prime\prime}\to x_\infty\in\overline{\mathcal D}$ as $n\to\infty$,
and set $t_n^\prime=t_n-t_\infty$, where $t_\infty:=\frac{s_\infty+x_\infty\cdot e}{c_*}$.
Let
$$\bm u_n(t,x)=\bm u(t+t_n^\prime,x+x_n^\prime;\bm u_0),$$
then $\bm u_n$ is a solution of \eqref{initial-P} in
$(t,x)\in(-t_n^\prime,\infty)\times\mathbb R^N$ for each $n$.
Up to a subsequence, we assume that $\{\bm u_n\}_{n\in\mathbb{N}}$ converges to a solution $\bm u_\infty$ of \eqref{initial-P} uniformly in any compact subset of $\mathbb R\times\mathbb R^N$.
In view of Lemma \ref{sss-*},
\begin{align*}
U(x,s_n^\prime-\sigma_*)-\delta_*|\bm\xi_*|e^{-\beta(t+t_n^\prime)}\bm 1
\le\bm u_n(t,x)\le U(x,s_n^\prime+\sigma_*)+\delta_*|\bm\xi_*|e^{-\beta(t+t_n^\prime)}\bm 1
\end{align*}
for any $(t,x)\in[t_*-t_n^\prime,+\infty)\times\mathbb R^N$, where $s_n^\prime:=c_*(t+t_n^\prime)-(x+x_n^\prime)\cdot e$.
By passing the limits and noting that $c_*t_n^\prime-x_n^\prime\cdot e\to 0$ and $t_n^\prime\to+\infty$ as $n\to\infty$,
\begin{align*}
U(x,c_*t-x\cdot e-\sigma_*)&\le\bm u_\infty(t,x)\le U(x,c_*t-x\cdot e+\sigma_*),
\quad\forall\;(t,x)\in\mathbb R\times\mathbb R^N.
\end{align*}
On the other hand, it follows from Lemma \ref{s-eta-D-eta} that for any $\eta>0$, there exist $s_\eta\in\mathbb R$ and $D_\eta>0$ such that
\begin{align*}
&\quad U(x,s_n^\prime-\eta)-D_\eta e^{\lambda_*s_n^\prime}
\left(\bm{\Phi}_{\lambda_*}(x)-e^{\epsilon_*s_n^\prime}\bm{\Phi}_{\lambda_*+\epsilon_*}(x)\right)\\
&\le\bm u_n(t,x)\\
&\le U(x,s_n^\prime+\eta)+D_\eta e^{\lambda_*s_n^\prime}
\left(\bm{\Phi}_{\lambda_*}(x)-e^{\epsilon_*s_n^\prime}\bm{\Phi}_{\lambda_*+\epsilon_*}(x)\right)
\end{align*}
for any $(t,x)\in\{\mathbb R\times\mathbb R^N: \ t\ge -t_n^\prime,\ s_n^\prime\le s_\eta\}$.
By passing the limits,
\begin{align*}
&\quad U(x,c_*t-x\cdot e-\eta)-D_\eta e^{\lambda_*(c_*t-x\cdot e)}
\left(\bm{\Phi}_{\lambda_*}(x)-e^{\epsilon_*(c_*t-x\cdot e)}
\bm{\Phi}_{\lambda_*+\epsilon_*}(x)\right)\\
&\le\bm u_\infty(t,x)\\
&\le U(x,c_*t-x\cdot e+\eta)+D_\eta e^{\lambda_*(c_*t-x\cdot e)}
\left(\bm{\Phi}_{\lambda_*}(x)-e^{\epsilon_*(c_*t-x\cdot e)}
\bm{\Phi}_{\lambda_*+\epsilon_*}(x)\right)
\end{align*}
for any $(t,x)\in\{\mathbb R\times\mathbb R^N:\ c_*t-x\cdot e\le s_\eta\}$.
It then follows from Lemma \ref{louville} that
$$\bm u_\infty(t,x)=U(x,c_*t-x\cdot e),\quad
\forall\;(t,x)\in\mathbb R\times\mathbb R^N.$$
In particular,
$u_{i,\infty}(t_\infty,x_\infty)=U_i(x_\infty,c_*t_\infty-x_\infty\cdot e)$.
Noting that
\begin{align*}
u_{i,\infty}(t_\infty,x_\infty)&=\lim_{n\to\infty}u_i(t_n,x_n;\bm u_0),\\
U_i(x_\infty,c_*t_\infty-x_\infty\cdot e)&=\lim_{n\to\infty}U_i(x_n,c_*t_n-x_n\cdot e),
\end{align*}
which contradicts \eqref{alpha}. Now if $s_n\to-\infty$ as $n\to\infty$,
up to an extraction of a subsequence, Lemma \ref{sss-*} yields that
$\lim_{n\to\infty}\bm u(t_n,x_n;\bm u_0)=\lim_{n\to\infty}U(x_n,c_*t_n-x_n\cdot e)=\bm 0$,
and if $s_n\to+\infty$ as $n\to\infty$, then
$\lim_{n\to\infty}\bm u(t_n,x_n;\bm u_0)=\lim_{n\to\infty}U(x_n,c_*t_n-x_n\cdot e)=\bm 1$,
both contradict \eqref{alpha}. Hence the statement is true.
The proof is complete.
\end{proof}


\section*{Acknowledgments}

The authors would like to express great gratitude to Prof. Wenxian Shen (Auburn University) for her helpful comments. Research of L.-L. Du was partially supported by  NSF of China (No. 12201066) and the Fundamental Research Funds for  the Central Universities, CHD (No. 300102124202).  Research of W.-T. Li was partially supported by NSF of China (No. 12271226), NSF of Gansu Province of China (21JR7RA537) and the Fundamental Research Funds for the Central Universities (lzujbky-2022-sp07).


\end{document}